\documentclass[leqno,11pt]{amsart}
\usepackage[top=1.25in, bottom=1.25in, left=1in, right=1in]{geometry}
\usepackage{amssymb,amsmath,latexsym,amsfonts,amsbsy, amsthm,dsfont}
\usepackage{hyperref}
\usepackage{color}
\usepackage{mathrsfs}
\usepackage{graphicx}
\usepackage{comment}
\usepackage{enumitem}
\usepackage{xcolor,soul}
\definecolor{pku}{RGB}{139,0,18}

\let\pa=\partial
\let\f=\frac

\let\pa=\partial

\let\al=\alpha
\let\b=\beta
\let\e=\varepsilon
\let\d=\delta

\let\g=\gamma

\let\ka=\kappa
\let\lam=\lambda
\let\Lam=\Lambda
\let\om=\omega

\let\r=\rho

\let\th=\theta

\def\Re{\mathrm{Re}}
\def\Im{\mathrm{Im}}

\newcommand{\beq}{\begin{equation}}
\newcommand{\eeq}{\end{equation}}
\newcommand{\beqo}{\begin{equation*}}
\newcommand{\eeqo}{\end{equation*}}
\newcommand{\ben}{\begin{eqnarray}}
\newcommand{\een}{\end{eqnarray}}
\newcommand{\beno}{\begin{eqnarray*}}
\newcommand{\eeno}{\end{eqnarray*}}

\newtheorem{thm}{Theorem}[section]
\newtheorem{lem}{Lemma}[section]
\newtheorem{cor}{Corollary}[section]
\newtheorem{prop}{Proposition}[section]

\theoremstyle{definition}

\theoremstyle{remark}
\newtheorem{step}{Step}

\newtheorem{rmk}{Remark}[section]

\newcommand{\pv}{\mathrm{p.v.}}

\newcommand{\CE}{\mathcal{E}}
\newcommand{\CL}{\mathcal{L}}
\newcommand{\CD}{\mathcal{D}}
\newcommand{\CI}{\mathcal{I}}

\newcommand{\CF}{\mathcal{F}}

\newcommand{\CH}{\mathcal{H}}

\newcommand{\CZ}{\mathcal{Z}}

\newcommand{\R}{\mathbb{R}}
\newcommand{\BC}{\mathbb{C}}
\newcommand{\BR}{\mathbb{R}}
\newcommand{\BZ}{\mathbb{Z}}
\newcommand{\BT}{\mathbb{T}}

\newcommand{\BN}{\mathbb{N}}
\newcommand{\T}{\mathbb{T}}

\numberwithin{equation}{section}

\allowdisplaybreaks

\begin{document}

\title{Geometric Properties of the 2-D Peskin Problem}
\author{Jiajun Tong}
\address{Beijing International Center for Mathematical Research, Peking University, China}
\email{tongj@bicmr.pku.edu.cn}
\author{Dongyi Wei}
\address{School of Mathematical Sciences, Peking University, China}
\email{jnwdyi@pku.edu.cn}

\date{\today}
\maketitle

\begin{abstract}
The 2-D Peskin problem describes a 1-D closed elastic string immersed and moving in a 2-D Stokes flow that is induced by its own elastic force.
The geometric shape of the string and its internal stretching configuration evolve in a coupled way, and they combined govern the dynamics of the system.
In this paper, we show that certain geometric quantities of the moving string satisfy extremum principles and decay estimates.
As a result, we can prove that the 2-D Peskin problem admits a unique global solution when the initial data satisfies a medium-size geometric condition on
the string shape, while no assumption on the size of stretching is needed.
\end{abstract}

\section{Introduction}
\label{sec: 2D Peskin}

\subsection{The 2-D Peskin problem}
In this paper, we study the 2-D Peskin problem, also known as the Stokes immersed boundary problem in 2-D.
It describes a 1-D closed elastic string immersed and moving in a 2-D Stokes flow induced by itself.
With $s\in \BT = \BR/(2\pi \BZ)$ being the Lagrangian coordinate, we use $X = X(s,t)\in \BR^2$ to represent position of the string point with the Lagrangian label $s$ at time $t$.
Then the 2-D Peskin problem reads that
\begin{align}
&\;-\Delta u+ \nabla p = \int_{\mathbb{T}}F_X(s,t)\delta(x-X(s,t))\,ds,\label{eqn: Stokes equation}\\
&\;\mathrm{div}\, u = 0,\quad |u|,|p|\rightarrow 0\mbox{ as }|x|\rightarrow \infty,\label{eqn: incompressible condition}\\
&\;\frac{\partial X}{\partial t}(s,t) = u(X(s,t),t),\quad
X(s,0) = X_0(s).\label{eqn: motion of the membrane and initial configuration}
\end{align}
Here \eqref{eqn: Stokes equation} and \eqref{eqn: incompressible condition} are the stationary Stokes equation.
We use $u = u(x,t)$ and $p = p(x,t)$ to denote the flow velocity field and the pressure in $\BR^2$, respectively.
The right-hand side of \eqref{eqn: Stokes equation} gives the force applied to the fluid that is singularly supported along the curve
$X(\BT,t): = \{X(s,t):\,s\in \BT\}$. 
$F_X$ is the elastic force density in the Lagrangian coordinate.
In general \cite{peskin2002immersed},
\begin{equation}
F_X(s,t) = \partial_s\left(\mathcal{T}(|X'(s,t)|, s,t)\frac{X'(s,t)}{|X'(s,t)|}\right).
\label{eqn: Lagrangian representation of the elastic force}
\end{equation}
For brevity, here and in what follows, we shall use $X'(s,t)$ and $X''(s,t)$ to denote $\partial_s X(s,t)$ and $\partial_{ss} X(s,t)$, respectively.
In \eqref{eqn: Lagrangian representation of the elastic force},
$\mathcal{T}$ represents the tension along the string, which is determined by the elastic property of the string material.
In the simple case of Hookean elasticity, 
$\mathcal{T}(|X'(s,t)|,s,t) = k_0 |X'(s,t)|$ with $k_0>0$ being the Hooke's constant, so $F_X(s,t) = k_0X''(s,t)$.
\eqref{eqn: motion of the membrane and initial configuration} specifies that each string point should move with the flow.

In early 1970s, aiming at studying blood flows around heart valves, Peskin \cite{peskin2002immersed,peskin1972flow} introduced the general immersed boundary problem as well as the well-known numerical immersed boundary method.
Since then, the latter has been extensively investigated in numerical analysis and successfully applied to a wide range of fluid-structure
interaction problems in physics, biology, medical sciences, and engineering.
However, rigorous analytical studies of the immersed boundary problem have been absent until recent years.
In the independent works \cite{lin2019solvability} and \cite{mori2019well}, the authors considered the Stokes case of the 2-D immersed boundary problem, and reformulated the model \eqref{eqn: Stokes equation}-\eqref{eqn: motion of the membrane and initial configuration} with $F_{X} = X''(s,t)$ into a contour dynamic equation in the Lagrangian coordinate
\beq
\pa_t X(s,t)
=
\int_\BT G(X(s,t)-X(s',t)) X''(s',t)\,ds',\quad X(s,0) = X_0(s),
\label{eqn: contour dynamic equation}
\eeq
where $G(\cdot)$ denotes the 2-D Stokeslet
\beq
G_{ij}(x) = \f{1}{4 \pi}\left(-\ln |x|\d_{ij} + \f{x_i x_j}{|x|^2}\right),\quad i,j = 1,2.
\label{eqn: Stokeslet in 2-D}
\eeq
After integration by parts, \eqref{eqn: contour dynamic equation} becomes
\beq
\begin{split}
\pa_t X(s)
= &\;
-\pv \int_\BT \pa_{s'}[G(X(s,t)-X(s',t))] X'(s',t)\,ds'\\
= &\;\frac{1}{4\pi}\mathrm{p.v.}\int_{\mathbb{T}} \left[- \frac{|X'(s')|^2}{|X(s')-X(s)|^2}+\frac{2((X(s')-X(s))\cdot X'(s'))^2}{|X(s')-X(s)|^4}\right] (X(s')-X(s))\,ds'.
\end{split}
\label{eqn: contour dynamic equation integration by parts}
\eeq
Based on this, well-posedness of the Cauchy problem of the string evolution was studied. 
It was shown that, if the initial data $X(s,0) = X_0(s)$ lies in some subcritical regularity classes, the Sobolev space $H^{5/2}(\BT)$ in \cite{lin2019solvability} and the little H\"{o}lder space $h^{1,\al}(\BT)$ in \cite{mori2019well} respectively, and satisfies the so called \emph{well-stretched condition}, then there exists a unique local-in-time solution.
A general string configuration $Y(s)$ is said to satisfy the well-stretched condition with constant $\lam>0$ if for any $s_1,s_2\in \BT$,
\beq
\big|Y(s_1)-Y(s_2)\big|\geq \lam |s_1-s_2|_{\T}.
\label{eqn: well-stretched condition}
\eeq
Here $|s_1-s_2|_{\T}:=\min\{|s_1-s_2-2k\pi|:k\in \mathbb{Z}\}$.
For future use, we shall call the largest possible such constant $\lam$ as the \emph{best well-stretched constant} of $Y$.
If the initial data $X_0$ turns out to be sufficiently close to an equilibrium, which is an evenly-parameterized circular configuration, both works showed that the solution should exist globally up to $t = \infty$, and $X(\cdot,t)$ will converge exponentially to an equilibrium configuration as $t\to +\infty$.
In addition to these well-posedness results, \cite{mori2019well} showed that the solutions enjoy instant smoothing at positive times.
It also proved a blow-up criterion, stating that if a solution fails at a finite time, then either $\|X(\cdot,t)\|_{\dot{C}^{1,\alpha}}$ goes to infinity, or the best well-stretched constant shrinks to zero (or both occur at the same time).

After these breakthroughs, there have been many studies devoted to establishing well-posedness under scaling-critical regularities, which refer to $W^{1,\infty}(\BT)$ and other regularity classes with equivalent scaling.
\cite{garcia2020peskin} considered the scenario where the fluids can have different viscosities in the interior and the exterior of the string, and showed global well-posedness for medium-size initial data in the critical Wiener space $\CF^{1,1}(\BT)$.
\cite{gancedo2020global} introduced a toy scalar model that captures the string motion in the normal direction, and proved global well-posedness for small initial data in the critical Lipschitz class $W^{1,\infty}(\BT)$.
\cite{chen2021peskin} established well-posedness of the original 2-D Peskin problem in the critical space $\dot{B}_{\infty,\infty}^{1}(\BT)$, which contains non-Lipschitz functions.

To handle general elastic laws beyond the Hookean elasticity, \cite{rodenberg20182d} and the recent work \cite{cameron2021critical} extended the analysis to fully nonlinear Peskin problems.
Very recently, \cite{GarciaJuarez2023WellPosednessOT} introduced the 3-D Peskin problem and proved its well-posedness with general nonlinear elastic laws.
Readers are also directed to \cite{li2021stability} for a study on the case where the elastic string has both stretching and bending energy.

To better understand the nonlinear structure of the Peskin problem as well as its possible blow-up mechanisms, the first author of this paper recently derived a new scalar model called the tangential Peskin problem \cite{Tong2022GlobalST}.
It considers the 2-D Peskin problem in the case of an infinitely long and straight 1-D elastic string deforming only tangentially.
Global solutions were constructed for initial datum in the energy class $H^1$, and their properties were characterized.
Let us highlight that, despite its very different geometric setup, this problem turns out to have a surprising connection with the original 2-D Peskin problem with a closed circular string;
see Theorem \ref{thm: short version of the thm for circular initial data} or Theorem \ref{thm: circular string} below.

Let us also mention that the first author of this paper introduced the regularized Peskin problem in \cite{tong2021regularized}, and proved its well-posedness and convergence to the original Peskin problem as the regularization diminishes.

The Peskin problem has noticeable mathematical similarities with several other evolution free boundary problems, especially the Muskat problem (see \cite{muskat1934two,Ambrose2004WellposednessOT,cordoba2011,constantin2012global,cordoba2013porous,constantin2016muskat,deng2017two,cameron2018global,
Crdoba2018GlobalWF,gancedo2019muskat,Alazard2020EndpointST} and the references therein).
Nevertheless, a distinct feature of the Peskin problem is that the elastic string is not just a time-varying curve $X(\BT,t): = \{X(s,t)\in \BR^2:\,s\in \BT\}$, but it has internal elastic structure.
More precisely, the planar curve $X(\BT,t)$ only captures \emph{geometric information} of the string, with the function $(s,t)\mapsto X'(s,t)/|X'(s,t)|$ representing its unit tangent vector,
whereas the function $(s,t)\mapsto |X'(s,t)|$ encodes the \emph{stretching configuration}.
Strings with identical shape may have quite different internal stretching configurations, leading to different dynamics.
Thus, we need to track both the geometric and stretching information of the string at the same time, while they evolve in a strongly coupled way.

\subsection{Main results}
In this paper, we point out that
certain \emph{geometric} quantities of the curve $X(\BT,t)$ enjoy extremum principles and decay estimates.
This holds regardless of the stretching configuration of the string.
As a result, we can prove that, if the initial data satisfies some medium-size \emph{geometric} condition in addition to those classic assumptions, the 2-D Peskin problem admits a unique global solution.
No additional restriction is needed on size of the stretching $|X'(s,t)|$.
We can also show that the global solution converges exponentially to a final equilibrium state as $t\to +\infty$.

In the rest of the paper, we shall treat $X = X(s,t)$ as a complex-valued function instead of a vector-valued function, i.e., $X = X_1+iX_2$ with $X_1,X_2\in \BR$.
Then \eqref{eqn: contour dynamic equation integration by parts} can be recast as
\beq
\begin{split}
\pa_t X(s)
= &\; \f{1}{4\pi}\mathrm{p.v.}\int_\BT \Re \left[\f{X'(s')^2}{(X(s')-X(s))^2} \right] \big(X(s')-X(s)\big) \,ds'\\
= &\; \f{1}{4\pi}\pv\int_\BT \f{X'(s')^2}{X(s')-X(s)}\,ds'
 - \f{i}{4\pi}\int_\BT \Im \left[\f{X'(s')^2}{(X(s')-X(s))^2} \right] \big(X(s')-X(s)\big) \,ds'.
\end{split}
\label{eqn: contour dynamic equation complex form}
\eeq
Here and in what follows, the time dependence will get omitted whenever it is convenient.
\eqref{eqn: contour dynamic equation complex form} is equipped with the initial condition
\beq
X(s,0) = X_0(s).
\label{eqn: initial condition}
\eeq

Before we state the main results, let us introduce some notations.
Given $X(\cdot,t)$, we define $R_X>0$ by
\beq
\pi R_X^2 = \f12 \int_{\T}\Im\big[\overline{X(s')}X'(s')\big]\,ds'.
\label{eqn: effective radius}
\eeq
The right-hand side gives the area of the planar region enclosed by the curve $X(\BT,t)$, so $R_X$ will be referred as the \emph{effective radius} of $X$.
When $X(s,t)$ is a sufficiently smooth solution to \eqref{eqn: contour dynamic equation complex form}, $R_X$ is time-invariant since the flow field in $\BR^2$ is divergence-free and the area is conserved (see \cite{lin2019solvability,mori2019well}).

Assume that $X(\cdot,t)\in C^1(\BT)$ satisfies the well-stretched condition defined in
\eqref{eqn: well-stretched condition}.
This implies that $X(s_1,t)\neq X(s_2,t)$ for distinct $s_1,s_2\in \BT$, and that $\inf_{s\in \BT}|X'(s,t)|>0$.
Then we let
\beq
\Phi(s_1,s_2,t):=
\begin{cases}
\arg\left[\f{X'(s_1,t) X'(s_2,t)}{(X(s_1,t)-X(s_2,t))^2}\right], & \mbox{if } s_1,s_2\in \BT,\;s_1\neq s_2, \\
0, & \mbox{if }s_1 =s_2\in \BT.
\end{cases}
\label{eqn: def of Phi}
\eeq
Under the above assumptions, the angle $\Phi(\cdot,\cdot,t)$ is pointwise well-defined on $\BT\times \BT$ with its value in $\BT$.
It is fully determined by the geometric shape of $X(\BT,t)$, but hardly depends on the internal stretching configuration $|X'|$.
If $X(\BT)$ is in a perfectly circular shape, then $\Phi(s_1,s_2)\equiv 0$ for all $s_1,s_2\in \BT$; in fact, the converse is also true (see Lemma \ref{lem: zero Phi_+ or Phi_-}).
In general, $\Phi(s_1,s_2)$ measures the asymmetry of the curve $X(\BT)$ when observed from the points $X(s_1)$ and $X(s_2)$.
It is invariant under translation, rotation, and dilation of the system.
In our analysis below, $\Phi$ will play an extremely important role.

Given $\al\in (0,1)$ and a function $Y:\BT\to \BC$, we recall definitions of the $C^{1,\al}(\BT)$-semi-norm and the $C^{1,\al}(\BT)$-norm:
\begin{align*}
\|Y\|_{\dot{C}^{1,\al}(\BT)} := &\; \sup_{s_1,s_2\in \BT\atop s_1\neq s_2} \f{|Y'(s_1)-Y'(s_2)|}{|s_1-s_2|_\BT^\al},\\
\|Y\|_{C^{1,\al}(\BT)} := &\; \|Y\|_{L^\infty(\BT)}+\|Y'\|_{L^\infty(\BT)}
+\|Y\|_{\dot{C}^{1,\al}(\BT)}.
\end{align*}
We use $h^{1,\al}(\BT)$ to denote the little H\"{o}lder space, which is the completion of $C^\infty(\BT)$ under the $C^{1,\al}(\BT)$-norm.

Now we can state our main result.

\begin{thm}
\label{thm: main thm}
Suppose that $X_0(s)\in h^{1,\al}(\BT)$ for some $\al\in(0,1)$, and it satisfies the well-stretched condition.
Let $R_X$ be defined in \eqref{eqn: effective radius} in terms of $X_0$.
Define $\Phi_0(s,s')$ in terms of $X_0$ as in \eqref{eqn: def of Phi}.
If $|\Phi_0(s,s')|<\pi/4$ for all $s,s'\in \BT$, then \eqref{eqn: contour dynamic equation complex form} and \eqref{eqn: initial condition} admit a unique global solution $X = X(s,t)$ in the class $C_{loc}([0,+\infty);C^{1,\al}(\BT))\cap C_{loc}^1([0,+\infty);C^{\al}(\BT))$, in the sense that
\[
\mbox{$X(s,t)$ satisfies \eqref{eqn: contour dynamic equation complex form} in $\BT\times (0,+\infty)$, and
$X(\cdot,t)\to X_0(\cdot)$ in $C^{1,\al}(\BT)$ as $t\to 0$.}
\]
Moreover, the solution has the following properties:
\begin{enumerate}[label = (\roman*)]
\item For any $\d>0$ and any $k\in\BN$, $X\in C_{loc}^1([\d,+\infty); C^k(\BT))$.
\item There exists a universal non-negative and strictly increasing function $\lam_\circ(t)$ defined on $[0,+\infty)$, with $\lam_\circ(0) = 0$, such that $X(\cdot,t)$ satisfies the well-stretched condition with constant $\lam_\circ(t)R_X$.

\item Let $\Phi_*(t):=\sup_{s,s'\in \BT}|\Phi(s,s',t)|$.
Then $\Phi_*(t)$ is continuous and non-increasing on $[0,+\infty)$, and
\[
0\leq \Phi_*(t)\leq \Phi_*(0) \min\big\{e^{- \mu t},\,Ce^{-t/\pi^2}\big\},
\]
where $\mu = (4-2\sqrt{2})/\pi^3$ and where $C>0$ is a universal constant.

\item
Let $\ka(s,t)$ be the curvature of the curve $X(\BT,t)$ at the point $X(s,t)$, given by
\[
\kappa(s,t)=\f{\Im[X''(s,t)/X'(s,t)]}{|X'(s,t)|}.
\]
Then there exist universal constants $C,c>0$, such that for any $t>0$,
\beqo
\|\kappa(s,t)R_X-1\|_{L^\infty(\BT)}
\leq C \exp\left[C \left(\int_0^t \cos 2\Phi_*(\tau)\,d\tau\right)^{-1} -ct\right].
\eeqo

\item
There exist $x_\infty\in \BC$ and $\xi_\infty\in \BT$, such that, $X(\cdot,t)$ converges exponentially to an equilibrium
\[
X_\infty(s) := x_\infty + R_X e^{i(s+\xi_\infty)}
\]
in $H^2(\BT)$ as $t\to +\infty$.
More precisely, for any $\d>0$, there exists a constant $C>0$ depending on $\d$ and $X_0$, and a universal constant $c>0$ not depending on $\d$ or $X_0$, such that
\[
\|X(\cdot,t)-X_\infty(\cdot)\|_{H^2(\BT)} \leq C e^{-ct}
\]
for all $t\geq \d$.
\end{enumerate}
\end{thm}

In addition to the classic regularity assumption and the well-stretched condition on the initial data, Theorem \ref{thm: main thm} only adds a \emph{geometric} condition $|\Phi_0|<\pi/4$ to guarantee the global well-posedness, and imposes no restriction on the size of stretching of the string, which is very different from all the existing results on global well-posedness \cite{lin2019solvability,mori2019well,garcia2020peskin,chen2021peskin}.
The threshold $\pi/4$ reflects the nature of the Stokeslet.
We will see later that, if $X_0(s)$ is $O(1)$-close to an equilibrium in the $C^1(\BT)$-seminorm, the corresponding $\Phi_0$ will satisfy $|\Phi_0|<\pi/4$.
More generally, if we additionally let $\psi:\BT\to\BT$ be an arbitrary bijective diffeomorphism that is suitably smooth, and define $\tilde{\Phi}_0$ in terms of the reparameterized configuration  $\tilde{X}_0 = X_0\circ\psi$ with $X_0$ as above, then $|\tilde{\Phi}_0|<\pi/4$ since $\tilde{\Phi}_0(s,s') =\Phi_0(\psi(s),\psi(s'))$.
Thus, Theorem \ref{thm: main thm} applies to such initial data $X_0\circ\psi$, and there is no restriction on the size of $ \psi$.
See the precise statements in Remark \ref{rmk: initial data C1 close to equilibrium}.
In view of this (and also Proposition \ref{prop: geometric property of the curve}), the condition $|\Phi_0|<\pi/4$ can be understood as a medium-size $C^1$-condition on the geometry of the initial string.

A key ingredient in proving Theorem \ref{thm: main thm} is to show that, under the condition $|\Phi_0(s,s')|<\pi/4$, $\Phi(s,s',t)$ and $\ka(s,t)$ satisfy some extremum principles and decay estimates (see Proposition \ref{prop: bound for the max |Phi|} and Proposition \ref{prop: max principle and decay estimate for curvature}, respectively).
Such extremum principles can even be generalized to the case of general elasticity; see Section \ref{sec: general elasticity} and Remark \ref{rmk: general elasticity curvature}.

We use the sub-critical assumption $X_0\in h^{1,\al}(\BT)$ here for convenience.
One may be able to replace it by a weaker regularity assumption on the stretching function $|X'(s,t)|$ and a separate geometric assumption on the curve $X(\BT)$.
However in that case, the analysis would become much more involved.
In order to highlight our findings and not to distract the readers, we decide to work with the current assumptions, and leave further generalizations to future works.

With that being said, in the special case when $X_0(\BT)$ is a circle but $X_0(s)$ is not necessarily an equilibrium, it is relatively straightforward to weaken the assumption to $X_0\in H^1(\BT)$. 
Note that $H^1(\BT)$ is the energy class for the 2-D Peskin problem in the Hookean case, which is super-critical according to scaling of the equation.
For $H^1$-circular initial data satisfying some other natural assumptions, we can construct a global solution $X(s,t)$, with $X(\BT,t)$ being a circle of the same radius for all time.
Motion of the center of the circle can be determined, and surprisingly, the tangential deformation of the string is described \emph{exactly} by the tangential Peskin problem \cite{Tong2022GlobalST}.

\begin{thm}[Short version of Theorem \ref{thm: circular string}]
\label{thm: short version of the thm for circular initial data}

Assume $X_0\in H^1(\BT)$, such that $X_0(\BT)$ is a circle of radius $R_X$ centered at $x_0\in \BC$.
Suppose for a strictly increasing continuous function $\th_0:\BR\to \BR$ which satisfies $\th_0(s+2\pi) = \th_0(s) + 2\pi$ in $\BR$, it holds that $X_0(s) = x_0+ R_X e^{i\th_0(s)}$ in the notation of complex numbers.
Under suitable assumptions on $\th_0(s)$, \eqref{eqn: contour dynamic equation complex form} and \eqref{eqn: initial condition} admit a global solution $X = X(s,t)$ in $\BT \times [0,+\infty)$, which can be constructed as follows.
Let $\th = \th(s,t)$ be a solution to the tangential Peskin problem \cite[Corollary 2.1]{Tong2022GlobalST}
\[
\pa_t \th(s,t) = -\f{1}{4\pi} \pv \int_{\BR}\f{(\pa_{s'}\th(s',t))^2}{\th(s,t)-\th(s',t)}\,ds',\quad \th(s,0) = \th_0(s)
\]
in $\BR\times [0,+\infty)$.
Let
\beqo
v(t):= -\f{R_X}{8\pi}
\int_\BT e^{i\th(s',t)} \big(\pa_{s'}\th(s',t)\big)^2 \,ds'.
\eeqo
Then
\beqo
X(s,t) = x(t) + R_X e^{i\th(s,t)} \mbox{ with } x(t): = x_0+\int_0^t v(\tau) \,d\tau
\eeqo
gives the desired solution to \eqref{eqn: contour dynamic equation complex form} and \eqref{eqn: initial condition}.
Moreover, $X(s,t)$ is smooth in $\BT\times (0,+\infty)$.
For any $t> 0$, $X(\BT,t)$ is a circle, and $X(\cdot,t)$ satisfies the well-stretched condition.
For arbitrary $k\in \BN$, $X(s,t)$ converges exponentially to a final equilibrium configuration $X_\infty(s)$ in $H^k(\BT)$-norms as $t\to +\infty$.
\end{thm}

\subsection{Organization of the paper}
A major part of this work is devoted to proving various a priori estimates for sufficiently smooth solutions.
In Section \ref{sec: preliminary whole section}, we set up the problem, introduce necessary notations, and derive some basic equations.
Section \ref{sec: angle Phi} and Section \ref{sec: curvature} constitute the most important part of the paper, where we study geometric properties of the curve $X(\BT)$.
In Section \ref{sec: angle Phi}, we show that under the assumption $\sup_{s,s'\in \BT}|\Phi_0(s,s')|<\pi/4$, $|\Phi|$ satisfies the maximum principle and a decay estimate.
With the bound on $|\Phi|$, we can quantify geometric features of the curve $X(\BT)$.
We also prove Theorem \ref{thm: short version of the thm for circular initial data} (i.e., Theorem \ref{thm: circular string}) on $H^1$-initial data with a circular shape.
Section \ref{sec: curvature} is focused on the curvature $\ka(s,t)$, and proves its extremum principle and decay estimate.
We extend the extremum principles on $\Phi$ and $\ka$ to the case of general elasticity in Section \ref{sec: general elasticity} and Remark \ref{rmk: general elasticity curvature}, respectively.
In Section \ref{sec: estimates for |X'|}, we turn to deriving estimates for the stretching configuration $|X'|$.
Finally, we prove Theorem \ref{thm: main thm} in Section \ref{sec: proof of main results}.

\subsection*{Acknowledgement}
Both authors are supported by the National Key R\&D Program of China under the grant 2021YFA1001500.


\section{Preliminaries}
\label{sec: preliminary whole section}

\subsection{Setup and the notations}
\label{sec: preliminary}

Throughout this paper, except for the proof of the main results in Section \ref{sec: proof of main results}, we will always assume that for some $T>0$,
$X(s,t)$ solves \eqref{eqn: contour dynamic equation complex form} in $\BT\times [0,T]$, and satisfies
\begin{enumerate}[label = (A\arabic*)]
\item For each $t\in [0,T]$, $s\mapsto X(s,t)$ is injective, and the image $X(\BT,t)$ is a closed $C^2$-curve in $\mathbb{C}$;
    \label{assumption: geometry}
\item $X'(s,t), X''(s,t), X'''(s,t)\in C^1(\BT\times [0,T])$;
\item $|X'(s,t)|>0$ in $\BT\times [0,T]$.
\label{assumption: non-degeneracy}
\end{enumerate}
Besides, we assume the elastic string is parameterized by $X= X(s,t)$ in the counter-closewise direction (i.e., its winding number with respect to an arbitrary point in the interior of $X(\BT)$ is $1$).
Given these assumptions, it is not difficult to show that, for each $t\in [0,T]$,
\beq
\inf_{s\neq s'}\f{|X(s,t)-X(s',t)|}{|s-s'|} >0.
\label{eqn: assumption well stretched}
\eeq

We first prove a lemma that will be used repeatedly in the calculation below.

\begin{lem}
\label{lem: some singular integrals along the string curve}
Fix $t\in [0,T]$, which will be omitted below.
For any $s\in \BT$,
\[
\pv \int_\BT \f{X'(s')}{X(s')-X(s)}\,ds'= \pi i,
\]
and
\[
\int_\BT \Im\left[\f{X'(s')X'(s)}{(X(s')-X(s))^2}\right] ds' = 0.
\]
\begin{proof}
We calculate the first integral that
\[
\pv \int_\BT \f{X'(s')}{X(s')-X(s)}\,ds'= \lim_{\e\to 0^+} \ln (X(s+2\pi-\e)-X(s)) - \ln (X(s+\e)-X(s)).
\]
Its imaginary part is $\pi i$, as we assumed that $X\in C^1(\BT)$ and the curve $X(\BT)$ is parameterized in the counter-clockwise direction.
It suffices to show the real part is $0$.
In fact,
\[
\lim_{\e\to 0^+} \ln \f{|X(s-\e)-X(s)|}{|X(s+\e)-X(s)|} =\lim_{\e\to 0^+} \ln \f{|X(s-\e)-X(s)|/\e}{|X(s+\e)-X(s)|/\e}=\ln \f{|X'(s)|}{|X'(s)|}=0.
\]
This proves the first claim.

For the second integral, we first observe that
\beq
\begin{split}
\Im\left[\f{X'(s')X'(s)}{(X(s')-X(s))^2}\right]
= &\; \Im\left[\f{X'(s')-\f{X(s')-X(s)}{s'-s}}{X(s')-X(s)}
\cdot \f{X'(s)-\f{X(s')-X(s)}{s'-s}}{X(s')-X(s)}\right] \\
&\; -
2\Im\left[ \f{\f{X(s')-X(s)}{s'-s} - X'(s)-\f12(s'-s) X''(s)}{(s'-s)(X(s')-X(s))}\right] \\
&\; +
\Im\left[\f{X'(s') - X'(s)- (s'-s)X''(s)} {(s'-s)(X(s')-X(s))}\right].
\end{split}
\label{eqn: boundedness of Im J}
\eeq
One can show that this is bounded
by using \eqref{eqn: assumption well stretched}, the assumption $X\in C^3(\BT)$, and the Taylor expansion.
Then we calculate the second integral that
\[
\begin{split}
&\; \int_\BT \Im\left[\f{X'(s')X'(s)}{(X(s')-X(s))^2}\right] ds'\\
= &\; \lim_{\e\to 0^+} \Im\left[X'(s)\int_{\BT \setminus [-\e,\e]} \pa_{s'}\left(-\f{1}{X(s+s')-X(s)}\right) ds'\right]\\
= &\; \lim_{\e\to 0^+} \Im\left[\f{X'(s)}{X(s+\e)-X(s)}-\f{X'(s)}{X(s-\e)-X(s)}\right]\\
= &\; -\lim_{\e\to 0^+}\f{ \Im [\overline{X'(s)}(X(s+\e)-X(s))]}{|X(s+\e)-X(s)|^2}
+\lim_{\e\to 0^+}\f{\Im[\overline{X'(s)} (X(s-\e)-X(s))]}{|X(s-\e)-X(s)|^2}.
\end{split}
\]
We find
\[
\begin{split}
\lim_{\e\to 0^+}\f{ \Im[\overline{X'(s)} (X(s+\e)-X(s))]}{|X(s+\e)-X(s)|^2}
= &\;
\lim_{\e\to 0^+}\f{ \e^{-2}\Im [\overline{X'(s)}(X(s+\e)-X(s)-\e X'(s))]} {\e^{-2}|X(s+\e)-X(s)|^2}\\
= &\; \f{\Im (\overline{X'(s)} X''(s))}{2|X'(s)|^2}, 
\end{split}
\]
and similarly,
\[
\lim_{\e\to 0^+}\f{ \Im [\overline{X'(s)}(X(s-\e)-X(s))]}{|X(s-\e)-X(s)|^2}
= \f{\Im (\overline{X'(s)} X''(s))}{2|X'(s)|^2}. 
\]
Then the second claim follows.
\end{proof}
\end{lem}

Using Lemma \ref{lem: some singular integrals along the string curve}, we differentiate \eqref{eqn: contour dynamic equation complex form} to obtain
\beq
\begin{split}
&\;\pa_t X'(s)\\
= &\; \f{1}{4\pi}\pa_s
\left( \int_\BT \f{X'(s')(X'(s')- X'(s))}{X(s')-X(s)}\,ds'+\pi i X'(s)\right)\\
&\; -\f{i}{4\pi} \int_\BT \pa_{s}
\left[\Im \left(\f{X'(s')^2}{(X(s')-X(s))^2} \right) \big(X(s')-X(s)\big)\right]ds'\\
= &\; X'(s) \cdot \f{1}{4\pi}
\pv \int_\BT \f{X'(s')(X'(s')- X'(s))}{(X(s')-X(s))^2} \,ds' \\
&\; -\f{i}{4\pi} \int_\BT \left\{\Im \left[\f{2X'(s')^2 X'(s) }{(X(s')-X(s))^3} \right] \big(X(s')-X(s)\big) -
\Im \left[\f{X'(s')^2}{(X(s')-X(s))^2} \right]  X'(s) \right\}ds'.
\end{split}
\label{eqn: eqn for X' complex form}
\eeq
No principal value integral is needed in the last line because the integrand is bounded given the assumptions on $X$.
This can be justified as in \eqref{eqn: boundedness of Im J}.

We introduce a few more notations that will be used throughout the paper.
For $s\in \BT$, let
\[
\al(s) := \arg X'(s)\in \mathbb{T}.
\]
Here $\al(s)$ and the angles defined below should always be understood in the modulo $2\pi$.
For distinct $s_1,s_2\in \BT$, let
\[
I(s_1,s_2):=\frac{X'(s_1)}{X(s_1)-X(s_2)},
\]
and
\[
J(s_1,s_2):= \frac{X'(s_1)X'(s_2)}{(X(s_1)-X(s_2))^2} = \frac{\partial I}{\partial s_2}(s_1,s_2).
\]
As a result, for distinct $s_1,s_2\in \BT$ (cf.\;\eqref{eqn: def of Phi}),
\[
\Phi(s_1,s_2)=\arg J(s_1,s_2).
\]
Since we additionally defined $\Phi(s,s) = 0$ for all $s\in \BT$, thanks to the regularity of $X(s,t)$, $\Phi(s_1,s_2,t)$ is (at least) a $C^1$-function in $\BT\times \BT\times [0,T]$.
Obviously,
\[
J(s_1,s_2) = J(s_2,s_1), \quad \Phi(s_1,s_2) = \Phi(s_2,s_1).
\]
Besides, define for distinct $s_0,s_1,s_2\in \BT$,
\[
\tilde{J}(s_0,s_1,s_2)
=\frac{X'(s_0)(X(s_1)-X(s_2))}{(X(s_0)-X(s_1))(X(s_0)-X(s_2))}=I(s_0,s_1)-I(s_0,s_2).
\]

The total length of the elastic string $X(s,t)$ is given by
\beq
\CL(t):= \int_\BT |X'(s,t)|\,ds.
\label{eqn: length}
\eeq
By the isoperimetric inequality, $\CL(t) \geq 2\pi R_X$, where $R_X$ was defined in \eqref{eqn: effective radius}.
The elastic energy of the string, which is also the total energy of the system \eqref{eqn: Stokes equation}-\eqref{eqn: motion of the membrane and initial configuration} in the case $F_X(s,t) = X''(s,t)$, is given by
\beq
\CE(t): = \f12 \int_\BT |X'(s,t)|^2\,ds.
\label{eqn: energy}
\eeq
Lastly, let
\beq
d_*: = \sup_{s,s'\in \BT}| X(s)-X(s')|.
\label{eqn: diameter of X}
\eeq
It is clear that $2d_*\leq \CL(t)$.

\subsection{The equation for $|X'|$}
Let us derive the equation for $|X'|$, which characterizes local stretching of the string.

Since
\[
|X'|\cdot \pa_t |X'| = \f12 \pa_t |X'|^2 = \Re\left[\pa_t X'(s) \overline{X'(s)}\right],
\]
we derive from \eqref{eqn: eqn for X' complex form} that
\beq
\begin{split}
&\; \pa_t |X'(s)|\\
= &\; |X'(s)| \cdot \f{1}{4\pi}
\pv \int_{\T} \Re\left[\f{X'(s')(X'(s')- X'(s))}{(X(s')-X(s))^2} \right] ds' \\
&\; +|X'(s)|\cdot \f{1}{4\pi} \int_{\T} \Im \left[\f{2X'(s')^2 X'(s) }{(X(s')-X(s))^3} \right] \Im\left[\frac{X(s')-X(s)}{X'(s)}\right] ds'.
\end{split}
\label{eqn: eqn for |X'| complex form}
\eeq
With the notations introduced above,
\beqo
\begin{split}
&\; \pa_t |X'(s)|\\
= &\; |X'(s)| \cdot \f{1}{4\pi}
\pv \int_{\T} \Re\left[\f{X'(s')^2X'(s)}{(X(s')-X(s))^3}\cdot \f{X(s')-X(s)}{X'(s)} - J(s',s) \right] ds' \\
&\; +|X'(s)|\cdot \f{1}{4\pi}\pv \int_{\T} \Im \left[\f{2X'(s')^2 X'(s) }{(X(s')-X(s))^3} \right] \Im\left[\frac{X(s')-X(s)}{X'(s)}\right] ds'\\
= &\; |X'(s)| \cdot \f{1}{4\pi}
\pv \int_{\T} \left\{\Re\left[\f{X'(s')^2X'(s)}{(X(s')-X(s))^3}\right]
\Re\left[ \f{X(s')-X(s)}{X'(s)} \right] - \Re \, J(s',s)\right\}  ds' \\
&\; +|X'(s)|\cdot \f{1}{4\pi}\pv \int_{\T} \Im \left[\f{X'(s')^2 X'(s) }{(X(s')-X(s))^3} \right] \Im\left[\frac{X(s')-X(s)}{X'(s)}\right] ds'\\
= &\; |X'(s)| \cdot \f{1}{4\pi}
\pv \int_{\T} \left\{\Re\left[\f{X'(s')^2X'(s)}{(X(s')-X(s))^3}\cdot
\f{\overline{X(s')-X(s)}}{\overline{X'(s)}} \right] - \Re \, J(s',s)\right\} ds' \\
=&\;|X'(s)| \cdot \f{1}{4\pi}
\pv \int_{\T} \Re\left[J(s',s)^2\cdot \frac{|X(s')-X(s)|^2}{|X'(s)|^2}-J(s',s)\right] ds'. 
\end{split}
\eeqo
Therefore,
\beq
\begin{split}
&\;\pa_t |X'(s)|\\
= &\;|X'(s)| \cdot \f{1}{4\pi}
\pv \int_{\T} \frac{|X'(s')|}{|X(s')-X(s)|^2}\left[|X'(s')|\cos 2\Phi(s',s)-|X'(s)|\cos \Phi(s',s)\right] ds'.
\end{split}
\label{eqn: eqn for |X'|}
\eeq

\subsection{The equation for $\al(s)$}
Although $\al$ is a $\BT$-valued function, its time derivative is well-defined as a real-valued function
\[
\pa_t \al = \Im\frac{\partial_tX'(s)}{X'(s)}.
\]
Using Lemma \ref{lem: some singular integrals along the string curve}, we derive from \eqref{eqn: eqn for X' complex form} that
\beq
\begin{split}
&\;\Im\frac{\partial_tX'(s)}{X'(s)}\\
= &\; \f{1}{4\pi}
\pv \int_{\T} \Im\left[\f{X'(s')(X'(s')- X'(s))}{(X(s')-X(s))^2} \right] ds' \\
&\; -\f{1}{4\pi}\int_{\T} \left\{\Im \left[\f{2X'(s')^2 X'(s) }{(X(s')-X(s))^3} \right] \Re\left[\frac{X(s')-X(s)}{X'(s)}\right]-
\Im \left[\f{X'(s')^2 }{(X(s')-X(s))^2} \right]\right\} ds'\\
=&\; \f{1}{4\pi}
\pv \int_{\T} \left\{\Im\left[\f{2X'(s')^2}{(X(s')-X(s))^2} \right]-\Im \left[\f{2X'(s')^2 X'(s) }{(X(s')-X(s))^3} \right]
\Re\left[\frac{X(s')-X(s)}{X'(s)}\right]\right\} ds' \\
=&\;\frac{1}{2\pi}\pv\int_{\T}\Re\left[\frac{X'(s')^2X'(s)}{(X(s')-X(s))^3}\right]
\Im\left[\frac{X(s')-X(s)}{X'(s)}\right]ds'.
\end{split}
\label{eqn: equation for alpha}
\eeq
In order to take its derivative later, let us give an alternative form of this equation.
Again using Lemma \ref{lem: some singular integrals along the string curve}, we derive from \eqref{eqn: equation for alpha} that
\beq
\begin{split}
&\;\Im\frac{\partial_t X'(s)}{X'(s)}\\
=&\;    \frac{1}{2\pi}\int_{\T} \left\{\Re\left[\frac{X'(s')^2X'(s)}{(X(s')-X(s))^3}\right]
\Im\left[\frac{X(s')-X(s)}{X'(s)}\right] -\frac{1}{2}\Im\left[\f{X'(s')X''(s)}{X'(s)(X(s')-X(s))}\right]\right\} ds'\\
&\;+\Im\left[\f{i}{4}\f{X''(s)}{X'(s)} \right].
\end{split}
\label{eqn: equation for alpha differentiable}
\eeq

\section{The Angle $\Phi(s_1,s_2)$}
\label{sec: angle Phi}
In this section, we first prove in Proposition \ref{prop: bound for the max |Phi|} that $\sup_{s_1,s_2\in \BT}|\Phi(s_1,s_2,t)|$ satisfies a maximum principle and a decay estimate if it is initially less than $\pi/4$.
This holds even when the elasticity law takes a more general form.
Then we study the geometric properties of the curve $X(\BT)$ when there is a bound for $|\Phi|$ (see Proposition \ref{prop: geometric property of the curve}).
Lastly, we state and prove Theorem \ref{thm: circular string} on the global solution starting from $H^1$-initial data with a circular shape.

\subsection{Maximum principle and decay of $\Phi$}
We start from deriving the equation for $\Phi(s_1,s_2)$.

\begin{lem}
\label{lem: equation for Phi}
For distinct $s',s_1,s_2\in \BT$, denote
\[
\th = \th(s',s_1,s_2): = \Phi(s',s_1)+\Phi(s',s_2) - \Phi(s_1,s_2).
\]
Then given distinct $s_1,s_2\in \BT$, in the sense of modulo $2\pi$, $\Phi(s_1,s_2)$ satisfies
\beq
\begin{split}
&\;\partial_t\Phi(s_1,s_2)\\
= &\; \frac{1}{4\pi}\pv\int_{\T}
\left\{\frac{|X'(s')|^2(\sin \th - \sin 2\Phi(s',s_1))}{|X(s')-X(s_1)|^2}
+ \frac{|X'(s')|^2(\sin \th - \sin 2\Phi(s',s_2))}{|X(s')-X(s_2)|^2}\right\} ds'.
\end{split}
\label{eqn: equation for Phi(s_1,s_2)}
\eeq
\begin{proof}
By \eqref{eqn: def of Phi},
\beq
\partial_t\Phi(s_1,s_2)
=\mathrm{Im}\frac{\partial_tX'(s_1)}{X'(s_1)}+\mathrm{Im}\frac{\partial_tX'(s_2)}{X'(s_2)}-
2\mathrm{Im}\frac{\partial_t(X(s_1)-X(s_2))}{X(s_1)-X(s_2)}.
\label{eqn: expression of time derivative of Phi}
\eeq
From the first line of \eqref{eqn: contour dynamic equation complex form}, we derive that
\begin{align*}
\mathrm{Im}\frac{\partial_t(X(s_1)-X(s_2))}{X(s_1)-X(s_2)}
=&\;
\frac{1}{4\pi}\text{p.v.}\int_{\T}\mathrm{Re}\left[\frac{X'(s')^2}{(X(s')-X(s_1))^2}\right]\mathrm{Im}\frac{X(s')-X(s_1)}{X(s_1)-X(s_2)}\,ds'\\
&\;- \frac{1}{4\pi}\text{p.v.}\int_{\T}\mathrm{Re}\left[\frac{X'(s')^2}{(X(s')-X(s_2))^2}\right]\mathrm{Im}\frac{X(s')-X(s_2)}{X(s_1)-X(s_2)}\,ds'.
\end{align*}
Note that 
\[
\mathrm{Im}\frac{X(s')-X(s_1)}{X(s_1)-X(s_2)}=\mathrm{Im}\frac{X(s')-X(s_2)}{X(s_1)-X(s_2)},
\]
so
\beqo
\begin{split}
&\; \mathrm{Im}\frac{\partial_t(X(s_1)-X(s_2))}{X(s_1)-X(s_2)}\\
=&\;
\frac{1}{4\pi}\text{p.v.}\int_{\T}\mathrm{Re}\left[\frac{X'(s')^2}{(X(s')-X(s_1))^2}-\frac{X'(s')^2}{(X(s')-X(s_2))^2}\right]
\mathrm{Im}\frac{X(s')-X(s_1)}{X(s_1)-X(s_2)}\,ds'\\
=&\;
\frac{1}{4\pi}\text{p.v.}\int_{\T}\mathrm{Re}\left[\frac{X'(s')^2(X(s_1)-X(s_2))}{(X(s')-X(s_1))^2(X(s')-X(s_2))}\right]
\mathrm{Im}\frac{X(s')-X(s_2)}{X(s_1)-X(s_2)}\,ds'\\
&\; +
\frac{1}{4\pi}\text{p.v.}\int_{\T}\mathrm{Re}\left[\frac{X'(s')^2(X(s_1)-X(s_2))}{(X(s')-X(s_2))^2(X(s')-X(s_1))}\right]
\mathrm{Im}\frac{X(s')-X(s_1)}{X(s_1)-X(s_2)}\,ds'.
\end{split}
\eeqo
Thanks to \eqref{eqn: equation for alpha}, we also have that
\begin{align*}
&\mathrm{Im}\frac{\partial_tX'(s_1)}{X'(s_1)}
=
\frac{1}{2\pi}\text{p.v.}\int_{\T}\mathrm{Re}\left[\frac{X'(s')^2X'(s_1)}{(X(s')-X(s_1))^3}\right]
\mathrm{Im}\frac{X(s')-X(s_1)}{X'(s_1)}\,ds',\\
&\mathrm{Im}\frac{\partial_tX'(s_2)}{X'(s_2)}
=
\frac{1}{2\pi}\text{p.v.}\int_{\T}\mathrm{Re}\left[\frac{X'(s')^2X'(s_2)}{(X(s')-X(s_2))^3}\right]
\mathrm{Im}\frac{X(s')-X(s_2)}{X'(s_2)}\,ds'.
\end{align*}
Hence,
\[
\partial_t\Phi(s_1,s_2)=\frac{1}{2\pi}\text{p.v.}\int_{\T} (K_1+K_2) \, ds'
\]
where
\begin{align*}
K_1:=&\;
     \mathrm{Re}\left[\frac{X'(s')^2X'(s_1)}{(X(s')-X(s_1))^3}\right]
     \mathrm{Im}\frac{X(s')-X(s_1)}{X'(s_1)}\\&-\mathrm{Re}\left[\frac{X'(s')^2(X(s_1)-X(s_2))}{(X(s')-X(s_1))^2(X(s')-X(s_2))}\right]
     \mathrm{Im}\frac{X(s')-X(s_2)}{X(s_1)-X(s_2)},\\
K_2:=&\;
     \mathrm{Re}\left[\frac{X'(s')^2X'(s_2)}{(X(s')-X(s_2))^3}\right]
     \mathrm{Im}\frac{X(s')-X(s_2)}{X'(s_2)}\\&-\mathrm{Re}\left[\frac{X'(s')^2(X(s_1)-X(s_2))}{(X(s')-X(s_2))^2(X(s')-X(s_1))}\right]
     \mathrm{Im}\frac{X(s')-X(s_1)}{X(s_1)-X(s_2)}.
\end{align*}
For arbitrary complex numbers $A,B,C$ with $B,C\neq 0$,
\beqo
\begin{split}
&\; \mathrm{Re}\left[A/B\right]\mathrm{Im}\,B- \mathrm{Re}\left[A/C\right]\mathrm{Im}\,C\\
=&\; \f{1}{4i}\left[(A/B+\overline{A}/\overline{B})(B-\overline{B})-(A/C+\overline{A}/\overline{C})(C-\overline{C})\right]
\\
=&\; \f{1}{4i}\left[\overline{A}(B/\overline{B}-C/\overline{C})-{A}(\overline{B}/B-\overline{C}/C)\right]\\
=&\; \f12 \mathrm{Im}\left[{A}\overline{C}/C- A\overline{B}/B\right]\\
=&\; \f12 \mathrm{Im}\left[A/C^2\right]|C|^2 - \f12 \Im\left[A/B^2\right]|B|^2.
\end{split}
\eeqo
Hence, we find that
\begin{align*}
K_1=
&\;
\f12 \mathrm{Im}\left[\frac{X'(s')^2(X(s_1)-X(s_2))^2}{(X(s')-X(s_1))^2(X(s')-X(s_2))^2}\right]
\frac{|X(s')-X(s_2)|^2}{|X(s_1)-X(s_2)|^2}\\
&\;-\f12 \mathrm{Im}\left[\frac{X'(s')^2X'(s_1)^2}{(X(s')-X(s_1))^4}\right]\frac{|X(s')-X(s_1)|^2}{|X'(s_1)|^2},
\\
= &\;\f12\left(\sin \big[\Phi(s',s_1)+\Phi(s',s_2) - \Phi(s_1,s_2)\big]-\sin 2\Phi(s',s_1)\right) \frac{|X'(s')|^2}{|X(s')-X(s_1)|^2},
\end{align*}
and similarly,
\[
K_2=
\f12 \left(\sin \big[\Phi(s',s_1)+\Phi(s',s_2) - \Phi(s_1,s_2)\big]-\sin 2\Phi(s',s_2)\right) \frac{|X'(s')|^2}{|X(s')-X(s_2)|^2}.
\]
Combining all these calculations, we conclude with \eqref{eqn: equation for Phi(s_1,s_2)}.
\end{proof}
\end{lem}

\begin{lem}
\label{lem: range of argument of theta}
Suppose that $\Phi(s_1,s_2)\in[\Phi_-, \Phi_+]\subset(-\f{\pi}{3},\f{\pi}{3})$ for all $s_1,s_2\in \T$.
Then for all distinct $s_0,s_1,s_2\in \T$,
\[
\th(s_0,s_1,s_2) = \Phi(s_0,s_1)+\Phi(s_0,s_2)-\Phi(s_1,s_2)\in[2\Phi_-, 2\Phi_+].
\]
\begin{proof}
Without loss of generality, we can assume $s_0<s_2<s_1<s_0+2\pi$.
Since $\frac{\partial I}{\partial s}(s_0,s)=J(s_0,s)$,
\[
\tilde{J}(s_0,s_1,s_2)
= I(s_0,s_1) - I(s_0,s_2)
= \int_{s_2}^{s_1}J(s_0,s)\,ds.
\]
Hence, $\arg\tilde{J}(s_0,s_1,s_2)\in[\Phi_-, \Phi_+]$.
We also note that
\[
\arg\tilde{J}(s_0,s_1,s_2)^2=\arg \frac{X'(s_0)^2(X(s_1)-X(s_2))^2}{(X(s_0)-X(s_1))^2(X(s_0)-X(s_2))^2}
= \Phi(s_0,s_1)+\Phi(s_0,s_2)-\Phi(s_1,s_2),
\]
which proves the claim.
\end{proof}
\end{lem}

\begin{lem}
\label{lem: zero Phi_+ or Phi_-}
Suppose that $\Phi(s_1,s_2)\in[\Phi_-, \Phi_+]\subset(-\f{\pi}{3},\f{\pi}{3})$ for all $s_1,s_2\in \T$.
If $\Phi_+\leq 0$ or $\Phi_-\geq 0$, then $\Phi(s_1,s_2)\equiv 0$ for all $(s_1,s_2)\in \BT$, and thus $X(\BT)$ is a circle.
\begin{proof}
Suppose $\Phi_+\leq 0$.
Then for all distinct $s',s\in \BT$,
\[
\Im\left[\f{X'(s')X'(s)}{(X(s')-X(s))^2}\right] = \f{|X'(s')||X'(s)|}{|X(s')-X(s)|^2}\sin \Phi(s',s) \leq 0,
\]
and the equality holds if any only if $\Phi(s',s) = 0$.
Hence, by Lemma \ref{lem: some singular integrals along the string curve}, $\Phi \equiv 0$.

Now fix $s\in \BT$.
Assume $X(s) = 0$ and $X'(s)$ is a positive real number without loss of generality;
one can always achieve this by suitable rotation and translation.
Then $\Phi\equiv 0$ implies that, for all $s'\neq s$,
\[
\Im\left[\f{X'(s')X'(s)}{(X(s')-X(s))^2}\right] \equiv 0,
\]
and thus $\f{X'(s')}{X(s')^2}\in \BR$.
Hence, there exists some constant $C_*$, such that
\[
\Im\left[\f{1}{X(s')}\right] = C_*\quad \forall\, s'\in \BT\setminus\{s\}.
\]
If $C_*=0$, $X(\BT)\subset \BR$.
This is not possible because $X(s)$ is injective.
If $C_* \neq 0$, $X(\BT)$ is contained in a circle that goes through the origin.
Since $X(s)$ is injective, $X(\BT)$ must be the whole circle.

The case $\Phi_-\geq 0$ can be handled similarly.
\end{proof}
\end{lem}

\begin{lem}
\label{lem: discussions on the angles}
Suppose $\Phi(s,s')\in[\Phi_-, \Phi_+]\subset[-\f{\pi}{4},\f{\pi}{4}]$ for all $s,s'\in \T$.
Let $\Phi_*:=\max\{\Phi_+,|\Phi_-|\}\leq \f{\pi}{4}$.
For distinct $s',s_1,s_2\in \T$, denote
\[
\Phi_1:=\Phi(s',s_1),\quad
\Phi_2:=\Phi(s',s_2),\quad
\theta:=\Phi_1+\Phi_2-\Phi(s_1,s_2).
\]
Then $\theta,2\Phi_1,2\Phi_2 \in[2\Phi_-, 2\Phi_+]$, and the following holds.

\begin{enumerate}[label = (\roman*)]

\item If $\Phi(s_1,s_2)=\Phi_+=\Phi_*$, then $|\Phi_1-\Phi_2|\leq \Phi_+$.
Moreover, $\sin\theta-\sin 2\Phi_j\leq 0$ $(j = 1,2)$, and
    \[
    (\sin\theta-\sin 2\Phi_1)+(\sin\theta-\sin 2\Phi_2)\leq 2\sin\Phi_* - 2\sin 2\Phi_*.
    \]

\label{statement: sin theta - sin 2 Phi_j |Phi_*| no greater than pi over 4}

\item If $\Phi(s_1,s_2)=\Phi_-=-\Phi_*$, then $|\Phi_1-\Phi_2|\leq |\Phi_-|$.
    Moreover, $\sin\theta-\sin 2\Phi_j\geq 0$ $(j = 1,2)$, and
    \[
    (\sin\theta-\sin 2\Phi_1)+(\sin\theta-\sin 2\Phi_2)\geq 2\sin 2\Phi_* -2\sin \Phi_*.
    \]
\end{enumerate}

\begin{proof}
That $\theta,2\Phi_1,2\Phi_2 \in[2\Phi_-, 2\Phi_+]$ follows from Lemma \ref{lem: range of argument of theta}.

Suppose $\Phi(s_1,s_2) = \Phi_+$.
If $\Phi_+ =0$, there is nothing to prove, as $\Phi\equiv 0$ by virtue of Lemma \ref{lem: zero Phi_+ or Phi_-}.
If $\Phi(s_1,s_2) = \Phi_+>0$, $s_1$ and $s_2$ must be distinct.
By Lemma \ref{lem: range of argument of theta}, for any $s'\neq s_1,s_2$,
\[
\Phi(s_1,s_2)+\Phi(s',s_1)-\Phi(s',s_2)\leq 2\Phi_+.
\]
Since $\Phi(s_1,s_2) = \Phi_+$, we obtain $\Phi_1-\Phi_2\leq \Phi_+$.
Interchanging $s_1$ and $s_2$ yields $\Phi_1-\Phi_2\geq -\Phi_+$.
Hence, $|\Phi_1-\Phi_2|\leq \Phi_+$.
This further implies $\th \leq 2\Phi_j$ for $j = 1,2$ by the definition of $\th$.
Given $\Phi_*\leq \f{\pi}{4}$,
$\th,2\Phi_j\in [-\f{\pi}{2},\f{\pi}{2}]$.
This together with $\th\leq 2\Phi_j$ implies $\sin\theta-\sin 2\Phi_j\leq 0$.

Now we additionally assume $\Phi_+=\Phi_*$.
We derive that
\[
(\sin\theta-\sin 2\Phi_1)+(\sin\theta-\sin 2\Phi_2)
= 2\sin \th - 2\sin (\Phi_1+\Phi_2) \cos(\Phi_1-\Phi_2).
\]
If $\Phi_1+\Phi_2 \leq 0$, we have
\begin{align*}
&\th+\Phi_+ /2 = \Phi_1+\Phi_2-\Phi_+/2 \leq \Phi_1 + \Phi_2 \leq 0,\\
&\th+\Phi_+ /2 \geq 2\Phi_- +\Phi_+/2\geq 
-2\Phi_* +\Phi_*/2 =-3\Phi_*/2
>-\pi/2.
\end{align*}
Hence,
\beq
\begin{split}
&\;(\sin\theta-\sin 2\Phi_1)+(\sin\theta-\sin 2\Phi_2)\\
\leq &\; 2\sin \th - 2\sin (\Phi_1+\Phi_2)=2\sin \th - 2\sin (\th+\Phi_+)\\
= &\; -4\cos (\th+\Phi_+/2)\sin (\Phi_+/2)\leq -4\cos (-3\Phi_*/2)\sin (\Phi_*/2)\\
= &\; 
2\sin \Phi_* - 2\sin 2\Phi_*.
\end{split}
\label{eqn: bound when Phi_1+Phi_2 is negative}
\eeq
Next consider $\Phi_1+\Phi_2 \geq 0$.
We find
\[
0 \leq |\Phi_1-\Phi_2|= 2\max(\Phi_1,\Phi_2) - (\Phi_1+\Phi_2)\leq 2\Phi_+ - \Phi_1-\Phi_2\leq 2\Phi_+ \leq \pi/2.
\]
This implies
\[
\begin{split}
&\;(\sin\theta-\sin 2\Phi_1)+(\sin\theta-\sin 2\Phi_2)\\
\leq &\; 2\sin\theta - 2\sin (\Phi_1+\Phi_2) \cos(2\Phi_+-\Phi_1-\Phi_2)\\
= &\; 2\sin\theta -\sin(2\Phi_1+2\Phi_2-2\Phi_+) -\sin (2\Phi_+)\\ 
= &\; 2\sin\theta -\sin 2\theta -\sin 2\Phi_* .
\end{split}
\]
Since $f(\theta):=2\sin\theta -\sin 2\theta$ is increasing on $[-\pi/2,\pi/2]$,
and $-\pi/2\leq 2\Phi_-\leq \th=\Phi_1+\Phi_2-\Phi_+\leq \Phi_+ =\Phi_* <\pi/2$,
we find $f(\th)\leq f(\Phi_*)$.
Hence,
\[
\begin{split}
&\;(\sin\theta-\sin 2\Phi_1)+(\sin\theta-\sin 2\Phi_2)\\
\leq &\;2\sin\Phi_* -\sin 2\Phi_*  -\sin  2\Phi_* = 2\sin\Phi_* -2\sin 2\Phi_* ,
\end{split}
\]
Combining this with  \eqref{eqn: bound when Phi_1+Phi_2 is negative}, we obtain
\ref{statement: sin theta - sin 2 Phi_j |Phi_*| no greater than pi over 4}.

The case $\Phi(s_1,s_2) = \Phi_-$ can be studied similarly, which is omitted.
\end{proof}
\end{lem}

Now we are ready to show that $|\Phi(s_1,s_2,t)|$ satisfies a maximum principle and a decay estimate if $\sup_{s_1,s_2}|\Phi(s_1,s_2,0)|< \f{\pi}{4}$.

\begin{prop}
\label{prop: bound for the max |Phi|}
Define $\Phi_*(t):=\sup_{s_1,s_2\in \BT}|\Phi(s_1,s_2,t)|$.
If $\Phi_*(0)<\f{\pi}{4}$, then $\Phi_*(t)$ is a non-increasing Lipschitz function.
It satisfies
\[
0\leq \Phi_*(t)\leq \Phi_*(0) \min\big\{e^{- \mu t},\,Ce^{-t/\pi^2}\big\},
\]
where $\mu := (4-2\sqrt{2})/\pi^3$ and where $C$ is a universal constant.

In particular, if $\Phi(s_1,s_2,0)\equiv 0$, then $\Phi(s_1,s_2,t)\equiv 0$ for all $t$.
In other words, if $X_0(\BT)$ is a circle, then $X(\BT,t)$ must be a circle of the same radius.

\begin{proof}
Take an arbitrary $t$.
We assume that for some distinct $s_1,s_2\in\BT$, $|\Phi(s_1,s_2,t)| = \Phi_*(t)$.
Without loss of generality, assume $\Phi(s_1,s_2,t)\geq 0$, so $\Phi_* = \Phi_+ \geq |\Phi_-|$.
We shall derive an upper bound for $\pa_t \Phi(s_1,s_2)$.

Recall that $d_*$ was defined in \eqref{eqn: diameter of X}.
By Lemma \ref{lem: equation for Phi} and Lemma \ref{lem: discussions on the angles}\ref{statement: sin theta - sin 2 Phi_j |Phi_*| no greater than pi over 4},
\beqo
\begin{split}
&\;\partial_t\Phi(s_1,s_2)\\
\leq &\; \frac{1}{4\pi} \int_{\T}
\frac{|X'(s')|^2}{d_*^2}
\big[(\sin \th - \sin 2\Phi(s',s_1))+
(\sin \th - \sin 2\Phi(s',s_2))\big]\, ds'\\
\leq &\; - \frac{\sin 2\Phi_* - \sin \Phi_*}{2\pi d_*^2}
\int_{\T} |X'(s')|^2\, ds'.
\end{split}
\eeqo
The Cauchy-Schwarz inequality implies that
\[
(2d_*)^2 \leq \CL(t)^2 \leq \int_{\T} |X'(s')|^2\, ds' \int_\BT 1\,ds',
\]
so we obtain
\beqo
\partial_t\Phi(s_1,s_2)
\leq - \frac{1}{\pi^2}\big(\sin 2\Phi_* - \sin \Phi_*\big).
\eeqo
Similar analysis can be carried out if $\Phi(s_1,s_2,t) = -\Phi_* \leq 0$.

Given the assumptions \ref{assumption: geometry}-\ref{assumption: non-degeneracy} on $X$, $\Phi(s_1,s_2,t)$ is $C^1$ in $\BT\times \BT\times [0,T]$, so $\Phi_*(t)$ is a Lipschitz function.
By combining the above estimates and following the argument in e.g.\;\cite{cordoba2009maximum}, we find that, if $\Phi_*(t)<\f{\pi}{4}$, then $\Phi_*(t)\geq 0$ satisfies for almost all $t$ that
\[
\Phi_*'(t)
\leq - \frac{1}{\pi^2}\big(\sin 2\Phi_*(t) - \sin \Phi_*(t)\big)
\leq - \f{4-2\sqrt{2}}{\pi^3}\Phi_*(t).
\]
Since
$\Phi_*(0)< \f{\pi}{4}$, this implies $\Phi_*(t)<\f{\pi}{4}$ for all $t\in [0,T]$,
and
$\Phi_*(t)\leq \Phi_*(0)e^{-\mu t}$ with $\mu = (4-2\sqrt{2})/\pi^3$.
On the other hand, the above differential inequality can be written as
\[
\f{d}{dt}\left[\f12 \ln |\cos\Phi_*-1| + \f16\ln (\cos \Phi_* +1 ) -\f23 \ln \left|\cos \Phi_* -\f12\right|\right] \leq -\f{1}{\pi^2},
\]
Hence, with $\Phi_*(0)< \f{\pi}{4}$,
\[
\f{(1-\cos \Phi_*) (\cos\Phi_* + 1)^{1/3}}{(\cos\Phi_* -\f12)^{4/3}}
\leq \f{(1-\cos \Phi_*(0)) (\cos\Phi_*(0) + 1)^{1/3}}{(\cos\Phi_*(0) -\f12)^{4/3}}
e^{-2t/\pi^2},
\]
and thus
\[
1-\cos \Phi_*(t)\leq C\Phi_*(0)^2e^{-2t/\pi^2},
\]
where $C>0$ is a universal constant.
Then the desired decay estimate follows.

The last claim follows from the monotonicity of $\Phi_*(t)$, Lemma \ref{lem: zero Phi_+ or Phi_-}, and the time-invariance of $R_X$.
\end{proof}
\end{prop}

\subsection{Remark on the case of general elasticity}
\label{sec: general elasticity}
Although this paper is mainly focused on the 2-D Peskin problem with a Hookean string, we make a detour in this subsection to show that the maximum principle for $|\Phi|$ also holds for strings with more general elasticity laws.
Some new notations and equations will be introduced here, but we note that they should only apply within this subsection and Remark \ref{rmk: general elasticity curvature} below.

Assume that in \eqref{eqn: Stokes equation}, $F_X(s,t)$ is given by
\[
F_X(s,t) = \pa_s \big[ q(s,t)X'(s,t)\big]
\]
for some positive function $q(s,t)$ that is as smooth as $X'(s,t)$ (cf.\;\eqref{eqn: Lagrangian representation of the elastic force} and \ref{assumption: geometry}-\ref{assumption: non-degeneracy} in Section \ref{sec: preliminary}).
Then \eqref{eqn: contour dynamic equation complex form} should be modified to become
\beq
\pa_t X(s)
= \f{1}{4\pi}\mathrm{p.v.}\int_\BT \Re \left[\f{X'(s')^2}{(X(s')-X(s))^2} \right] \big(X(s')-X(s)\big)q(s') \,ds'.
\label{eqn: contour dynamic equation complex form general tension}
\eeq
Suppose that $X = X(s,t)$ solves this equation with $X(0,t) = X_0(s)$, and satisfies the assumptions in Section \ref{sec: preliminary}.
We still define $\Phi(s_1,s_2,t)$ by \eqref{eqn: def of Phi}, and let $\Phi_*(t)$ be defined as in Proposition \ref{prop: bound for the max |Phi|}.
Then we claim that, if $\Phi_*(0)<\f{\pi}{4}$, $\Phi_*(t)$ should still be a non-increasing Lipschitz function in $t$.
This claim can be justified by simply following the same argument as above, but we shall present a different proof that avoids lengthy calculation.

Fix $t$.
Let
\[
k_0 : = 2\pi \left(\int_\BT\f{1}{q(s)}\,ds\right)^{-1} > 0.
\]
and define for all $s\in [0,2\pi)$,
\[
\xi (s) := \int_0^{s}\f{k_0}{q(s')}\,ds'.
\]
Thanks to the assumptions on $q$ (see \ref{assumption: geometry}-\ref{assumption: non-degeneracy}), 
$s\mapsto \xi(s)$ is a strictly increasing $C^4$-bijection from $[0,2\pi)$ to itself, so it can be further understood as a $C^4$-diffeomorphism from $\BT$ to $\BT$.
Then we define $Y(\xi,t)$ such that
\beqo
Y(\xi(s),t) \equiv X(s,t)\mbox{ for all }s\in \BT.
\eeqo
$Y$ is well-defined, with $Y(\BT,t) = X(\BT,t)$ and $Y'(\xi(s)) = k_0^{-1}q(s)X'(s)$.
In addition, by \eqref{eqn: contour dynamic equation complex form general tension} and change of variables,
\beqo
\begin{split}
\pa_t X(s,t)
= &\; \f1{4\pi}\pv \int_\BT \Re \left[\f{ k_0^{-1}q(s')^2X'(s')^2}{(X(s')-X(s))^2} \right] \big(X(s
')-X(s)\big)\f{k_0}{q(s')}\,ds'\\
= &\; \f1{4\pi}\pv \int_\BT \Re \left[\f{k_0 Y'(\xi(s'))^2}{(Y(\xi(s'))-Y(\xi(s)))^2} \right] \big(Y(\xi(s
'))-Y(\xi(s))\big)\,d\xi(s')\\
=
&\; \f1{4\pi}\pv\int_\BT \Re \left[\f{k_0 Y'(\eta)^2}{(Y(\eta)-Y(\xi(s)))^2} \right] \big(Y(\eta)-Y(\xi(s))\big)\,d\eta.
\end{split}
\eeqo
Note that given the assumptions on $q$, all the principal value integrals above can be justified.
If we let $Y(\xi,\tau)$ solve the Peskin problem with the Hookean elasticity (cf.\;\eqref{eqn: contour dynamic equation complex form}),
\[
\pa_\tau Y(\xi,\tau) = \f{1}{4\pi}\mathrm{p.v.}\int_\BT \Re \left[\f{ Y'(\eta,\tau)^2}{(Y(\eta,\tau)-Y(\xi,\tau))^2} \right] \big(Y(\eta,\tau)-Y(\xi,\tau)\big)\,d\eta,
\]
then for the given $t$,
\beq
\pa_t X(s,t) = k_0\pa_\tau Y(\xi(s),t).
\label{eqn: relation between time derivatives of X and Y}
\eeq
Please be reminded that this may not hold for later times.
Using this, we calculate that
\begin{align}
&\Im\f{\pa_t X'(s,t)}{X'(s,t)}
=\Im \f{k_0\pa_\tau Y'(\xi(s),t)\xi'(s)}{Y'(\xi(s),t)\xi'(s)}
=k_0 \left.\left[\Im \f{\pa_\tau Y'(\xi,\tau)}{Y'(\xi,\tau)}\right]\right|_{(\xi,\tau) = (\xi(s),t)},\label{eqn: time derivative of alpha change of variables}\\
&\Im\f{\pa_t(X(s_1,t)-X(s_2,t))}{X(s_1,t)-X(s_2,t)}
= k_0\left. \left[\Im\f{\pa_\tau(Y(\xi_1,\tau)-Y(\xi_2,\tau))}{Y(\xi_1,\tau)-Y(\xi_2,\tau)}\right]
\right|_{(\xi_1,\xi_2,\tau) = (\xi(s_1),\xi(s_2),t)}.\nonumber
\end{align}
Define $\Phi_Y = \Phi_Y(\xi_1,\xi_2,\tau)$ in terms of $Y(\xi,\tau)$ by \eqref{eqn: def of Phi}.
Combining these identities with \eqref{eqn: expression of time derivative of Phi} yields
\beq
\pa_t \Phi(s_1,s_2,t) = k_0 \pa_\tau \Phi_Y(\xi(s_1),\xi(s_2),t).
\label{eqn: relation between time derivatives of two different Phis}
\eeq
Following the argument in Proposition \ref{prop: bound for the max |Phi|}, we conclude that $\Phi_*(t)$ is a non-increasing Lipschitz function, which satisfies
\[
\Phi_*'(t) \leq  - \frac{k_0}{\pi^2}\big(\sin 2\Phi_*(t) - \sin \Phi_*(t)\big)
\]
for almost all $t$.
Given further information of $q(s,t)$, a quantitative decay estimate for $\Phi_*(t)$ may also be established, but we omit the discussion here.

It may be of independent interest that,
by Lemma \ref{lem: equation for Phi}, \eqref{eqn: relation between time derivatives of two different Phis}, and the definitions of $Y$, $\Phi_Y$, and $\th_Y$, $\Phi(s_1,s_2,t)$ in this case satisfies
\beqo
\begin{split}
\pa_t\Phi(s_1,s_2)
= &\; \frac{1}{4\pi} \pv\int_{\T}
\left\{\frac{|X'(s')|^2(\sin \th(s',s_1,s_2)- \sin 2\Phi(s',s_1))}{|X(s')-X(s_1)|^2}\right. \\
&\;\qquad  \qquad \quad \left.
+ \frac{|X'(s')|^2(\sin \th(s',s_1,s_2) - \sin 2\Phi(s',s_2))}{|X(s')-X(s_2)|^2}\right\}q(s')\, ds'.
\end{split}
\eeqo


\subsection{Geometric characterizations of $X(\BT,t)$}

It is conceivable that the bound for $\Phi_*(t)$ should provide useful geometric information of the curve $X(\BT,t)$.
We prove a few in the following proposition.
The time-dependence is omitted.

\begin{prop}
\label{prop: geometric property of the curve}
Suppose $ \Phi_* < \pi/2$. 
Recall that $R_X$ and $d_*$ were defined in \eqref{eqn: effective radius} and \eqref{eqn: diameter of X}, respectively.
Then the following holds.

\begin{enumerate}[label = (\roman*)]
\item \label{property: diameter lower bound} $d_*\geq 2R_X$.
\item
\label{property: polar coordinate representation}
There exists $z_*\in \BC\setminus X(\BT)$, such that
$X(\BT)$ can be parameterized in the polar coordinate centered at $z_*$, i.e., there exists $\rho = \rho(\om)$ defined on $\BT$, such that
\[
X(s) - z_* = \rho(\om(s))e^{i\om(s)}, \mbox{ where }\om(s) = \arg(X(s)-z_*) \in \BT = [-\pi,\pi).
\]
Here $\om(\cdot)$ is a bijection from $\BT$ to $\BT$, and it is orientation-preserving in the sense that, as $s$ increases from $-\pi$ to $\pi$ in $\BT$, $\om(s)$ goes through $\BT$ in the same positive (counter-clockwise) direction.
Moreover, $|\rho'(\om)/\rho(\om)|\leq \tan \Phi_*$ and
\beq
\rho(\om)\in \left[\f{d_*}2 \tan \left(\f{\pi}{4}-\f{\Phi_*}{2}\right),\,
\f{d_*}2 \tan \left(\f{\pi}{4}+\f{\Phi_*}{2}\right)\right].
\label{eqn: upper and lower bounds for rho}
\eeq
As a result,
\[
R_X \geq \f{d_*}2 \tan \left(\f{\pi}{4}-\f{\Phi_*}{2}\right).
\]

\item
Let $d(s) := \sup_{s'\in \BT} |X(s)-X(s')|$.
For all $s\in \BT$,
\beq
d(s)\geq d_* \tan\left(\f{\pi}{4}-\f{\Phi_*}{2}\right)\geq 2R_X \tan\left(\f{\pi}{4}-\f{\Phi_*}{2}\right).
\label{eqn: largest distance from a point}
\eeq

\item \label{property: chord-arc}
The curve $X(\BT)$ satisfies a chord-arc condition, i.e., for all $s,s'\in \BT$,
\beq
|X(s)-X(s')|\geq d_* \tan \left(\f{\pi}{4}-\f{\Phi_*}{2}\right) \sin \left[\f{L(s,s')}{d_*}(1-\sin \Phi_*)\right].
\label{eqn: chord-arc condition}
\eeq
Here, $L(s,s')$ denotes the length of the shorter arc between $X(s)$ and $X(s')$.
For $s\leq s'\leq s+2\pi$ without loss of generality,
\[
L(s,s')=L(s',s)= \min\left\{\int_{s}^{s'}|X'(s'')|\,ds'',\, \int_{s'}^{s+2\pi}|X'(s'')|\,ds''  \right\}.
\]
It satisfies
\beq
L(s,s')\leq \f{\pi d_*}{2(1-\sin \Phi_*)}.
\label{eqn: upper bound for arclength}
\eeq
In particular, if $\Phi_*\leq \pi/4$, there exists some universal constant $C>0$, such that
\[
|X(s)-X(s')|\geq C  L(s,s').
\]

\item Recall that $\CL(t)$ was defined in \eqref{eqn: length}.
If $\Phi_*\leq \pi/4$, for any $s\in \BT$,
\beq
c_1 R_X \leq 2d(s)\leq 2d_* \leq \CL(t) \leq c_2 d_* \leq c_3 R_X,
\label{eqn: R_X d_* CL comparable}
\eeq
where $c_j$ $(j = 1,2,3)$ are universal constants.
\end{enumerate}

\begin{proof}
By definition, $d_*$ is also the diameter of the convex hull of $X(\BT)$.
By the isodiametric inequality \cite{Gruber2007}, $\pi R_X^2 \leq \f14 \pi d_*^2$, which gives $d_*\geq 2R_X$.

Assume that $d_*$ is attained at some $s_1,s_2\in \BT$, i.e., $d_* = |X(s_1)-X(s_2)|$.
By the maximality of $|X(s_1)-X(s_2)|^2$,
\beqo
\big|\arg X'(s_j) - \arg (X(s_1)-X(s_2))\big|
= \f{\pi}{2},\quad j = 1,2.
\eeqo
Since $|\Phi(s_1,s_2)|< \pi/2$, it must hold $|\arg X'(s_1)-\arg X'(s_2) | = \pi$ (as otherwise $|\Phi(s_1,s_2)|=\pi $) and
\beq
\arg \left[\f{X'(s_1)}{X(s_1)-X(s_2)}\right]
= -\arg \left[\f{X'(s_2)}{X(s_1)-X(s_2)}\right] \in \left\{\f{\pi}{2},\,- \f{\pi}{2}\right\}.
\label{eqn: X'(s_j) is perpendicular to X(s_1)-X(s_2)}
\eeq

Denote the straight line that goes through $X(s_1)$ and $X(s_2)$ by $l$.
We claim that $X(\BT)\cap l = \{X(s_1),\, X(s_2)\}$.
Indeed, if not, suppose $X(s)$ lies on this line $(s\neq s_1,s_2)$.
Since
\[
\f{J(s_1,s)}{J(s_2,s)}
=\f{X'(s_1)}{X'(s_2)}\cdot \f{(X(s)-X(s_2))^2}{(X(s)-X(s_1))^2},
\mbox{ and }\f{(X(s)-X(s_2))^2}{(X(s)-X(s_1))^2}>0,
\]
we find
$\Phi(s_1,s)-\Phi(s_2,s)=\arg [J(s_1,s)/J(s_2,s)]=\arg X'(s_1)-\arg X'(s_2) =\pi$ in the modulo $2\pi$.
This contradicts with $|\Phi(s_1,s)|< \pi/2$ and $|\Phi(s_2,s)|< \pi/2$.
Let $z_* := (X(s_1)+X(s_2))/2\not \in X(\BT)$.
In the rest of the proof, we want to show that the curve $X(\BT)$ can be parameterized in the polar coordinate centered at $z_*$.
We proceed in three steps.

\setcounter{step}{0}
\begin{step}
Take an arbitrary $s\in \BT$.
We first bound $|X(s)-z_*|$.
%
By the definition of $\Phi$,
\beqo
\begin{split}
&\; 2\big|\arg(X(s)-X(s_1)) - \arg(X(s)-X(s_2))\big| \\
= &\; \big|\arg X'(s_1)- \arg X'(s_2) + \Phi(s,s_2) - \Phi(s,s_1)\big|
\in \big[\pi-2\Phi_*,\, \pi + 2\Phi_*\big].
\end{split}
\eeqo
Hence,
\[
\left|\cos \big[\arg(X(s)-X(s_1)) - \arg(X(s)-X(s_2))\big]\right|\leq \sin \Phi_*.
\]
By virtue of the law of cosines,
\begin{align*}
d_*^2 = |X(s_1)-X(s_2)|^2
\leq &\;
\left(1+ \sin \Phi_*\right)\big[|X(s)-X(s_1)|^2 +  |X(s)-X(s_2)|^2\big],\\
d_*^2 = |X(s_1)-X(s_2)|^2
\geq &\; (1-\sin \Phi_*)\big[|X(s)-X(s_1)|^2+|X(s)-X(s_2)|^2\big].
\end{align*}
Then using the parallelogram identity,
\beqo
2|X(s)-X(s_1)|^2+2|X(s)-X(s_2)|^2 = d_*^2+ 4|X(s)-z_*|^2,
\eeqo
we obtain that
\beq
\f{d_*}2 \tan \left(\f{\pi}{4}-\f{\Phi_*}{2}\right) \leq |X(s)-z_*|
\leq
\f{d_*}2 \tan \left(\f{\pi}{4}+\f{\Phi_*}{2}\right).
\label{eqn: upper and lower bound for distance to the midpoint}
\eeq
\end{step}

\begin{step}
Next we bound the angle between $X'(s)$ and $(X(s)-z_*)$ by proving that
\beq
\left|\cos \left(\arg\left[ \f{X'(s)}{X(s)-z_*}\right]\right)\right| \leq \sin \Phi_*.
\label{eqn: estimate for cosine of the angle}
\eeq
For convenience, denote $z := X(s)$ and let
\[
\psi(w) := \ln \left(\f{w-X(s_1)}{w-X(s_2)}\right),\quad w\in \mathbb{C}\setminus l.
\]
$\psi$ is well-defined as a single-valued function on $\mathbb{C}\setminus l$, with $\Im\, \psi(w) \in (-\pi,\pi)\setminus \{0\}$.
Define
\[
\zeta := \overline{\psi'(z)} = \overline{\left(\f{X(s_1)-X(s_2)}{(z-X(s_1))(z-X(s_2))}\right)}.
\]
It is worth noting that $\zeta$ gives the tangent direction of the level set of $\Im \,\psi(w)$ at the point $z$, because $\Im(\zeta \psi'(z)) = 0$.
Then we derive that
\begin{align*}
\Phi(s,s_1) 
= &\; \arg \left[\f{X'(s)}{\zeta}\cdot\f{X'(s_1)}{X(s_1)-X(s_2)} \cdot \f{z-X(s_2)}{z-X(s_1)}\right]\\
= &\; \arg \left[\f{X'(s)}{\zeta}\right]
+ \arg \left[\f{X'(s_1)}{X(s_1)-X(s_2)} \right] - \Im\,\psi(z),
\end{align*}
and similarly,
\[
\Phi(s,s_2)
= \arg \left[\f{X'(s)}{\zeta}\right]
+ \arg \left[\f{X'(s_2)}{X(s_1)-X(s_2)} \right] +\Im\,\psi(z).
\]
These equalities should be understood in the modulo $2\pi$.

Take sines and cosines on both sides of them.
In view of \eqref{eqn: X'(s_j) is perpendicular to X(s_1)-X(s_2)}, $|\sin \Phi(s,s_j)|\leq \sin \Phi_*$, and $\cos \Phi(s,s_j)\geq \cos \Phi_*>0$,
we find that
\beqo
\left|\cos \left(\arg \left[\f{X'(s)}{\zeta}\right] \pm \Im\,\psi(z)\right)\right|\leq \sin \Phi_*,
\eeqo
and
\beqo
\sin \left(\arg \left[\f{X'(s)}{\zeta}\right] \pm \Im\,\psi(z)\right)\mbox{ have opposite signs}.
\eeqo
For convenience, denote
\[
\g_1:= \arg \left[\f{X'(s)}{\zeta}\right],\quad
\g_2:= \Im\,\psi(z).
\]
Then they further imply that
\begin{align}
&|\cos \g_1 \cos \g_2|+ |\sin \g_1 \sin \g_2|\leq \sin \Phi_*, \label{eqn: bounding the angle inequality 1}\\
&|\sin \g_1 \cos \g_2|-|\cos \g_1 \sin \g_2| < 0.
\label{eqn: sin having opposite signs}
\end{align}

We also derive that
\begin{align*}
\g_3:= \arg \left[\f{z-z_*}{\zeta}\right]
= &\; \arg \left[[(z-X(s_1)) + (z-X(s_2))]  \psi'(z)\right]\\
= &\;\arg \left[\f{z-X(s_2)}{z-X(s_1)} -\f{z-X(s_1)}{z-X(s_2)}\right]
= \arg \left[e^{-\psi(z)} - e^{\psi(z)}\right].
\end{align*}
Hence,
\beqo
\begin{split}
|\cos \g_3|
= &\; \f{|\Re (e^{-\psi(z)} - e^{\psi(z)})| }
{|e^{-\psi(z)} - e^{\psi(z)}|}\\
\leq &\; \f{(e^{|\Re \,\psi(z)|} - e^{- |\Re\,\psi(z)|})|\cos \Im\,\psi(z)| }
{e^{|\Re \,\psi(z)|} - e^{- |\Re\,\psi(z)|}}
=
|\cos \g_2|,
\end{split}
\eeqo
which also gives $|\sin \g_2|\leq |\sin \g_3|$.
Using this and \eqref{eqn: sin having opposite signs} yields
\[
|\sin \g_1 \cos \g_3| \leq |\sin \g_1 \cos \g_2|
<|\cos \g_1 \sin \g_2| \leq |\cos \g_1 \sin \g_3|.
\]
Hence,
\[
|\sin \g_1 \cos \g_3|-|\cos \g_1 \sin \g_3|
\leq |\sin \g_1 \cos \g_2|-|\cos \g_1 \sin \g_2| < 0.
\]
Since for $j = 2,3$,
\[
1 = \big(|\cos \g_1 \cos \g_j|+ |\sin \g_1 \sin \g_j|\big)^2 + \big(|\sin \g_1 \cos \g_j|-|\cos \g_1 \sin \g_j| \big)^2,
\]
we use \eqref{eqn: bounding the angle inequality 1} to obtain that
\[
|\cos \g_1 \cos \g_3|+ |\sin \g_1 \sin \g_3|
\leq |\cos \g_1 \cos \g_2|+ |\sin \g_1 \sin \g_2|
\leq \sin \Phi_*.
\]
Since
\[
\arg\left[ \f{X'(s)}{X(s)-z_*}\right] = \arg \left[\f{X'(s)}{\zeta}\right] - \arg \left[\f{z-z_*}{\zeta}\right]
= \g_1-\g_3,
\]
\eqref{eqn: estimate for cosine of the angle} follows immediately.
\end{step}

\begin{step}
With $\om \in \BT = [-\pi,\pi)$, let $l_\om :=\{z_* + te^{i\om}:\, t\geq 0\}$ be the ray emanating from $z_*$ with the directional angle $\om$.
Then that $X(\BT)\cap l = \{X(s_1),\, X(s_2)\}$ implies that $l_\om \cap X(\BT)$ contains exactly one element when $\om = \arg(X(s_j) - z_*)$ $(j = 1,2)$.
In addition, \eqref{eqn: estimate for cosine of the angle} implies that the number of element of $l_\om\cap X(\BT)$ changes continuously as $\om$ varies, and thus it must be identically $1$.
Therefore, $X(\BT)$ can be parameterized in the polar coordinate centered at $z_*$.
More precisely, there exists $\rho = \rho(\om)$ defined on $\BT $ such that $X(s) - z_* = \rho(\om(s))e^{i\om(s)}$, where $\om(s) = \arg(X(s)-z_*) \in \BT$.
Here $\om(\cdot )$ is bijective from $\BT$ to $\BT$.
Since $X(\cdot,t)$ is assumed to parameterize the curve $X(\BT,t)$ in the counter-clockwise direction, the map $\om(s)$ is orientation-preserving.
\eqref{eqn: estimate for cosine of the angle} implies that $|\rho'(\om)/\rho(\om)|\leq \tan \Phi_*$.
Moreover, \eqref{eqn: upper and lower bound for distance to the midpoint} gives \eqref{eqn: upper and lower bounds for rho}.
The estimate $R_X \geq \f{d_*}2 \tan (\f{\pi}{4}-\f{\Phi_*}{2})$ follows from the fact that the disk of radius $\f{d_*}2 \tan (\f{\pi}{4}-\f{\Phi_*}{2})$
centered at $z_*$ is contained in the interior of $X(\BT)$.
\eqref{eqn: largest distance from a point} follows from part \ref{property: diameter lower bound}
and part \ref{property: polar coordinate representation}. 

For arbitrary $s,s'\in \BT$,
\beqo
|X(s)-X(s')|\geq d_* \tan \left(\f{\pi}{4}-\f{\Phi_*}{2}\right) \left|\sin \f{\om(s)-\om(s')}{2}\right|.
\eeqo
Indeed, by the law of cosines, with $\r_\dag := \min\{\r(\om(s)),\r(\om(s'))\}$,
\begin{align*}
|X(s)-X(s')| = &\; \left[\r(\om(s))^2 + \r(\om(s'))^2 -2\r(\om(s))\r(\om(s'))\cos \big(\om(s)-\om(s')\big)\right]^{1/2}\\
\geq &\; \left[\r_\dag^2 + \r_\dag^2 -2\r_\dag^2\cos \big(\om(s)-\om(s')\big)\right]^{1/2}\\
\geq &\; d_* \tan \left(\f{\pi}{4}-\f{\Phi_*}{2}\right) \left|\sin \f{\om(s)-\om(s')}{2}\right|.
\end{align*}
On the other hand,
assume $\om (s')\in [\om(s),\om(s)+\pi]$ without loss of generality (otherwise, one may consider $\om(s)\in [\om(s'),\om(s')+\pi]$ instead).
By \eqref{eqn: upper and lower bounds for rho} and \eqref{eqn: estimate for cosine of the angle},
\begin{align*}
\cos \Phi_* L(s,s')
\leq &\; \cos \Phi_* \int_{\om(s)}^{\om(s')} \big(\r(\eta)^2+\r'(\eta)^2\big)^{1/2} \,d\eta\\
\leq &\; \int_{\om(s)}^{\om(s')} \r(\eta) \,d\eta
\leq \f{d_*}2 \tan \left(\f{\pi}{4}+\f{\Phi_*}{2}\right) |\om(s)-\om(s')|_\BT,
\end{align*}
where $|\om(s)-\om(s')|_\BT\in [0,\pi]$ denotes the distance between $\om(s)$ and $\om(s')$ along $\BT$.
Then \eqref{eqn: chord-arc condition} and \eqref{eqn: upper bound for arclength} follow.
Finally, \eqref{eqn: R_X d_* CL comparable} follows from parts \ref{property: diameter lower bound}-\ref{property: chord-arc} and the fact that
$\CL(t) = 2\sup_{s'}L(s,s')$ for arbitrary $s\in \BT$.
\end{step}
\end{proof}
\end{prop}

If we assume $\Phi_*(0)<\pi/4$, Proposition \ref{prop: geometric property of the curve} implies that $X_0(\BT)$ can have at most a medium-size deviation from being a perfect circle.
The following remark may be viewed a converse statement in some sense, which says that if $X_0(s)$ is $O(1)$-close to an equilibrium in the $C^1(\BT)$-seminorm up to suitable reparameterizations, then the corresponding $\Phi_0$ will satisfy $|\Phi_0|<\pi/4$.

\begin{rmk}
\label{rmk: initial data C1 close to equilibrium}
There exists a universal constant $c>0$, such that, if $X_0\in C^1(\BT)$ satisfies that
\[
\big\|X_0(s)-(x_0+R e^{i(s+\xi_0)})\big\|_{\dot{C}^1(\BT)} \leq cR
\]
for some $x_0\in \BC$, $R>0$, and $\xi_0\in \BT$, then $\Phi_0(s,s')$ defined in terms of $X_0$ will satisfy $|\Phi_0(s,s')|<\pi/4$ for all $s,s'\in \BT$.
Here $R$ does not have to coincide with $R_X$.
More generally, if we additionally let $\psi:\BT\to \BT$ be an arbitrary suitably smooth bijective diffeomorphism, and define $\tilde{\Phi}_0(s,s')$ in terms of the reparameterized configuration $\tilde{X}_0 = X_0\circ \psi$,
then $|\tilde{\Phi}_0(s,s')|<\pi/4$ still holds.

Indeed, the second claim follows from the first one and the fact that $\tilde{\Phi}_0(s,s') = \Phi_0(\psi(s),\psi(s'))$ for all $s,s'\in \BT$.

To show the first claim, we denote $Y_0 (s) := x_0+ R e^{i(s+\xi_0)}$.
Then $|Y_0'(s)|\equiv R$, and for distinct $s, s'\in \BT$,
\[
\arg \left[\f{Y_0'(s)Y_0'(s')}{(Y_0(s)-Y_0(s'))^2}\right] \equiv 0.
\]
Assume $c\leq 1/2$.
Then the facts $|X_0'(s)-Y_0'(s)|\leq cR$ and $|Y_0'(s)| = R$ imply that
\[
\big|\arg \big(X_0'(s)/Y_0'(s)\big) \big| \leq \arcsin c.
\]
For distinct $s,s'\in \BT$, we assume $s<s'\leq s+\pi$ without loss of generality.
Then
\[
\left|\f{Y_0(s')-Y_0(s)}{s'-s}\right|
= \f{\big|2R \sin(\f{s'-s}{2})\big|}{|s'-s|_{\BT}} \geq \f{2}{\pi}R,
\]
and
\[
\left|\f{X_0(s')-X_0(s)}{s'-s} - \f{Y_0(s')-Y_0(s)}{s'-s}\right|
\leq \f{1}{|s'-s|_\BT}
\int_{s}^{s'}\big|X_0'(s'')-Y_0'(s'')\big|\,ds'' \leq c R.
\]
Hence,
\[
\left|\arg \left[\f{Y_0(s')-Y_0(s)}
{X_0(s')-X_0(s)}\right]\right|
=\left| \arg \left[\f{\f{1}{s'-s}(Y_0(s')-Y_0(s))}
{\f{1}{s'-s}(X_0(s')-X_0(s))}\right]\right|
\leq \arcsin \f{\pi c}{2},
\]
which implies
\[
|\Phi_0(s,s') |
=
\left|\arg \left[\f{X_0'(s)X_0'(s')}{Y_0'(s)Y_0'(s')}
\left(\f{Y_0(s')-Y_0(s)}
{X_0(s')-X_0(s)}\right)^{2}\right]\right|
\leq 2\arcsin c + 2\arcsin \f{\pi c}{2}.
\]
Assuming $c$ to be suitably small but universal, we can achieve that $|\Phi_0(s,s')|<\pi/4$.
\end{rmk}

\subsection{Strings with a circular shape}

In Proposition \ref{prop: bound for the max |Phi|}, we show that
if the initial curve $X_0(\BT)$ is a circle of radius $R_X$ (see \eqref{eqn: effective radius}), then for all $t>0$ at which the solution $X(s,t)$ is well-defined, $X(\BT,t)$ should still be a circle of radius $R_X$, possibly centered at a different point.
Given that, it seems feasible to pursue a more precise characterization of the solution starting from such initial data with a circular shape.
We thus make the following derivation.

Under the assumption that $\tilde{X}(s,t)\neq 0$ for all $(s,t)$, we consider the following normalized problem (we still use the notation of complex numbers):
\beq
\pa_t \tilde{X}(s) = \f{1}{8\pi}\pv
\int_{\BT} \left[\frac{\tilde{X}'(s')^2}{(\tilde{X}(s')-\tilde{X}(s))^2}+
\frac{\overline{\tilde{X}'(s')}^2}{(\overline{\tilde{X}(s')}-\overline{\tilde{X}(s)})^2}\right]\big(\tilde{X}(s')-\tilde{X}(s)\big)\,ds'
- v(t),
\label{eq:2}
\eeq
with
\beq
v(t)= \f{1}{8\pi}\int_{\BT}\f{\tilde{X}'(s')^2}{\tilde{X}(s')}\,ds'.
\label{eq:3}
\eeq
It is clear that, if $(\tilde{X},v)$ solves \eqref{eq:2}-\eqref{eq:3}, then $\tilde{X}(s)+\int_0^t v(\tau)\,d\tau$ satisfies
\eqref{eqn: contour dynamic equation complex form}.
In addition, we find $\tilde{X}$ enjoys the following property.

\begin{lem}
Suppose $\tilde{X}_0(s)$ satisfies that $\tilde{X}_0(\BT)$ is a circle of radius $R_X$ centered at the origin, and $|\tilde{X}_0'(s)|>0$.
Let $\tilde{X}(s,t)$ be a solution to \eqref{eq:2}-\eqref{eq:3} in $\BT\times [0,T]$ corresponding to the initial condition $\tilde{X}(s,0) = \tilde{X}_0(s)$, satisfying the assumptions \ref{assumption: geometry}-\ref{assumption: non-degeneracy} with $X$ there replaced by $\tilde{X}$
and that 
$|\tilde{X}(s,t)|>0$.
%
Then $\tilde{X}(\BT,t)$ is the circle of radius $R_X$ centered at the origin for all $t\in [0,T]$.

\begin{rmk}
In view of this, $X(\BT,t)$ should be a circle of all time, and $v(t)$ in \eqref{eq:3} can be interpreted as the velocity of its center.
\end{rmk}

\begin{proof}
Thanks to the assumptions on $\tilde{X}$, it suffices to prove $|\tilde{X}(s,t)|\equiv R_X$ for all $(s,t)\in \BT\times [0,T]$.
Plugging \eqref{eq:3} into \eqref{eq:2} yields that
\beq
\label{eq:4}
\pa_t \tilde{X}(s)=\f{1}{8\pi}\pv \int_{\BT}
\left[\f{\tilde{X}'(s')^2\tilde{X}(s)}{(\tilde{X}(s')-\tilde{X}(s))\tilde{X}(s')} +\overline{\tilde{X}'(s')}^2\cdot \f{\tilde{X}(s')-\tilde{X}(s)}{(\overline{\tilde{X}(s')}-\overline{\tilde{X}(s)})^2}\right]ds'.
\eeq
Since
\[
\f{\tilde{X}(s')-\tilde{X}(s)} {(\overline{\tilde{X}(s')}-\overline{\tilde{X}(s)})^2}
=-\f{\tilde{X}(s)}{(\overline{\tilde{X}(s')}-\overline{\tilde{X}(s)})\overline{\tilde{X}(s')}}
+\f{|\tilde{X}(s')|^2-|\tilde{X}(s)|^2}{(\overline{\tilde{X}(s')}-\overline{\tilde{X}(s)})^2\overline{\tilde{X}(s')}},
\]
\eqref{eq:4} becomes
\beq
\label{eq:5}
\begin{split}
\f{\pa_t \tilde{X}(s)}{\tilde{X}(s)}
= &\; \f{1}{8\pi}\pv \int_{\BT}\left[\frac{\tilde{X}'(s')^2}{(\tilde{X}(s')-\tilde{X}(s))\tilde{X}(s')}-
\frac{\overline{\tilde{X}'(s')}^2}{(\overline{\tilde{X}(s')}-\overline{\tilde{X}(s)})\overline{\tilde{X}(s')}} \right]ds'\\
&\; +\f{1}{8\pi}\pv \int_{\BT} \overline{\tilde{X}'(s')}^2\cdot \f{|\tilde{X}(s')|^2-|\tilde{X}(s)|^2}{(\overline{X(s')}-\overline{X(s)})^2\overline{\tilde{X}(s')}\tilde{X}(s)}\, ds'.
\end{split}
\eeq
Hence,
\beq
\begin{split}
\f{\pa_t|\tilde{X}(s)|}{|\tilde{X}(s)|}
= &\; \Re\left[\f{\pa_t \tilde{X}(s)}{\tilde{X}(s)}\right]
=\f{1}{8\pi}
\pv\int_{\BT}\Re\left[\f{\overline{\tilde{X}'(s')}^2(|\tilde{X}(s')|^2-|\tilde{X}(s)|^2)}
{(\overline{\tilde{X}(s')}-\overline{\tilde{X}(s)})^2\overline{\tilde{X}(s')}\tilde{X}(s)}\right]ds'\\
= &\; \f{1}{8\pi}
\pv\int_{\BT}
\Re\left[\f{\tilde{X}'(s')^2\tilde{X}(s)}{(\tilde{X}(s')-\tilde{X}(s))^2\tilde{X}(s')}\right] \left(\f{|\tilde{X}(s')|^2}{|\tilde{X}(s)|^2} - 1\right) ds'.
\end{split}
\label{eqn: eqn for ln |X|}
\eeq

Define for $s\neq s'$ and $t\in [0,T]$ that
\[
B(s,s',t) :=
\Re\left[\f{\tilde{X}'(s',t)^2\tilde{X}(s,t)} {(\tilde{X}(s',t)-\tilde{X}(s,t))^2\tilde{X}(s',t)}\right].
\]
Let
\[
T_* : = \sup\left\{t\in [0,T]:\,
B(s,s',\tau) \geq 0 \mbox{ for all }s\neq s'\mbox{ and }\tau\in [0,t)\right\}.
\]
We claim that $T_*>0$.
Suppose not.
Then there exists a sequence $\{(s_k, s_k', t_k)\}_{k = 1}^\infty$, such that $t_k\to 0^+$ as $k\to +\infty$, $s_k\neq s_k'$, and $B(s_k,s_k',t_k)<0$.
Using the smoothness assumptions on $\tilde{X}$,
one can justify by Taylor expansion that there exists $\d>0$ depending on $\tilde{X}$, such that $B(s,s',t)\geq 0$ whenever $t\in [0,\d]$ and $|s-s'|_\BT\leq \d$.
Hence, up to a subsequence, we may assume that for some $s,s'\in \BT$, $s_k \to s$ and $s_k'\to s'$ as $k\to +\infty$, where $|s-s'|_\BT\geq \d$.
By the continuity of $\tilde{X}$,
\begin{align*}
0\geq \lim_{k\to +\infty} B(s_k,s_k',t_k) = B(s,s',0)
= &\; \Re\left[\f{\tilde{X}_0'(s')\tilde{X}_0'(s)}{(\tilde{X}_0(s')-\tilde{X}_0(s))^2}
\cdot \f{\tilde{X}_0'(s')/\tilde{X}_0(s')}{\tilde{X}_0'(s)/\tilde{X}_0(s)}
\right]\\
= &\; \f{|\tilde{X}_0'(s')||\tilde{X}_0'(s)|}{|\tilde{X}_0(s')-\tilde{X}_0(s)|^2}
\cdot \f{|\tilde{X}_0'(s')|/|\tilde{X}_0(s')|}{|\tilde{X}_0'(s)|/|\tilde{X}_0(s)|}
 >0.
\end{align*}
In the last line, we used the assumptions on $\tilde{X}_0$.
This leads to a contradiction, so $T_*>0$.

Then for $t\in [0,T_*)$, we may use \eqref{eqn: eqn for ln |X|} and proceed as in the proof of Proposition \ref{prop: bound for the max |Phi|} to show that $\max_s|\tilde{X}(s,t)|$ does not increase in $t$, and $\min_s|\tilde{X}(s,t)|$ does not decrease in $t$.
That implies $|\tilde{X}(s,t)|\equiv R_X$ for all $s\in \BT$ and $t\in [0,T_*)$.
By the continuity of $\tilde{X}$, $|\tilde{X}(\cdot,T_*)|\equiv R_X$.
If $T_* < T$, viewing $T_*$ as the new initial time, we argue as before to find that $B(s,s',t)$ should stay non-negative beyond $T_*$ at least for some time.
This contradicts with the definition of $T_*$.
Therefore, $T_* = T$ and we complete the proof.
\end{proof}
\end{lem}

Since $|\tilde{X}(s)|\equiv R_X$, we can write $\tilde{X}(s)=R_X e^{i\theta(s)}$ with slight abuse of notation, where $\theta(s)\in\BR$.
Then \eqref{eq:5} implies that
\beq
\begin{split}
i\pa_t\theta(s) & =\f{\pa_t \tilde{X}(s)}{\tilde{X}(s)}
= \f{1}{8\pi}\pv \int_{\BT} \left[\frac{\tilde{X}'(s')^2}{(\tilde{X}(s')-\tilde{X}(s))\tilde{X}(s')}-
\frac{\overline{\tilde{X}'(s')}^2}{(\overline{\tilde{X}(s')}-\overline{\tilde{X}(s)})\overline{\tilde{X}(s')}}\right]ds'\\
& =\f{i}{8\pi}\pv\int_{\T}\theta'(s')^2\cot\frac{\theta(s')-\theta(s)}{2}\, ds'\\
& =\f{i}{4\pi}\pv\int_{\R}\frac{\theta'(s')^2}{\theta(s')-\theta(s)}\,ds'.
\end{split}
\label{eqn: theta equation}
\eeq
In the last equality, we extended $\th = \th(s')$ to the entire real line such that $\theta(s'+2\pi)=\theta(s')+2\pi$ for all $s'\in \BR$.
If $\th$ is strictly increasing and suitably smooth in $\BR$, the $\th$-equation is equivalent to the tangential Peskin problem in 2-D \cite[Section 2.2]{Tong2022GlobalST}.
Indeed, \eqref{eqn: theta equation} is exactly the tangential Peskin problem in the Lagrangian coordinate.
If we define $f = f(x,t)$ for $x\in \BR$ and $t>0$ such that $f(\theta(s,t),t)=\theta'(s,t)$, then $f(\cdot,t)$ is $2\pi$-periodic on $\BR$, and $f$ solves
\[
\pa_t f=\f{1}{4}\big(\CH f\cdot \pa_x f-f\cdot \pa_x\CH f\big),
\]
where
\[
\CH f(x) := \f{1}{\pi}\pv\int_{\BR}\f{f(y)}{x-y}\,dy =\f{1}{2\pi}\pv\int_{\BT}f(y)\cot\f{x-y}{2}\,dy.
\]
This is the tangential Peskin problem in the Eulerian coordinate.

Therefore, we can apply the results in \cite[Corollary 2.1]{Tong2022GlobalST} to obtain a global solution to \eqref{eqn: contour dynamic equation complex form} starting from initial data that has a circular shape but is not necessarily in equilibrium.

\begin{thm}
\label{thm: circular string}
Assume $X_0\in H^1(\BT)$, such that the curve $X_0(\BT)$ is a circle of radius $R_X$ centered at $x_0\in \BC$.
Suppose that, for a strictly increasing continuous function $\th_0:\BR\to \BR$ which satisfies $\th_0(s+2\pi) = \th_0(s) + 2\pi$ for all $s\in \BR$, it holds that $X_0(s) = x_0+ R_X e^{i\th_0(s)}$ in the notation of complex numbers.
Also assume that the inverse function of $\th_0$ on $[0,2\pi]$ is absolutely continuous.
Then \eqref{eqn: contour dynamic equation complex form} and \eqref{eqn: initial condition} admit a solution $X = X(s,t)$ in $\BT \times [0,+\infty)$ in the following sense:
\begin{enumerate}[label = (\roman*)]
\item $X(s,t)$ is a classic solution to \eqref{eqn: contour dynamic equation complex form} in $\BT\times (0,+\infty)$, i.e., \eqref{eqn: contour dynamic equation complex form} holds pointwise in $\BT\times (0,+\infty)$.

\item $X(\cdot,t)$ converges uniformly to $X_0(\cdot)$ as $t\to 0^+$.
\end{enumerate}

More precisely, the solution is constructed as follows.
Let $\th = \th(s,t)$ be a solution to (cf.\;\eqref{eqn: theta equation})
\[
\pa_t \th(s,t) = -\f{1}{4\pi} \pv \int_{\BR}\f{(\pa_{s'}\th(s',t))^2}{\th(s,t)-\th(s',t)}\,ds',\quad \th(s,0) = \th_0(s)
\]
in $\BR\times [0,+\infty)$, which is defined in Corollary 2.1 of \cite{Tong2022GlobalST} (with $X$ and $X_0$ there replaced by $\th$ and $\th_0$, respectively).
Let (cf.\;\eqref{eq:3})
\beq
v(t):= -\f{R_X}{8\pi}
\int_\BT e^{i\th(s',t)} \big(\pa_{s'}\th(s',t)\big)^2 \,ds'.
\label{eqn: velocity of the center}
\eeq
Then
\beq
X(s,t) = x(t) + R_X e^{i\th(s,t)} \mbox{ with } x(t): = x_0+\int_0^t v(\tau) \,d\tau
\label{eqn: def the solution X}
\eeq
gives the desired solution to \eqref{eqn: contour dynamic equation complex form} and \eqref{eqn: initial condition}.
It has the following properties:
\begin{enumerate}[label = (\arabic*)]
\item For any $t\geq 0$, the curve $X(\BT,t)$ is a circle of radius $R_X$ centered at $x(t)$.
\item 
$X(s,t)$ is smooth in $\BT\times (0,+\infty)$.
For any $\al\in (0,\f12)$,
$X(s,t) \in C^\al(\BT\times [0,+\infty))$.

\item

    $\|X(\cdot,t)\|_{\dot{H}^1(\BT)} = R_X \|\th(\cdot,t)\|_{\dot{H}^1(\BT)}$ is non-increasing in $t\in [0,+\infty)$.

\item For any $t> 0$, $|X'(\cdot,t)|$ has a positive lower bound.
As a result, $X(\cdot,t)$ satisfies the well-stretched condition, i.e., for some $\lambda = \lambda(t)>0$
\[
|X(s_1,t)-X(s_2,t)|\geq \lambda(t)|s_1-s_2|_\BT\quad \forall\, s_1,s_2 \in \BT.
\]
In fact, we may choose $\lam(t)$ to be strictly increasing in $t$.

\item
There exist $x_\infty\in\BC$ and $\xi_\infty\in\BT$ such that, as $t\to +\infty$, $X(\cdot,t)$ converges uniformly to $X_\infty(s): = x_\infty + R_X e^{i(s+\xi_\infty)}$ with an exponential rate.
The exponential convergence also holds in $H^k(\BT)$-norms with arbitrary $k\in \BN$.
\end{enumerate}

\begin{proof}
According to Corollary 2.1 of \cite{Tong2022GlobalST}, $\th(s,t)$ constructed here is smooth in $\BR\times (0,+\infty)$.
Besides, $\|\th(\cdot,t)\|_{\dot{H}^1(\BT)}$ is uniformly bounded for all $t$, and $\|\th'(\cdot,t)-1\|_{L^2(\BT)}$ decays to $0$ exponentially as $t\to +\infty$.
Hence, by \eqref{eqn: velocity of the center}, $v(t)$ is smooth in $(0,+\infty)$, and
\[
|v(t)|= \f{R_X}{8\pi} \left|\int_\BT e^{i\th(s',t)} \pa_{s'}\th(s',t) \big(\pa_{s'}\th(s',t)-1\big)\, ds'\right|
\leq \f{R_X}{8\pi}\|\th\|_{\dot{H}^1(\BT)}\|\th'-1\|_{L^2(\BT)}.
\]
This implies that $v(t)$ is uniformly bounded and decays exponentially as $t\to +\infty$, so $x(t)$ defined by \eqref{eqn: def the solution X} converges to some $x_\infty \in \BC$ exponentially.

These facts together with the properties of $\th(s,t)$ established in \cite[Corollary 2.1]{Tong2022GlobalST} imply the desired claims.
\end{proof}
\end{thm}

\section{The Curvature $\ka(s)$}
\label{sec: curvature}
Let $\ka(s,t)$ denote the curvature of the curve $X(\BT,t)$ at the point $X(s,t)$.
It is given by
\beq
\kappa(s)=\f{\Im[\overline{X'(s)}X''(s)]}{|X'(s)|^3}
=\f{\Im[X''(s)/X'(s)]}{|X'(s)|}.
\label{eqn: def of curvature s variable}
\eeq
The sign convention of the curvature is that, if $X(\BT)$ is a circle, $\ka(s)$ is positive.
By the assumptions \ref{assumption: geometry}-\ref{assumption: non-degeneracy} on $X$, $\ka(s,t)\in C^1(\BT\times [0,T])$.
In this section, we want to establish an extremum principle and a decay estimate for $\ka(s,t)$ under the condition that $\Phi_*(0) < \pi/4$.
The main result of this section is Proposition \ref{prop: max principle and decay estimate for curvature}.

We start from deriving the equation for $\ka(s)$.

\begin{lem}
\label{lem: equation for the curvature}
The curvature $\ka(s,t)$ satisfies
\beq
\partial_t\kappa(s) = \frac{3}{2\pi}\pv\int_{\T}
\frac{|X'(s')|^2 \cos 2\Phi(s',s) }{|X(s')-X(s)|^2} \left[\Im \f{I(s,s')}{|X'(s)|} -\f12 \kappa(s)\right] ds'.
\label{eqn: eqn for curvature}
\eeq
\begin{proof}
By definition,
\[
|X'(s)|\partial_t\kappa(s)
=\partial_t\mathrm{Im}\f{X''(s)}{X'(s)}-\kappa(s)\partial_t|X'(s)|.
\]
We differentiate \eqref{eqn: equation for alpha differentiable} to obtain that
\begin{align*}
&\;\partial_t\mathrm{Im}\f{X''(s)}{X'(s)}
= \partial_s\Im\frac{\partial_tX'(s)}{X'(s)}\\
=&\;
\frac{1}{2\pi}\pv\int_{\T}\bigg\{\Re\left[\frac{3X'(s')^2X'(s)^2}{(X(s')-X(s))^4}\right]
\Im\left[\frac{X(s')-X(s)}{X'(s)}\right]\\
&\;\qquad +\Re\left[\frac{X'(s')^2X'(s)}{(X(s')-X(s))^3}\cdot \f{X''(s)}{X'(s)}\right]
\Im\left[\frac{X(s')-X(s)}{X'(s)}\right]\\
&\;\qquad -\Re\left[\frac{X'(s')^2X'(s)}{(X(s')-X(s))^3}\right]
\Im\left[\frac{X(s')-X(s)}{X'(s)}\cdot \f{X''(s)}{X'(s)}\right] -\frac{1}{2}\Im\left[\f{X'(s')X''(s)}{(X(s')-X(s))^2}\right]\bigg\}\, ds'.
\end{align*}
Here we used Lemma \ref{lem: some singular integrals along the string curve}.
Since
\beq
\f{X''(s)}{X'(s)}=\Re\f{X''(s)}{X'(s)}+i\kappa(s)|X'(s)|,
\label{eqn: a relation for curvature}
\eeq
we find that
\begin{align*}
&\;\partial_t\mathrm{Im}\f{X''(s)}{X'(s)}  \\
= &\;
\frac{1}{2\pi}\pv\int_{\T}\bigg\{\Re\left[\frac{3X'(s')^2X'(s)^2}{(X(s')-X(s))^4}\right]
\Im\left[\frac{X(s')-X(s)}{X'(s)}\right]\\
&\;\qquad -\kappa(s)|X'(s)|\Im\left[\frac{X'(s')^2X'(s)}{(X(s')-X(s))^3}\right]
\Im\left[\frac{X(s')-X(s)}{X'(s)}\right]\\
&\;\qquad -\kappa(s)|X'(s)|\Re\left[\frac{X'(s')^2X'(s)}{(X(s')-X(s))^3}\right]
\Re\left[\frac{X(s')-X(s)}{X'(s)}\right]\\
&\;\qquad -\f12 \kappa(s)|X'(s)| \Re\left[\f{X'(s')X'(s)}{(X(s')-X(s))^2}\right]\bigg\}\,ds'.
\end{align*}
We used Lemma \ref{lem: some singular integrals along the string curve} again in the last term.
Combining this with \eqref{eqn: eqn for |X'| complex form}, we obtain that
\begin{align*}
&\;|X'(s)|\partial_t\kappa(s)=\partial_t\mathrm{Im}\f{X''(s)}{X'(s)}-\kappa(s)\partial_t|X'(s)|\\
=&\;
\frac{1}{2\pi}\pv\int_{\T}\bigg\{3\Re\left[ J(s',s)^2\right]
\Im\left[\frac{X(s')-X(s)}{X'(s)}\right]\\
&\;
\qquad -2\kappa(s)|X'(s)|\Im\left[\frac{X'(s')^2X'(s)}{(X(s')-X(s))^3}\right]
\Im\left[\frac{X(s')-X(s)}{X'(s)}\right]\\
&\;
\qquad -\kappa(s)|X'(s)|\Re\left[\frac{X'(s')^2X'(s)}{(X(s')-X(s))^3}\right]
\Re\left[\frac{X(s')-X(s)}{X'(s)}\right]\\
&\;
\qquad -\f12 \kappa(s)|X'(s)|\Re\left[\f{X'(s')^2}{(X(s')-X(s))^2}\right]\bigg\}\,ds'\\
=&\;
\frac{1}{2\pi}\pv\int_{\T}\bigg\{3\Re\left[ J(s',s)^2\right]
\Im\left[\frac{X(s')-X(s)}{X'(s)}\right]\\
&\;\qquad
-\frac{3}{2}\kappa(s)|X'(s)|\Im\left[\frac{X'(s')^2X'(s)}{(X(s')-X(s))^3}\right]
\Im\left[\frac{X(s')-X(s)}{X'(s)}\right]\\
&\;\qquad -\frac{3}{2}\kappa(s)|X'(s)|\Re\left[\frac{X'(s')^2X'(s)}{(X(s')-X(s))^3}\right]
\Re\left[\frac{X(s')-X(s)}{X'(s)}\right]\bigg\}ds'\\
=&\;
\frac{3}{2\pi}\pv\int_{\T}\bigg\{\Re\left[ J(s',s)^2\right]
\Im\left[\frac{X(s')-X(s)}{X'(s)}\right]\\
&\;\left. \qquad
-\frac{1}{2}\kappa(s)|X'(s)|\Re\left[\frac{X'(s')^2X'(s)}{(X(s')-X(s))^3}\cdot
\frac{\overline{X(s')-X(s)}}{\overline{X'(s)}}\right]\right\}\,ds'\\
=&\;
\frac{3}{2\pi}\pv\int_{\T} \Re\left[ J(s',s)^2\right]
\bigg\{ \Im\left[\frac{X(s')-X(s)}{X'(s)}\right]
 -\frac{1}{2}\kappa(s)|X'(s)|
\frac{|X(s')-X(s)|^2}{|X'(s)|^2}\bigg\}\,ds'\\
=&\;
\frac{3}{2\pi}\pv\int_{\T} \Re\left[ J(s',s)^2\right]
\frac{|X(s')-X(s)|^2}{|X'(s)|^2}\left\{\Im\,I(s,s') -\f12 \kappa(s)|X'(s)|\right\} ds'.
\end{align*}
Here we used $\Re(AB)=\Re\,A\Re\,B-\Im\,A\Im\,B$ and $\Re(A\bar{B})=\Re\,A\Re\,B+\Im\,A\Im\,B$ with
\[
A=\f{X'(s')^2X'(s)}{(X(s')-X(s))^3},\quad  B=\f{X(s')-X(s)}{X'(s)},\quad
AB=\f{X'(s')^2}{(X(s')-X(s))^2}.
\]
We also used the definitions of $I(s,s')$ and $J(s,s')$ (cf.\;Section \ref{sec: preliminary}), as well as
\[
\Im\left[\f{X(s')-X(s)}{X'(s)}\right]=\Im\left[\f{-1}{I(s,s')}\right]=\f{\Im\,I(s,s')}{|I(s,s')|^2}.
\]
This completes the proof of \eqref{eqn: eqn for curvature}.
\end{proof}
\end{lem}

We prove the following lemma to study the sign of the bracket on the right-hand side of \eqref{eqn: eqn for curvature}.

\begin{lem}
\label{lem: the curve is outside the osculating circle}
Suppose $\Phi_*\leq \pi/4$. 
Denote $\ka_+(t) : = \max_{s\in \BT}\ka(s,t)$ and
$\ka_-(t) : = \min_{s\in \BT}\ka(s,t)$.
Let $d(s)$ be defined in Proposition \ref{prop: geometric property of the curve}.
%
Then the following holds.
\begin{enumerate}[label = (\roman*)]
\item \label{claim: upper bound for generalized curvature}
For all distinct $s,s'\in \BT$,
$\Im\frac{I(s,s')}{|X'(s)|} \leq\frac{\kappa_+}{2\cos\Phi_*}$.

\item \label{claim: improved upper bound for generalized curvature}

    Let
    $C_0: = \f{2\cos(\Phi_*/2)}{\cos(3\Phi_*/2)}$.
If $\kappa_+\geq C_0 d(s)^{-1}$, then $\Im\frac{I(s,s')}{|X'(s)|} \leq\frac{\kappa_+}{2}$ for all $s'\in \T\setminus\{s\}$.

\item \label{claim: lower bound for generalized curvature}
For all distinct $s,s'\in \BT$, $\Im\frac{I(s,s')}{|X'(s)|}\geq \frac{\kappa_-}{2}$.

\end{enumerate}

\begin{proof}
Fix an arbitrary $s\in \BT$.
Denote
\beqo
Y(s') := \f{I(s,s')}{|X'(s)|}= \f{X'(s)/|X'(s)|}{X(s)-X(s')}\quad \forall\, s'\in \BT \setminus \{s\}.
\eeqo
We calculate that
\beq
|Y(s')| = \f{1}{|X(s)-X(s')|},\quad
Y'(s') = \f{J(s,s')}{|X'(s)|},\quad
|Y'(s')| = \f{|X'(s')|}{|X(s)-X(s')|^2},
\label{eqn: formula for Y'}
\eeq
so $\lim_{s'\to s} |Y(s')| = +\infty$ and $\arg Y'(s')= \Phi(s,s') \in [-\Phi_*,\Phi_*]$.
Hence, there exists a $C^2$-function $f:\BR\to \BR$, such that $\|f'\|_{L^\infty}\leq \tan \Phi_*\leq 1$, and
\[
\Im \, Y(s') = f\big(\Re\, Y(s')\big)\quad \forall\, s'\in \BT\setminus\{s\}.
\]
In fact, the function $\Re \,Y$ maps $\BT\setminus\{s\}$ onto $\BR$, so $Y(\BT\setminus\{s\})$ is exactly the graph of $f$.
As a result, it suffices to study the upper and lower bounds for $f$.

With $ \eta=\Re\, Y(s')$, we write
\[
X(s')=X(s)-\frac{X'(s)}{|X'(s)|Y(s')} =X(s)-\frac{X'(s)}{|X'(s)|(\eta+if(\eta))}=:g(\eta).
\]
Note that here $s$ is treated as a given constant.
Since the curvature formula \eqref{eqn: def of curvature s variable}
applies to any non-degenerate parameterization of the curve, we have
$\kappa(s')=\Im[g''(\eta)/g'(\eta)]/|g'(\eta)|$ under the change of variable $ \eta=\Re\, Y(s')$.
Then we calculate that
\begin{align*}
&g'(\eta)=\frac{X'(s)(1+if'(\eta))}{|X'(s)|(\eta+if(\eta))^2},\quad |g'(\eta)|=\frac{|1+if'(\eta)|}{\eta^2+f(\eta)^2},\\ &\frac{g''(\eta)}{g'(\eta)}=\frac{if''(\eta)}{1+if'(\eta)}-2\cdot \frac{1+if'(\eta)}{\eta+if(\eta)},\\
&\Im\frac{g''(\eta)}{g'(\eta)}=\frac{f''(\eta)}{1+|f'(\eta)|^2}+2\cdot \frac{f(\eta)-\eta f'(\eta)}{\eta^2+f(\eta)^2}.
\end{align*}
Hence, still with $ \eta=\Re\, Y(s')$, we obtain that
\beq
\kappa(s')
=
\f{1}{|g'(\eta)|}\Im\frac{g''(\eta)}{g'(\eta)}
=
\frac{(\eta^2+f(\eta)^2)f''(\eta)}{(1+|f'(\eta)|^2)^{3/2}}+2\frac{f(\eta)-\eta f'(\eta)}{(1+|f'(\eta)|^2)^{1/2}}
\in[\ka_-, \ka_+].
\label{eqn: curvature in terms of f}
\eeq

We then proceed in four steps.

\setcounter{step}{0}
\begin{step}
We first show that, for all $s'\in \BT\setminus\{s\}$,
\beq
|\Im\, Y(s')|\leq \tan \Phi_*|\Re\, Y(s')| + (d(s) \cos \Phi_* )^{-1}.
\label{eqn: imaginary part should not be large compared with real part}
\eeq 

Assume that, at some $s_*\in \BT\setminus\{s\}$, $|X(s) -X(s_*)| = \sup_{s'}|X(s)-X(s')| = d(s)$.
Then
\[
|Y(s_*)| = \f{1}{|X(s)-X(s_*)|} = d(s)^{-1}. 
\]
Since $\arg Y'(s') \in [-\Phi_*,\Phi_*]$, for all $s'\in \BT\setminus\{s\}$,
\beq
\begin{split}
|\Im\, Y(s')|
\leq &\; |\Im\, Y(s_*)| + \tan \Phi_*  \big|\Re\, Y(s')- \Re\, Y(s_*)\big|\\
\leq &\; \tan \Phi_* |\Re\, Y(s')| + |\Im\, Y(s_*)| + \tan \Phi_* |\Re\, Y(s_*)|\\
\leq &\; \tan \Phi_*|\Re\, Y(s')| + (d(s) \cos \Phi_* )^{-1}.
\end{split}
\label{eqn: the string does not appear in certain area}
\eeq
We applied the Cauchy-Schwarz inequality in the last line.

\end{step}

\begin{step}
\label{step: far field}
We then show that
$\lim_{\eta\to\infty}f(\eta)=\f{\ka(s)}{2}$. 
By definition,
\begin{align*}
\Im\,Y(s')=&\; \Im\f{X'(s)/|X'(s)|}{X(s)-X(s')}=\Im\f{\overline{X'(s)}(X(s')-X(s))}{|X'(s)||X(s)-X(s')|^2}\\
=&\; \Im\f{\overline{X'(s)}(X(s')-X(s)-X'(s)(s'-s))}{|X'(s)||X(s)-X(s')|^2}.
\end{align*}
Hence,
\beq
\begin{split}
\lim_{s'\to s}\Im\,Y(s')
= &\;
\lim_{s'\to s}\Im\f{\overline{X'(s)}(X(s')-X(s)-X'(s)(s'-s))/|s'-s|^2}{|X'(s)||X(s)-X(s')|^2/|s'-s|^2}\\
= &\;
\Im\f{\overline{X'(s)}X''(s)/2}{|X'(s)||X'(s)|^2}=\f{\ka(s)}{2}.
\end{split}
\label{eqn: limit of Im Y}
\eeq
We used \eqref{eqn: def of curvature s variable} in the last equality.
This proves the desired claim.
\end{step}

\begin{step}
It is clear that $\ka_+>0$, since $\int_\BT \ka(s)|X'(s)|\,ds = \int_\BT \Im [\pa_{s}\ln X'(s)] \,ds = 2\pi$ by \eqref{eqn: def of curvature s variable}.

If $ \sup_{x\in\R}f(x)\leq \frac{\kappa_+}{2}$, then
\beq
\Im\frac{I(s,s')}{|X'(s)|} = \Im\, Y(s')=
f(\Re\,Y(s')) \leq\frac{\kappa_+}{2}\leq\frac{\kappa_+}{2\cos\Phi_*},
\label{eqn: link generalized curvature with f}
\eeq
which gives \ref{claim: upper bound for generalized curvature} and \ref{claim: improved upper bound for generalized     curvature}.

Next we assume $ \sup_{x\in\R}f(x)> \frac{\kappa_+}{2}$.
Thanks to the previous step,
it must be attained at some $x_\dag\in \BR$, where
\beq
f(x_\dag)>\f{\ka_+}{2},\quad f'(x_\dag) = 0.
\label{eqn: basic estimates for x f dag}
\eeq
Taking $\eta=x_\dag$ in \eqref{eqn: curvature in terms of f}, we find that
\[
f''(x_\dag) \leq
\f{\kappa_+-2 f(x_\dag)}{x_\dag^2+f(x_\dag)^2} <0.
\]
If $f''(x)<0$ for all $x\geq x_\dag$, then $f'$ is strictly decreasing on $[x_\dag,+\infty)$, so $f'(x)\leq f'(x_\dag+1)<0$ for all $x\geq x_\dag+1$, and thus $f(x)\leq f(x_\dag+1)+f'(x_\dag+1)(x-x_\dag-1)$ for $x\geq x_\dag+1$, which contradicts with $\lim_{\eta\to\infty}f(\eta)=\f{\ka(s)}{2}$.
This implies that $\{x\geq x_\dag:\,f''(x)\geq 0\}$ must be nonempty.

Let $\tilde{x}:=\inf\{x\geq x_\dag:\, f''(x)\geq0\}$.
Then $x_\dag<\tilde{x}<+\infty$, $f''(\tilde{x})=0$, and $f''(x)<0$ for $x\in [x_\dag,\tilde{x})$.
Taking $\eta=\tilde{x}$ in \eqref{eqn: curvature in terms of f} yields
\beq
f(\tilde{x})-\tilde{x} f'(\tilde{x})\le \frac{\kappa_+}{2}(1+|f'(\tilde{x})|^2)^{1/2}.
\label{f1}
\eeq
Since $f''(x)<0$ for $x\in [x_\dag,\tilde{x})$, we have $0=f'(x_\dag) \geq f'(x)>f'(\tilde{x})$ for $x\in [x_\dag,\tilde{x})$, and
\[
f(\tilde{x})-\tilde{x} f'(\tilde{x})-(f(x_\dag)-x_\dag f'(\tilde{x}))=\int_{x_\dag}^{\tilde{x}} [f'(x)-f'(\tilde{x})]\, dx>0.
\]
This together with \eqref{f1} implies
\begin{align*}
f(x_\dag)-x_\dag f'(\tilde{x})<f(\tilde{x})-\tilde{x} f'(\tilde{x})\le \frac{\kappa_+}{2}(1+|f'(\tilde{x})|^2)^{1/2}.
\end{align*}
Now using $f'(\tilde{x})<0$ and \eqref{eqn: basic estimates for x f dag},
\begin{align*}
x_\dag\le &\; \frac{({\kappa_+}/{2})(1+|f'(\tilde{x})|^2)^{1/2}-f(x_\dag)}{-f'(\tilde{x})}\\
\le &\; \frac{\kappa_+}{2}\frac{(1+|f'(\tilde{x})|^2)^{1/2}-1}{-f'(\tilde{x})}
=\frac{\kappa_+}{2}\frac{|f'(\tilde{x})|}{(1+|f'(\tilde{x})|^2)^{1/2}+1}.
\end{align*}
Since $ |f'(\tilde{x})|\leq \|f'\|_{L^{\infty}}\leq \tan \Phi_*$,
\beq
x_\dag\leq\frac{\kappa_+}{2}\frac{\tan \Phi_*}{(1+|\tan \Phi_*|^2)^{1/2}+1}
=\frac{\kappa_+}{2}\tan \left(\f{\Phi_*}{2} \right). 
\label{eqn: bound for x_dag kappa plus case}
\eeq
One can analogously prove $x_\dag\geq-\frac{\kappa_+}{2}\tan (\Phi_*/2)$ by studying the point $\tilde{x}':= \sup\{x\leq x_\dag:\,f''(x)\geq 0\}$.
Therefore, $|x_\dag|\leq\frac{\kappa_+}{2}\tan (\Phi_*/2)$.

If $x_\dag \geq 0$, since $f''(x)<0$ for $x\in [x_\dag,\tilde{x})$, we find that
\beq
f(\tilde{x})-\tilde{x} f'(\tilde{x})-(f(x_\dag)-x_\dag f'(x_\dag))=\int_{x_\dag}^{\tilde{x}} [-x f''(x)]\, dx>0.
\label{eqn: certain quantity is monotone}
\eeq
Then by $f'(x_\dag)=0$, \eqref{f1} and
$|f'(\tilde{x})|\leq \tan \Phi_*$, we have
\[
f(x_\dag)<f(\tilde{x})-\tilde{x} f'(\tilde{x})\le \frac{\kappa_+}{2}(1+|f'(\tilde{x})|^2)^{1/2}\leq\frac{\kappa_+}{2}(1+|\tan \Phi_*|^2)^{1/2}=\frac{\kappa_+}{2\cos\Phi_*}.
\]
If $x_\dag\leq 0$, we have $f''(x)<0$ for $x\in (\tilde{x}',x_\dag]$, where $\tilde{x}'$ is defined above, so \eqref{eqn: certain quantity is monotone}
still holds with $\tilde{x}$ there replaced by $\tilde{x}'$.
Then arguing as above, we still obtain $f(x_\dag)<\frac{\kappa_+}{2\cos\Phi_*}$.
Hence, we have proved that $\sup_{x\in \BR}f(x)=f(x_\dag)\leq\f{\ka_+}{2\cos\Phi_*}$, provided $\sup_{x\in \BR}f(x) >\f{\ka_+}{2}$.
This implies that $\sup_{x\in \BR}f(x)\leq\f{\ka_+}{2\cos\Phi_*}$ is always true.
Proceeding as in \eqref{eqn: link generalized curvature with f}, we obtain \ref{claim: upper bound for generalized curvature}.

Now assume that $\kappa_+\geq \f{2\cos(\Phi_*/2)}{\cos(3\Phi_*/2)}d(s)^{-1}$.
If $\sup_{x\in \BR}f(x) >\f{\ka_+}{2}$, thanks to \eqref{eqn: imaginary part should not be large compared with real part} and \eqref{eqn: bound for x_dag kappa plus case}, 
\[
f(x_\dag)
\leq \tan \Phi_*|x_\dag| + (d(s)\cos \Phi_* )^{-1} 
\leq
\tan \Phi_* \cdot \frac{\ka_+}{2}
\tan \left(\f{\Phi_*}{2}\right)
+\f{\ka_+ \cos(3\Phi_*/2)}{2\cos(\Phi_*/2)\cos\Phi_*}
=\frac{\ka_+}{2},
\]
which contradicts with \eqref{eqn: basic estimates for x f dag}.
Hence, $\sup_{x\in \BR}f(x) \leq\f{\ka_+}{2}$, and thus $\Im\, Y(s')\leq \f{\ka_+}{2}$ for $s'\in\BT\setminus\{s\}$.
This proves \ref{claim: improved upper bound for generalized     curvature}.
\end{step}

\begin{step}
The proof of \ref{claim: lower bound for generalized curvature} follows a similar argument.
We only sketch it.

Suppose $\inf_{x\in\R}f(x)<\frac{\kappa_-}{2}$.
With abuse of the notation, the infimum must be attained at some $x_\dag\in \BR$, where
\[
f(x_\dag)<\f{\ka_-}{2},\quad f'(x_\dag) = 0.
\]
We use \eqref{eqn: curvature in terms of f} to show $f''(x_\dag) >0$.
Similar as before, there exists $\tilde{x}'< x_\dag < \tilde{x}$ such that $f''>0$ on $(\tilde{x}',\tilde{x})$ while
$f''(\tilde{x}) = f''(\tilde{x}')=0$.
Following the argument in the previous step, we can show that
\[
\frac{\kappa_-}{2}\frac{|f'(\tilde{x}')|}{(1+|f'(\tilde{x}')|^2)^{1/2}+1}
\leq
x_\dag
\leq -\frac{\kappa_-}{2}\frac{|f'(\tilde{x})|}{(1+|f'(\tilde{x})|^2)^{1/2}+1},
\]
where $f'(\tilde{x}')<0<f'(\tilde{x})$.
If $\ka_->0$, this cannot hold, so we must have $\inf_{x\in \BR} f(x)\geq \f{\ka_-}{2}$.
If $\ka_-\leq 0$, we obtain a naive bound 
\beq
|x_\dag|\leq \frac{|\kappa_-|}{2}.
\label{eqn: bound for x_dag kappa minus case}
\eeq

Now we need an improved version of the estimate \eqref{eqn: the string does not appear in certain area}.
Let $ \beta(s'): =\arg [X(s')-X(s)]$.
Then $\b(s')\in C(s,s+2\pi)$ and $ \b(s^+)=\arg X'(s)$; in addition, $\beta((s+2\pi)^-) - \b(s^+)=\pi$ since $X(\T)$ is parameterized in the counter-clockwise direction and $X\in C^1(\BT)$.
Hence, there exists $\tilde{s}\in (s,s+2\pi)$, such that
$\b(\tilde{s})-\b(s^+)=\pi/2$.
We thus find
\begin{align*}
Y(\tilde{s}) = &\; \f{X'(s)/|X'(s)|}{X(s)-X(\tilde{s})} = \f{e^{i\b(s^+)}}{-|X(s)-X(\tilde{s})|e^{i\b(\tilde{s})}}\\
= &\;
\f{1}{-|X(s)-X(\tilde{s})|e^{i\pi/2}}=\f{i}{|X(s)-X(\tilde{s})|}
=i|Y(\tilde{s})|.
\end{align*}
By the definition of $d(s)$, we have $|X(s)-X(\tilde{s})|\leq d(s)$, so $|Y(\tilde{s})|\geq  d(s)^{-1}$.
Hence, for any $s'\in \BT \setminus \{s\}$, 
\beq
\begin{split}
\Im\, Y(s')
\geq &\; \Im\, Y(\tilde{s}) - \tan \Phi_*  \big|\Re\, Y(s')- \Re\, Y(\tilde{s})\big|\\
= &\; - \tan \Phi_*  |\Re\, Y(s')| + |Y(\tilde{s})|\\
\geq &\; -\tan \Phi_* |\Re\, Y(s')| + d(s)^{-1}.
\end{split}
\label{eqn: lower bound for generalized curvature at arbitrary point}
\eeq
Therefore, by \eqref{eqn: bound for x_dag kappa minus case} and the fact $\Phi_*\leq \pi/4$,
\[
f(x_\dag)
\geq -\tan \Phi_* |x_\dag| + d(s)^{-1}
\geq -\tan \Phi_* \cdot \frac{|\kappa_-|}{2}
\geq \f{\ka_-}{2},
\]
which contradicts with the assumption.
This proves \ref{claim: lower bound for generalized curvature}.
\end{step}
\end{proof}
\end{lem}

Now we can establish an extremum principle for $\ka(s,t)$, as well as its upper and lower bounds.

\begin{prop}
\label{prop: max principle and decay estimate for curvature}
Suppose $\Phi_*(0)< \f{\pi}{4}$, where $\Phi_*(t)$ is defined in Proposition \ref{prop: bound for the max |Phi|}.
Denote $\ka_+(t) : = \max_{s\in \BT}\ka(s,t)$ and
$\ka_-(t) : = \min_{s\in \BT}\ka(s,t)$ as in Lemma \ref{lem: the curve is outside the osculating circle}.
Let $\ka_*(t) := \sup_{s\in \BT}|\ka(s,t)|$.
Then
\begin{enumerate}[label = (\roman*)]

  \item $\max \{\ka_+(t)R_X, 7+5\sqrt{2}\}$ is a non-increasing Lipschitz function.
For $t>0$,
\beqo
\kappa_+(t)R_X
\leq 1+C \exp\left[C \left(\int_0^t \cos 2\Phi_*(\tau)\,d\tau\right)^{-1} -ct\right],
\eeqo
where $C,c>0$ are universal constants.

  \item $\ka_-(t)R_X$ is a non-decreasing Lipschitz function.
For $t>0$,
\beqo
\kappa_-(t)R_X
\geq 1-C \exp\left[C \left(\int_0^t \cos 2\Phi_*(\tau)\,d\tau\right)^{-1} -ct\right],
\eeqo
where $C,c>0$ are universal constants.
\item $\max\{\ka_*(t)R_X,7+5\sqrt{2}\}$ is a non-increasing Lipschitz function.
    In addition, $\ka_*(t)R_X \geq 1$ for all $t>0$.
\end{enumerate}

\begin{rmk}
In particular, Proposition \ref{prop: bound for the max |Phi|} implies that the above upper and lower bounds are finite for any $t>0$, and they both converge to $1$ exponentially as $t\to +\infty$.
\end{rmk}

\begin{proof}
We shall still use the notations in Lemma \ref{lem: the curve is outside the osculating circle} as well as its proof.

\setcounter{step}{0}
\begin{step}
By Proposition \ref{prop: bound for the max |Phi|}, $\Phi_*(t) = \sup |\Phi(s_1,s_2,t)|<\f{\pi}{4}$ for all $t$.
Take an arbitrary $t$, and we first study $\ka_+(t)$.
Assume that $\ka_+(t)$ is attained at some $s\in \BT$, i.e., $\ka(s,t) = \ka_+(t)$.
Without loss of generality, we additionally assume
\beq
\ka(s)>  \f{\cos (\Phi_*/2)}{\cos(3\Phi_*/2)} \tan^2 \left(\f{\pi}{4}+ \f{\Phi_*}{2}\right) R_X^{-1} \geq R_X^{-1} >0.
\label{eqn: largeness assumption on curvature}
\eeq
Then Proposition \ref{prop: geometric property of the curve} implies
\beq
\ka(s)> C_0\tan \left(\f{\pi}{4}+ \f{\Phi_*}{2}\right) d(s)^{-1}\geq 2\tan \left(\f{\pi}{4}+ \f{\Phi_*}{2}\right) d(s)^{-1}\geq 2d(s)^{-1},
\label{eqn: implication of condition on kappa}
\eeq
where $C_0 = \f{2\cos(\Phi_*/2)}{\cos(3\Phi_*/2)}\geq 2$ was defined in Lemma \ref{lem: the curve is outside the osculating circle}.
By Lemma \ref{lem: the curve is outside the osculating circle}, $\Im\frac{I(s,s')}{|X'(s)|}\leq\frac{\kappa(s)}{2} $ for all $s'\in \T\setminus\{s\}$.
We also assume that $d(s) = |X(s)-X(s_*)|$ for some $s_*\in \BT$.

Define
\[
A : = \left\{s'\in \BT:\, \Im \f{I(s,s')}{|X'(s)|}\leq
\left(\f{\ka(s)}{2d(s)} \right)^{1/2} \right\}.
\]
If $|X(s)-X(s')|\geq [(2d(s))/\ka(s)]^{1/2}$, then
\[
\Im\f{I(s,s')}{|X'(s)|} \leq \f{1}{|X(s)-X(s')|} \leq \left( \f{\ka(s)}{2d(s)}\right)^{1/2},
\]
so $s'\in A$.
In particular, $s_*\in A$ since $d(s) = |X(s)-X(s_*)|>2\ka(s)^{-1}$, which means $A$ is non-empty.
Then by Lemma \ref{lem: equation for the curvature} and Lemma \ref{lem: the curve is outside the osculating circle},
\[
\partial_t\kappa(s)
\leq -\frac{3 \ka(s)}{4\pi} \cos 2\Phi_*
\left[1- \left(\f{2}{d(s)\ka(s)}\right)^{1/2} \right]
\int_A
\frac{|X'(s')|^2 }{|X(s')-X(s)|^2} \, ds'. 
\]

By the Cauchy-Schwarz inequality,
\begin{align*}
\int_A
\frac{|X'(s')|^2}{|X(s')-X(s)|^2} \, ds'
\geq &\;
\left(\int_{A} \frac{|X'(s')|}{|X(s')-X(s)|} \, ds' \right)^2
\left(\int_A\,ds'\right)^{-1}
\geq
\frac{1}{2\pi}
\left(\int_{X(A)} \frac{d\CH^1(z)}{|z-X(s)|} \right)^2,
\end{align*}
where $d\CH^1(z)$ denotes the 1-dimensional Hausdorff measure.
The definition of $A$ and $d(s)$ 
allows us to derive a naive bound by the co-area formula
\[
\int_{X(A)} \frac{d\CH^1(z)}{|z-X(s)|}
\geq
\int_{ (2d(s)/\ka(s))^{1/2}}^{d(s)} \f{1}{r}\,dr
= \f12 \ln\big(d(s)\ka(s)/2\big).
\]
If $d(s)\ka(s)\leq 3$, we need an improved estimate.
Using \eqref{eqn: formula for Y'} and the notations in Lemma \ref{lem: the curve is outside the osculating circle},
\beq
\int_{A} \frac{|X'(s')|}{|X(s')-X(s)|} \, ds'
= \int_{A} \f{|Y'(s')|}{|Y(s')|}\,ds'
\geq 
\int_{\{f(x)\leq [\ka(s)/(2d(s))]^{1/2}\}} \f{dx}{(x^2+f(x)^2)^{1/2}}.
\label{eqn: estimate the arclength outside a circle}
\eeq
With $x_*:= \Re \, Y(s_*)$,
\[
x_*^2 + f(x_*)^2 = |Y(s_*)|^2 = \f{1}{|X(s)-X(s_*)|^2} = d(s)^{-2},
\]
so $f(x_*)\leq d(s)^{-1} < [\ka(s)/(2d(s))]^{1/2}$.
Since $|f'|\leq \tan \Phi_*$, $f(x)\leq [\ka(s)/(2d(s))]^{1/2}$ for all $x$ such that 
\[
|x-x_*|\leq r:= \min\left\{\f{[\ka(s)/(2d(s))]^{1/2} - d(s)^{-1}}{\tan \Phi_*},\,d(s)^{-1}\right\}.
\]
On the other hand, \eqref{eqn: implication of condition on kappa} implies 
$\tan \left(\f{\pi}{4}+ \f{\Phi_*}{2}\right)\leq d(s)\ka(s)/2$,
so $\tan \Phi_* \leq C(d(s)\ka(s)-2)$ for some universal $C$.
Combining these estimates, we find that
$r \in [cd(s)^{-1},  d(s)^{-1}]$ for some universal $c\in (0,1)$.
Hence, on $[x_*-r ,x_*+r]$, $x^2+f(x)^2 \leq C d(s)^{-2}$ for some universal $C$.
Then \eqref{eqn: estimate the arclength outside a circle} gives that
\beqo
\int_{A} \frac{|X'(s')|}{|X(s')-X(s)|} \, ds'
\geq
\int_{|x-x_*|\leq r} \f{dx}{(x^2+f(x)^2)^{1/2}}
\geq C,
\eeqo
where $C$ is a universal constant.
As a result, in all cases,
\[
\int_A
\frac{|X'(s')|^2}{|X(s')-X(s)|^2} \, ds'
\geq C\left[1+ \ln\big(d(s)\ka(s)/2\big) \right]^2.
\]
Therefore, at the point $s$,
\[
\partial_t\kappa(s)
\leq -C \ka(s) \cos 2\Phi_*
\left[1- \left(\f{2}{d(s)\ka(s)}\right)^{1/2} \right]
\left[1+ \ln\big(d(s)\ka(s)/2\big) \right]^2. 
\]
By Proposition \ref{prop: geometric property of the curve} and \eqref{eqn: largeness assumption on curvature},
\[
\partial_t\kappa(s,t)
\leq -C\ka(s) \cos 2\Phi_*
\left[ 1 - \tan \left(\f{\pi}{4}+ \f{\Phi_*}{2}\right)^{1/2} \left( \ka(s)R_X \right)^{-1/2} \right] 
\big[1+ \ln(\ka(s)R_X)
\big]^2,
\]
where $C>0$ is a universal constant.


Since $\ka(s,t)$ is $C^1$ in the space-time, $\ka_+(t)$ is a Lipschitz function.
Arguing as in Proposition \ref{prop: bound for the max |Phi|},
%
we can show that, if 
\beq
\ka_+(t)R_X >  \f{\cos (\Phi_*/2)}{\cos(3\Phi_*/2)} \tan^2\left(\f{\pi}{4}+ \f{\Phi_*}{2}\right),
\label{eqn: lower threshold for ka_*}
\eeq
it holds
\beq
\begin{split}
\f{d}{dt}[\kappa_+(t)R_X]
\leq &\; -C\ka_+ R_X \cos 2\Phi_*
\big[1+ \ln(\ka_+ R_X)
\big]^2 \\
&\;\quad \cdot \left[ 1 - \tan \left(\f{\pi}{4}+ \f{\Phi_*}{2}\right)^{1/2} \left( \ka_+ R_X \right)^{-1/2} \right],
\end{split}
\label{eqn: ODE for kappa_*}
\eeq
for almost every $t$.

Denote
\beqo
h(\phi):= \f{\cos (\phi/2)}{\cos(3\phi/2)} \tan^2 \left(\f{\pi}{4}+ \f{\phi}{2}\right).
\eeqo
Since $\Phi_*(t)\leq \pi/4$ for all time, if $\ka_+(t)R_X > h(\pi/4)$, \eqref{eqn: lower threshold for ka_*} holds and \eqref{eqn: ODE for kappa_*} reduces to
\[
\f{d}{dt}[\kappa_+(t)R_X]
\leq -C\cos 2\Phi_*\cdot \ka_+ R_X
\big[\ln(\ka_+ R_X) \big]^2,
\]
which gives
\beqo
\kappa_+(t)R_X
\leq \max\left\{\exp\left[C \left(\int_0^t \cos 2\Phi_*(\tau)\,d\tau\right)^{-1}\right],\, h\left(\f{\pi}{4}\right)\right\}
\eeqo
for all time.
Here $C>0$ is a universal constant.
In view of Proposition \ref{prop: bound for the max |Phi|}, there exists $t_0>0$, such that
\[
\exp\left[C \left(\int_0^{t_0} \cos 2\Phi_*(\tau)\,d\tau\right)^{-1}\right] = h\left(\f{\pi}{4}\right).
\]
Here $t_0$ has universal upper and lower bounds.
Whenever $t\geq t_0$, $\ka_+(t)R_X \leq h(\pi/4)$, and $\cos 2\Phi_*$ admits a universal positive lower bound. 
Hence, for $t\geq t_0$, if \eqref{eqn: lower threshold for ka_*} holds, \eqref{eqn: ODE for kappa_*} implies
\beqo
\begin{split}
\f{d}{dt}\big[\ka_+(t)R_X\big]
\leq &\; -C
\left[\ka_+ R_X  - \tan \left(\f{\pi}{4}+ \f{\Phi_*}{2}\right)^{1/2} \left( \ka_+ R_X \right)^{1/2} \right]\\
\leq
&\; -C\left[ \ka_+(t) R_X - \tan \left(\f{\pi}{4}+ \f{\Phi_*(t)}{2}\right) \right].
\end{split}
\eeqo
By virtue of Proposition \ref{prop: bound for the max |Phi|}, as $t\to +\infty$, $h(\Phi_*)$ and $\tan  (\f{\pi}{4}+ \f{\Phi_*}{2})$ converge to $1$ exponentially with some explicit rates.
Assuming the constant $C>0$ in the above inequality to be smaller if necessary, we obtain that for some $c>0$,
\[
\ka_+(t) R_X \leq 1+ e^{-c(t-t_0)} h\left(\f{\pi}{4}\right),\quad \forall \, t\geq t_0.
\]
Note that with $c>0$ being small, the right-hand side of the above estimate can be made greater than the right-hand side of \eqref{eqn: lower threshold for ka_*}
for all $t\geq t_0$.
Combining all these estimates, the desired upper bound for $\ka_+(t)R_X$ follows.
Lastly, it is clear from the above proof that $\max \{\ka_+(t)R_X, h(\pi/4)\}$ is a non-increasing Lipschitz function in $t$, with $h(\pi/4) = 7+5\sqrt{2}$.

\end{step}

\begin{step}
Next we study $\ka_-(t)$.
Assume that $\ka_-(t)$ is achieved at some $\tilde{s}\in \BT$.
By Lemma \ref{lem: the curve is outside the osculating circle}, $\Im\frac{I(\tilde{s},s')}{|X'(\tilde{s})|}\geq\f{\kappa(\tilde{s})}{2} $ for all $s'\in \T\setminus\{\tilde{s}\}$.
By Lemma \ref{lem: equation for the curvature}, $\pa_t \ka(\tilde{s})\geq 0$.
Arguing as in Proposition \ref{prop: bound for the max |Phi|}, we know that $\ka_-(t)$ is a non-decreasing Lipschitz function.
It remains to prove the lower bound for $\ka_-(t)$.

First we consider the case $\ka(\tilde{s})< -4R_X^{-1}<0$.
Given this, by Proposition \ref{prop: geometric property of the curve},
\beq
d(\tilde{s})\geq 2R_X \tan\left(\f{\pi}{4}-\f{\Phi_*}{2}\right) \geq (2\sqrt{2}-2)R_X,
\label{eqn: lower bound for d at s tilde}
\eeq
so $d(\tilde{s})|\ka(\tilde{s})| \geq 8(\sqrt{2}-1) >3$.
Let
\[
\tilde{A} : = \left\{s'\in \BT:\, \Im \f{I(\tilde{s},s')}{|X'(\tilde{s})|}\geq
-\left(\f{|\ka(\tilde{s})|}{2d(\tilde{s})} \right)^{1/2} \right\}.
\]
If $|X(\tilde{s})-X(s')|\geq [(2d(\tilde{s}))/|\ka(\tilde{s})|]^{1/2}$, then $s'\in \tilde{A}$ because
\[
\Im\f{I(\tilde{s},s')}{|X'(\tilde{s})|} \geq -\f{1}{|X(\tilde{s})-X(s')|} \geq -\left( \f{|\ka(\tilde{s})|}{2d(\tilde{s})}\right)^{1/2}.
\]
 Using Lemma \ref{lem: equation for the curvature} and Lemma \ref{lem: the curve is outside the osculating circle}, we derive as before to see that
\beqo
\begin{split}
\partial_t\kappa(\tilde{s})
\geq &\; -\frac{3 \ka(\tilde{s})}{4\pi} \cos 2\Phi_*
\left[1- \left(\f{2}{d(\tilde{s})|\ka(\tilde{s})|}\right)^{1/2} \right]
\int_{\tilde{A}}
\frac{|X'(s')|^2 }{|X(s')-X(\tilde{s})|^2} \, ds'\\
\geq &\; C  |\ka(\tilde{s})| \cos 2\Phi_*
\left[\ln\big(d(\tilde{s})|\ka(\tilde{s})|/2\big)\right]^2\\
\geq &\; C  |\ka(\tilde{s})| \cos 2\Phi_*
\left[\ln\big(|\ka(\tilde{s})|R_X\big)\right]^2.
\end{split}
\eeqo
In the last line, we used \eqref{eqn: lower bound for d at s tilde} and the fact $|\ka(\tilde{s})|R_X >4$.

Since $\ka(s,t)$ is $C^1$ in the space-time, $\ka_-(t)$ is a Lipschitz function.
Hence, if $\ka_-(t)R_X <-4$, it holds for almost every $t$ that
\beqo
\f{d}{dt}[\kappa_-(t)R_X]
\geq -C\cos 2\Phi_*\cdot \ka_-(t)R_X
\left[\ln\big|\ka_-(t) R_X\big|\right]^2,
\eeqo
where $C>0$ is a universal constant.
This gives
\beq
\kappa_-(t)R_X
\geq \min\left\{-\exp\left[C \left(\int_0^t \cos 2\Phi_*(\tau)\,d\tau\right)^{-1}\right],\, -4\right\}
\label{eqn: lower bound for curvature small time}
\eeq
for all $t>0$.

Next we improve the lower bound \eqref{eqn: lower bound for curvature small time} for large $t$.
Arguing as in \eqref{eqn: lower bound for generalized curvature at arbitrary point}, we find that
\[
\begin{split}
\Im\, \f{I(\tilde{s},s')}{|X'(\tilde{s})|}
\geq &\; -\tan \Phi_* \left|\Re\, \f{I(\tilde{s},s')}{|X'(\tilde{s})|}\right|
 + d(\tilde{s})^{-1} \\
\geq &\; -\tan \Phi_* | X(\tilde{s})-X(s')|^{-1}
 + d(\tilde{s})^{-1}.
\end{split}
\]
If $| X(\tilde{s})-X(s')| \geq d(\tilde{s})/2$,
\[
\Im\, \f{I(\tilde{s},s')}{|X'(\tilde{s})|}
\geq d(\tilde{s})^{-1}\big(1 - 2 \tan \Phi_* \big).
\]
Using Lemma \ref{lem: equation for the curvature} and Lemma \ref{lem: the curve is outside the osculating circle} and deriving as before, we find
\begin{align*}
\partial_t\kappa(\tilde{s})
\geq &\;
\frac{3}{2\pi} \cos 2\Phi_* \int_{\{| X(\tilde{s})-X(s')| \geq d(\tilde{s})/2\}}
\frac{|X'(s')|^2}{|X(s')-X(\tilde{s})|^2} \left[\Im \f{I(\tilde{s},s')}{|X'(\tilde{s})|} -\f12 \kappa(\tilde{s})\right]_+ ds'\\
\geq &\;\frac{3}{2\pi} \cos 2\Phi_*
\left[d(\tilde{s})^{-1}\big(1 - 2 \tan \Phi_* \big) -\f12 \kappa(\tilde{s})\right]_+
\int_{\{| X(\tilde{s})-X(s')| \geq d(\tilde{s})/2\}}
\frac{|X'(s')|^2}{|X(s')-X(\tilde{s})|^2} \,ds'\\
\geq &\;C\cos 2\Phi_*
\left[d(\tilde{s})^{-1}\big(1 - 2 \tan \Phi_* \big) -\f12 \kappa(\tilde{s})\right]_+.
\end{align*}
Here the notation $(\cdot)_+$ means taking the positive part, i.e., $a_+ = \max\{a,0\}$ for any $a\in \BR$.

By virtue of Proposition \ref{prop: bound for the max |Phi|}, with abuse of the notation, there exists a universal $t_0>0$, such that for all $t\geq t_0$, it holds $1 - 2 \tan \Phi_* >0$ and $\cos 2\Phi_*\geq C$ with $C>0$ being universal.
We will assume $t\geq t_0$ for simplicity.
Using $d(\tilde{s})\leq d_*$ and the upper bound for $d_*$ in Proposition \ref{prop: geometric property of the curve},
\begin{align*}
\partial_t\kappa(\tilde{s})
\geq &\; C
\left[R_X^{-1} \tan \left(\f{\pi}{4}-\f{\Phi_*}{2}\right) \big(1 - 2 \tan \Phi_* \big) - \kappa(\tilde{s})\right]_+.
\end{align*}
Therefore, for almost all $t\geq t_0$,
\[
\f{d}{dt}[\kappa_-(t)R_X]
\geq C
\left[\tilde{h}\big(\Phi_*(t)\big) - \kappa_-(t)R_X\right]_+,
\]
where 
\[
\tilde{h}(\phi): =  \tan \left(\f{\pi}{4}-\f{\phi}{2}\right) \big(1- 2 \tan \phi \big).
\]
Proposition \ref{prop: bound for the max |Phi|} implies that, as $t\to +\infty$, $\tilde{h}(\Phi_*(t))$ converges to $1$ exponentially with some explicit rate.
Moreover, \eqref{eqn: lower bound for curvature small time} and Proposition \ref{prop: bound for the max |Phi|} imply that $\ka_-(t_0)R_X\geq -C$ for some universal $C>0$.
Therefore, for $t\geq t_0$,
\[
\kappa_-(t)R_X \geq 1- C e^{-c(t-t_0)},
\]
where $C,c>0$ are universal.
This combined with \eqref{eqn: lower bound for curvature small time} yields the desired lower bound.
\end{step}

\begin{step}
Lastly, that $\ka_*(t)R_X \geq 1$ follows from \cite{MR0107214,Pankrashkin2015AnIF}.
The monotonicity and Lipschitz regularity of $\max\{\ka_*(t)R_X,7+5\sqrt{2}\}$ follows from that of $\max\{\ka_+(t)R_X,7+5\sqrt{2}\}$ and $\ka_-(t)R_X$.
\end{step}
\end{proof}
\end{prop}

We conclude this section by briefly remarking on the case of general elasticity.
We follow the setup in Section \ref{sec: general elasticity}.
\begin{rmk}
\label{rmk: general elasticity curvature}
Assume that $X$ solves \eqref{eqn: contour dynamic equation complex form general tension} in Section \ref{sec: general elasticity} with $X(0,t) = X_0(s)$, and satisfies the assumptions \ref{assumption: geometry}-\ref{assumption: non-degeneracy} in Section \ref{sec: preliminary}.
Then the claims in Proposition \ref{prop: max principle and decay estimate for curvature} other than the quantitative bounds should still hold, i.e., when $\Phi_*(0)<\f{\pi}{4}$, $\ka(s,t)$ satisfies extreme principles as in Proposition \ref{prop: max principle and decay estimate for curvature}.

The justification is similar to that in Section \ref{sec: general elasticity}, which we will only sketch.
Fix $t$, and let $k_0$, $\xi =\xi(s)$ and $Y(\xi,t)$ be defined as in Section \ref{sec: general elasticity}.
By \eqref{eqn: time derivative of alpha change of variables}, 
\begin{align*}
\pa_t\Im \f{X''(s,t)}{X'(s,t)}
= &\; \pa_s\left[\Im\f{\pa_t X'(s,t)}{X'(s,t)}\right]
= k_0 \xi'(s) \cdot
\left.\pa_\xi \left[\Im \f{\pa_\tau Y'(\xi,\tau)}{Y'(\xi,\tau)}\right]\right|_{(\xi,\tau) = (\xi(s),t)}\\
= &\; k_0 \xi'(s) \cdot
\left.\pa_\tau \left[\Im \f{\pa_\xi Y'(\xi,\tau)}{Y'(\xi,\tau)}\right]\right|_{(\xi,\tau) = (\xi(s),t)}.
\end{align*}
By \eqref{eqn: def of curvature s variable}, this implies
\beq
\begin{split}
&\; |X'(s,t)|\partial_t\kappa(s,t)
+\kappa(s,t)\partial_t|X'(s,t)| \\
= &\; k_0 \xi'(s)
\big(|Y'(\xi,\tau)|\partial_\tau\kappa_Y(\xi,\tau)
+\kappa_Y(\xi,\tau)\partial_\tau|Y'(\xi,\tau)| \big)\big|_{(\xi,\tau) = (\xi(s),t)},
\end{split}
\label{eqn: identity between two representations of curvature}
\eeq
where $\ka_Y  = \ka_Y(\xi,\tau)$ is the curvature defined in terms of $Y(\xi,\tau)$ by \eqref{eqn: def of curvature s variable}.
Note that \eqref{eqn: def of curvature s variable} holds for any non-degenerate parameterization of the curve.
Using the fact that $\xi'(s)>0$ is real-valued,
\begin{align*}
\kappa(s)= &\;
\f{\Im[X''(s)/X'(s)]}{|X'(s)|}
=\f{1}{|Y'(\xi(s))|\xi'(s)}\Im\left[\f{\pa_s (Y'(\xi(s))\xi'(s))}{Y'(\xi(s))\xi'(s)}\right]\\
= &\; \f{1}{|Y'(\xi(s))|\xi'(s)}\Im\left[\f{Y''(\xi(s))\xi'(s)^2 + Y'(\xi(s))\xi''(s)}{Y'(\xi(s))\xi'(s)}\right]\\
= &\;
\f{1}{|Y'(\xi(s))|}
\Im\left[\f{Y''(\xi(s))}{Y'(\xi(s))}\right] = \ka_Y(\xi(s)).
\end{align*}
This should be expected since the curvature does not depend on the parameterization of the curve.
Moreover, by \eqref{eqn: relation between time derivatives of X and Y},
\begin{align*}
\partial_t|X'(s,t)|
= &\; \f{\Re[\overline{X'(s,t)}\pa_t X'(s,t)]}{|X'(s,t)|}\\
= &\; \f{\Re[\overline{Y'(\xi(s),t)} \xi'(s) \cdot k_0\pa_\tau Y'(\xi(s),t)\xi'(s)]}{|Y'(\xi(s),t)|\xi'(s)}
= k_0 \xi'(s)\partial_\tau|Y'(\xi(s),t)|.
\end{align*}
Plugging them into \eqref{eqn: identity between two representations of curvature} and using the identity $X'(s,t) = Y'(\xi(s),t)\xi'(s)\neq 0$, we obtain that
\beqo
\partial_t\kappa(s,t)
= k_0 \partial_\tau\kappa_Y (\xi(s),t).
\eeqo
Then we can justify the desired assertion as in the proof of Proposition \ref{prop: max principle and decay estimate for curvature}.

\end{rmk}

\section{Estimates for $|X'|$}
\label{sec: estimates for |X'|}
In this section we study $|X'|$, which encodes the stretching information of the elastic string.
Recall that $|X'(s,t)|$ solves \eqref{eqn: eqn for |X'|}, and by the assumptions \ref{assumption: geometry}-\ref{assumption: non-degeneracy}, it is $C^1$ in the space-time.

\subsection{$L^1$-, $L^2$-, and $L^\infty$-estimates}
Recall that $\CL$ and $\CE$ be defined in \eqref{eqn: length} and \eqref{eqn: energy}, respectively.

\begin{lem}
\label{lem: energy estimate}
$X$ satisfies the length estimate
\[
\f{d \CL(t)}{dt} 
= -\f{1}{4\pi} \int_{\BT}
\int_{\T} \frac{|X'(s)||X'(s')|^2}{|X(s')-X(s)|^2}\left[\cos \Phi(s',s)-\cos 2\Phi(s',s)\right] ds'\,ds,
\]
and the energy estimate
\beq
\begin{split}
&\;\f{d \CE(t)}{dt} \\ 
=
&\; -\f{1}{16\pi} \int_{\BT}
\int_{\T} |J(s,s')| \left(|X'(s)|-|X'(s')|\right)^2 \left[\cos \Phi(s',s)+\cos 2 \Phi(s',s)\right] ds'\,ds\\
&\; -\f{1}{16 \pi} \int_{\BT}
\int_{\T} |J(s,s')| \big(|X'(s)| + |X'(s')| \big)^2
\left[\cos \Phi(s',s)-\cos 2 \Phi(s',s)\right] ds'\,ds\\
=
&\; -\f{1}{8\pi} \int_{\BT}
\int_{\T} \frac{|X'(s)||X'(s')|}{|X(s')-X(s)|^2} \left(|X'(s)|-|X'(s')|\right)^2 \cos \Phi(s',s)\, ds'\,ds\\
&\; -\f{1}{4\pi} \int_{\BT}
\int_{\T} \frac{|X'(s)|^2|X'(s')|^2}{|X(s')-X(s)|^2}
\left[\cos \Phi(s',s)-\cos 2 \Phi(s',s)\right] ds'\,ds.
\end{split}
\label{eqn: energy estimate}
\eeq
As a result, if $\Phi_*(0)< \pi/4$, both $\CL(t)$ and $\CE(t)$ are non-increasing in time.

\begin{proof}
Integrating \eqref{eqn: eqn for |X'|} yields
\begin{align*}
&\;\f{d}{dt}\int_{\BT} |X'(s)|\,ds\\
= &\; \f{1}{4\pi} \int_{\BT}
\pv\int_{\T} \frac{|X'(s)||X'(s')|}{|X(s')-X(s)|^2}\left[|X'(s')|\cos 2\Phi(s',s)-|X'(s)|\cos \Phi(s',s)\right] ds'\,ds \\
= &\; -\f{1}{4\pi} \int_{\BT}
\int_{\T} \frac{|X'(s)||X'(s')|^2}{|X(s')-X(s)|^2}\left[\cos \Phi(s',s)-\cos 2\Phi(s',s)\right] ds'\,ds.
\end{align*}
In the last line, we exchanged the $s$- and $s'$-variables.
\eqref{eqn: energy estimate} can be derived analogously by multiplying \eqref{eqn: eqn for |X'|} by $|X'(s)|$ and taking integral on $\BT$.
\end{proof}
\end{lem}

\begin{lem}
\label{lem: max principle for |X'|}

If $\Phi_*(0)<\pi/4$, $|X'(s,t)|$ satisfies a maximum principle, i.e., $\max_s|X'(s,t)|$ is non-increasing.

\begin{proof}
Take an arbitrary $t$.
Suppose $\max_s |X'(s,t)|$ is attained at some $s\in \BT$.
Since $|\Phi_*(t)|<\f{\pi}{4}$ by Proposition \ref{prop: bound for the max |Phi|}, $\cos \Phi(s,s') \geq \cos 2\Phi(s,s') \geq 0$, so $|X'(s')|\cos 2\Phi(s',s)\leq |X'(s)|\cos \Phi(s',s)$.
By \eqref{eqn: eqn for |X'|}, $\pa_t |X'(s)|\leq 0$.
Following the argument in Proposition \ref{prop: bound for the max |Phi|}, we conclude that $|X'|$ enjoys the maximum principle.
\end{proof}
\end{lem}

\subsection{The lower bound and the well-stretched condition}
To derive a lower bound for $|X'|$, we first prove a few auxiliary lemmas.

\begin{lem}
\label{lem: change of variable to the arclength parameter}
For any $\d\in [0,\CL(t)/2]$ and continuous $f:[0,+\infty)\to \BR$,
it holds
\[
\int_{\{L(s,s')\geq \d\}}f\big(L(s,s')\big)|X'(s')|\, ds' = 2\int_{\d}^{\mathcal{L}(t)/2}f(x)\, dx.
\]

\begin{proof}
By the definition of $\CL(t)$, there exists $s_*\in (s,s+2\pi)$ such that
\[
\int_{s}^{s_*}|X'(s')|\,ds'
=\int_{s_*-2\pi}^{s}|X'(s')|\,ds'=\f12 \CL(t).
\]
Then
\[
L(s,s')=
\begin{cases}
\int_{s}^{s'}|X'(s'')|\,ds'' & \mbox{if } s'\in[s,s_*), \\
\int_{s'}^{s}|X'(s'')|\,ds'' & \mbox{if } s'\in(s_*-2\pi,s).
\end{cases}
\]
The desired claim then follows from a change of variable.
\end{proof}
\end{lem}

\begin{lem}
\label{lem: bound Phi in terms of kappa}
For any $s,s'\in \BT$, $ |\Phi(s,s')|\leq \kappa_*L(s,s')$.
\begin{proof}
Without loss of generality, we assume $s'\in (s,s+2\pi)$ and $L(s,s')=\int_s^{s'}|X'(s'')|\,ds''$.
Since $\Im\,\frac{X(\tau)}{X(s)-X(s')}\big|_{\tau=s}^{s'}=0$, by Rolle's theorem, there exists $s_*\in(s,s')$ such that $\Im \,\f{(s-s')X'(s_*)}{X(s)-X(s')}=0$, which implies $\f{X'(s_*)}{X(s)-X(s')}\in \BR\setminus \{0\}$.
Hence,
\[
\Phi(s,s')
=\arg \frac{X'(s)X'(s')}{(X(s)-X(s'))^2}
=\arg \frac{X'(s)X'(s')}{X'(s_*)^2} =\alpha(s)+\alpha(s')-2\alpha(s_*),
\]
where the equalities are understood in the modulo $2\pi$.
Since $\al' = \ka|X'|$ by \eqref{eqn: def of curvature s variable}, we have that
\begin{align*}
|\Phi(s,s')|
\leq &\; |\alpha(s)-\alpha(s_*)|_{\T}+|\alpha(s')-\alpha(s_*)|_{\T}\\
\leq &\; \int_s^{s_*} |\kappa(s'')X'(s'')|\,ds''+\int_{s_*}^{s'} |\kappa(s'')X'(s'')|\,ds''\\
\leq &\; \kappa_*\int_s^{s'}|X'(s'')|ds''
=\kappa_* L(s,s').
\end{align*}
Here $|\cdot|_\BT$ denotes the distance on $\BT$.
\end{proof}
\end{lem}

\begin{prop}
\label{prop: lower bound for |X'| and the well-stretched condition}
Suppose $\Phi_*(0)<\pi/4$.
For some universal constant $\b>0$, it holds
\beq
\min_s|X'(s,t)| \geq R_X \exp\left[-2\coth\left(\f{\b}{ 2} \int_0^t \cos 2\Phi_*(\tau)\,d\tau\right)\right]
\label{eqn: lower bound for |X'|}
\eeq
for all $t\in (0,T]$.
In particular, by Proposition \ref{prop: bound for the max |Phi|}, the lower bound is positive and strictly increasing for $t>0$.
As a result, for any $t>0$, $X(\cdot,t)$ satisfies the well-stretched condition: more precisely, for distinct $s,s'\in \BT$,
\[
|X(s,t)-X(s',t)|\geq |s-s'|_\BT \cdot C\min_s|X'(s,t)|,
\]
where $C$ is a universal constant.

\begin{proof}
The proof is similar to that of Lemma 4.1 in \cite{Tong2022GlobalST}.
We proceed in several steps.

\setcounter{step}{0}
\begin{step}
Fix $t$.
Assume that $\min_s |X'(s,t)|$ is attained at $s\in\BT$.
We only consider the case $|X'(s)|\leq R_X$.
With $\d\leq\CL(t)/2$ to be determined, we derive that
\begin{align*}
&\; \pv \int_\BT \frac{|X'(s')|(|X'(s')|- |X'(s)|)}{|X(s')-X(s)|^2} \, ds'
\geq \int_{\{L(s,s')\geq \d\}} \frac{|X'(s')|(|X'(s')|- |X'(s)|)}{L(s,s')^2} \, ds'\\
\geq &\; \left(\int_{\{L(s,s')\geq \d\}} \frac{|X'(s')|^2}{L(s,s')^2} \, ds'\right)
\left(\f{1}{2\pi}\int_{\{L(s,s') \geq \d\}} 1\, ds'\right)
-|X'(s)| \int_{\{L(s,s')\geq \d\}} \frac{|X'(s')|}{L(s,s')^2} \, ds'\\
\geq &\; \f1{2\pi} \left(\int_{\{L(s,s')\geq \d\}}
\f{ |X'(s')|} {L(s,s')}\,ds' \right)^2
-|X'(s)| \int_{\{L(s,s')\geq \d\}} \frac{|X'(s')|}{L(s,s')^2} \, ds'.
\end{align*}
In the last step, we applied the Cauchy-Schwarz inequality.
By Lemma \ref{lem: change of variable to the arclength parameter},
\[
\pv \int_\BT \frac{|X'(s')|(|X'(s')|- |X'(s)|)}{|X(s')-X(s)|^2} \, ds'
\geq \f1{2\pi} \left( 2\ln \f{\CL(t)/2}{\d}
\right)^2
-2|X'(s)| \d^{-1}.
\]
Taking $\d = \pi|X'(s)|\leq \CL(t)/2$ and applying the isoperimetric inequality $\CL(t) \geq 2\pi R_X$, we obtain
\beq
\pv \int_\BT \frac{|X'(s')|(|X'(s')|-|X'(s)|)}{|X(s')-X(s)|^2} \, ds'
\geq \f{2}{\pi} \left( \ln \f{R_X}{|X'(s)|}\right)^2 -\f{2}{ \pi}.
\label{eqn: restoring term in |X'| fractional Laplacian}
\eeq
On the other hand, by Proposition \ref{prop: geometric property of the curve}, Lemma \ref{lem: change of variable to the arclength parameter}, and Lemma \ref{lem: bound Phi in terms of kappa},
\beq
\begin{split}
&\; \int_\BT \frac{|X'(s')|}{|X(s')-X(s)|^2}
\left[\cos \Phi(s',s)- \cos 2\Phi(s',s)\right] ds'\\
\leq &\; C\int_\BT \frac{|X'(s')| \sin^2\Phi(s',s)}{L(s,s')^2}
\, ds'\\
\leq &\; C\int_{\T} \f{|X'(s')|(\min\big\{\ka_* L(s,s'),\,\sin \Phi_*\big\})^2}{L(s,s')^2} \, ds'\\
= &\; 2C\int_{0}^{\mathcal{L}(t)/2} \f{(\min\big\{\ka_* x,\,\sin \Phi_*\big\})^2}{x^2} \, dx
\leq  C\ka_* \sin \Phi_*,
\end{split}
\label{eqn: positive term due to nonzero Phi}
\eeq
where $C$ is universal.
Combining \eqref{eqn: eqn for |X'|}, \eqref{eqn: restoring term in |X'| fractional Laplacian} and \eqref{eqn: positive term due to nonzero Phi}, at $s\in \BT$,
\beqo
\begin{split}
&\;\pa_t\left( \ln\f{ |X'(s)|}{R_X}\right)\\
= &\; \f{1}{4\pi}
\pv \int_\BT \frac{|X'(s')|(|X'(s')|- |X'(s)|)}{|X(s')-X(s)|^2} \cos 2\Phi(s',s) \, ds'\\
&\;- |X'(s)| \cdot \f{1}{4\pi}
\pv \int_\BT \frac{|X'(s')|}{|X(s')-X(s)|^2}
\left[\cos \Phi(s',s)- \cos 2\Phi(s',s)\right] ds'\\
\geq &\;\f1{2\pi^2}\cos 2\Phi_* \left(\left|\ln \f{|X'(s)|}{ R_X}\right|^2- 1\right) - C |X'(s)| \ka_* \sin \Phi_* .
\end{split}
\eeqo

Since $X'(s,t)$ is $C^1$ in the space-time, $\min_s |X'(s,t)|$ is a Lipschitz function in $t$.
Let
\[
\Lam(t) := \f{ \min_s|X'(s,t)|}{R_X}.
\]
Recall that Proposition \ref{prop: max principle and decay estimate for curvature} implies
\beq
\ka_*(t)R_X \leq 1+C_* \exp\left[C_* \left(\int_0^t \cos 2\Phi_*(\tau)\,d\tau\right)^{-1} -c_*t\right] =: K(t),
\label{eqn: def of K}
\eeq
where $C_*,c_*>0$ are universal constants.
Then we argue as in Proposition \ref{prop: bound for the max |Phi|} to find that  
\beq
\pa_t \ln \Lam(t)
\geq \f1{2\pi^2}\cos 2\Phi_* \left[
\big|\ln \Lam(t)\big|^2- 1
- C_\dag\Lam(t) K(t)\cdot  \f{\sin \Phi_*}{\cos 2\Phi_*}\right] =: F(\ln \Lam(t),t)
\label{eqn: ODE for lower bound for |X'|}
\eeq
for almost every $t$.
Here $C_\dag$ is a universal constant, and
\[
F(x,t) = \f1{2\pi^2}\cos 2\Phi_*(t) \left[
x^2- 1
- C_\dag e^x K(t)\cdot  \f{\sin \Phi_*(t)}{\cos 2\Phi_*(t)}\right].
\]
\end{step}

\begin{step}
We claim that, if we define for $t>0$
\beqo
\Lam_*(t)
:= \exp\left[-2\coth\left(\f{\b}{ 2} \int_0^t \cos 2\Phi_*(\tau)\,d\tau\right)\right]
\eeqo
with a sufficiently small but universal $\b$, it holds
\beq
\pa_t \ln \Lam_*(t) \leq  F(\ln\Lam_*(t),t).
\label{eqn: differential inequality for ln Lam_*}
\eeq

We first show that
\beq
C_\dag\Lam_*(t) K(t) \cdot \sin \Phi_*
\leq \f34\big|\ln\Lam_*(t)\big|^2\cos 2\Phi_*.
\label{eqn: an intermediate inequality}
\eeq
Indeed, it is not difficult to verify that
\[
2\coth x
\geq 1 + x^{-1}\quad \forall\,x>0.
\]
Denote
\[
A(t): = \left(\int_0^t \cos 2\Phi_* (\tau)\,d\tau\right)^{-1}.
\]
Then by Proposition \ref{prop: bound for the max |Phi|} and \eqref{eqn: def of K},
\begin{align*}
C_\dag \Lam_*(t) K(t) \cdot \sin \Phi_*
\leq &\; C_\dag\exp\left[-2\coth\left(\f\b 2 A(t)^{-1}\right)\right] \left[1+C_* \exp\big(C_* A(t)\big)\right] \cdot Ce^{-t/\pi^2}\\
\leq &\; C \exp\left[- 1 - \f{2}{\b} A(t)\right] \exp\big(C_* A(t)\big) \cdot Ce^{-t/\pi^2}\\
\leq &\; \tilde{C}_*\exp\left[\big(C_*-2\b^{-1}\big) A(t)\right]\cdot e^{-t/\pi^2},
\end{align*}
where $C$ and $\tilde{C}_*$ are universal constants.
Assuming $2\b^{-1}\geq C_*$ yields
\[
C_\dag\Lam_*(t) K(t) \cdot \sin \Phi_*
\leq
\tilde{C}_* e^{-t/\pi^2}.
\]
On the other hand, since $\Phi_*(t)$ is decreasing in $t$,
\[
A(t)\geq \big[t\cos 2\Phi_*(t)\big]^{-1} \geq t^{-1},
\]
so
\[
\f34\big|\ln\Lam_*(t)\big|^2\cos 2\Phi_*
=
\f34\left[2\coth\left(\f{\b}{2}A(t)^{-1}\right)\right]^2
\cos 2\Phi_*
\geq \f34\left[1+\f{ 2}{\b}A(t)\right]^2\cos 2\Phi_*
\geq 3\b^{-2} t^{-2}.
\]
By choosing a smaller (but still universal) $\b$ if necessary, we can guarantee $\tilde{C}_* e^{-t/\pi^2}\leq 3\b^{-2} t^{-2}$ for all $t>0$.
Combining these estimates, \eqref{eqn: an intermediate inequality} is proved.
Without loss of generality, we assume $\b \leq 1/(2\pi^2)$.

Observe that $\Lam_*(t)$ is an increasing continuous function for $t>0$, satisfying that
\beqo
\pa_t \ln \Lam_*(t)
= \b \cos 2\Phi_* \left[
\f14 \big|\ln \Lam_*(t)\big|^2- 1\right],\quad \lim_{t\to 0^+} \ln \Lam_*(t) = -\infty.
\eeqo
By the assumption $\b \leq 1/(2\pi^2)$ and \eqref{eqn: an intermediate inequality},
\beqo
\begin{split}
\pa_t \ln \Lam_*(t)
%
%
\leq &\; \f{1}{2\pi^2}\cos 2\Phi_*
\left[\big|\ln \Lam_*(t)\big|^2- 1 - \f34\big|\ln \Lam_*(t)\big|^2\right]\\
\leq &\; \f{1}{2\pi^2}\cos 2\Phi_*
\left[ \big|\ln \Lam_*(t)\big|^2- 1 - C_\dag \Lam_*(t) K(t)\cdot  \f{\sin \Phi_*}{\cos 2\Phi_*}\right] = F(\ln\Lam_*(t),t).
\end{split}
\eeqo
\end{step}

\begin{step}
Now we prove $\Lam(t) \geq \Lam_*(t)$ on $(0,T]$ by following the standard justification of comparison principles for ordinary differential equations.

With abuse of notations, denote
\[
t_0 := \sup\{ \tau\in (0,T]:\, \Lam_*(t)< \Lam(t)\mbox{ for all } t\leq \tau\}.
\]
By the continuity of $\Lam_*$ and $\Lam$ (see the assumptions \ref{assumption: geometry}-\ref{assumption: non-degeneracy} on $X$), $t_0\in (0,T]$ is well-defined.
If $t_0 = T$, the desired claim is proved due to the time continuity at $T$.
Suppose $t_0<T$.
Then $\ln\Lam(t) > \ln \Lam_*(t)$ for any $t<t_0$ and $\ln\Lam(t_0) = \ln \Lam_*(t_0)<0$ by continuity.
In a neighborhood of $(\ln \Lam(t_0),t_0)$, by virtue of the continuity of $\Phi_*(t)$ (see Proposition \ref{prop: bound for the max |Phi|}), one can show that $F(x,t)$ is continuous in $(x,t)$, Lipschitz continuous in $x$, and decreasing in $x$ since $\ln \Lam(t_0)<0$.
We denote the Lipschitz constant to be $L$.
By \eqref{eqn: ODE for lower bound for |X'|}, for all $t<t_0$,
\[
\ln \Lam(t)
\leq
\ln \Lam(t_0) - \int_t^{t_0} F(\ln \Lam(\tau),\tau)\,d\tau.
\]
Combining this with \eqref{eqn: differential inequality for ln Lam_*} and the fact that $\ln \Lam(t_0) = \ln \Lam_* (t_0)$, we obtain that, for $t<t_0$ with $|t-t_0|\ll 1$,
\begin{align*}
\ln \Lam(t) - \ln \Lam_*(t)
\leq &\; -\int_{t}^{t_0}\big( F(\ln\Lam(\tau),\tau) - F(\ln\Lam_*(\tau),\tau)\big)\,d\tau\\
\leq &\; L \int_{t}^{t_0} \big(\ln\Lam(\tau) - \ln\Lam_*(\tau)\big)\,d\tau.
\end{align*}
Then by the Gr\"onwall's inequality, we must have
$\ln \Lam(t) - \ln \Lam_*(t) \leq  0$ for $t$ satisfying $0<t_0-t\ll1$, which is a contradiction.
Therefore, $\Lam(t) \geq \Lam_*(t)$ for all $t\in (0,T]$, and this gives \eqref{eqn: lower bound for |X'|}.

Lastly, given distinct $s,s'\in \BT$, assume that $s<s'<s+2\pi$ and $L(s,s')$ is the length of the arc $X([s,s'],t)$ (otherwise, consider $X([s',s+2\pi],t)$ instead).
By Proposition \ref{prop: geometric property of the curve},
\begin{align*}
|X(s,t)-X(s',t)|
\geq &\; CL(s,s')
= C\int_{s}^{s'} |X'(s'',\tau)|\,ds'' \\
\geq &\; (s'-s) \cdot C\min_s|X'(s,t)|
\geq |s-s'|_\BT \cdot C\min_s|X'(s,t)|.
\end{align*}
\end{step}
\end{proof}
\end{prop}

\subsection{Higher-order estimates}
\label{sec: higher-order estimates}
Denote
\beq
\pa_s \ln X'(s) = \f{X''(s)}{X'(s)} =: Z(s)+iW(s),
\label{eqn: introducing Z and W}
\eeq
where (see \eqref{eqn: a relation for curvature})
\beq
Z(s)= \f{\pa_s |X'(s)|}{|X'(s)|},\quad W(s)=\ka(s)|X'(s)|.
\label{eqn: def of Z and W}
\eeq
In this subsection, we shall first bound $Z$ in $L^2$, which is motivated by a special estimate in the tangential Peskin problem \cite[Lemma 5.3]{Tong2022GlobalST}.
This then allows us to bound $X''$.

We first derive the equation for $Z$.
\begin{lem}
$Z(s) = \pa_s |X'(s)|/|X'(s)|$ solves
\beq
\begin{split}
&\;\pa_t Z(s)\\
=
&\; \f{1}{4\pi} \pv\int_{\T} \left[\Re\, J(s,s') \big(Z(s') - Z(s)\big)
- \Im\, J(s,s') \big(W(s') - W(s)\big)
\right] ds'\\
&\; + \f{1}{2\pi}\int_{\T}
\frac{|X'(s')|^2\sin 2\Phi(s,s')}{|X(s')-X(s)|^2}
\left(3\Im\,I(s,s')-\kappa(s)|X'(s)|
\right) ds'.
\end{split}
\label{eqn: equation for Z prelim}
\eeq
\begin{proof}
By Lemma \ref{lem: some singular integrals along the string curve} and \eqref{eqn: eqn for |X'| complex form},
\[
\frac{\pa_t |X'(s)|}{|X'(s)|} = \mathcal{I}_1+\mathcal{I}_2,
\]
where
\begin{align*}
\mathcal{I}_1:=&\; \f{1}{4\pi}
\int_{\T} \Re\left[\f{X'(s')(X'(s')- X'(s))}{(X(s')-X(s))^2}  -\f{X'(s')X''(s)}{X'(s)(X(s')-X(s))}\right] ds'
 + \f{1}{4\pi}
\Re\left[ \f{X''(s)}{X'(s)}\pi i\right],\\
\mathcal{I}_2:=&\; \f{1}{4\pi}\int_{\T}I_2(s,s')\,ds',\qquad
I_2(s,s'):=\Im \left[\f{2X'(s')^2 X'(s) }{(X(s')-X(s))^3} \right] \Im\left[\frac{X(s')-X(s)}{X'(s)}\right].
\end{align*}
Since
\[
\pa_t{Z}(s)
=\pa_t\frac{\pa_s |X'(s)|}{|X'(s)|}=\pa_s\frac{\pa_t |X'(s)|}{|X'(s)|}
=\pa_s\mathcal{I}_1+\pa_s\mathcal{I}_2.
\]
we take the derivative of $\CI_1$,
\begin{align*}
\pa_s\mathcal{I}_1
= &\; \f{1}{4\pi}
\pv\int_{\T} \Re\left[\f{2X'(s')(X'(s')- X'(s))X'(s)}{(X(s')-X(s))^3}  -\f{2X'(s')X''(s)}{(X(s')-X(s))^2}\right]ds'\\
=&\; \f{1}{4\pi} \pv\int_{\T} \Re\left[\f{X''(s')X'(s)-X'(s')X''(s)}{(X(s')-X(s))^2}\right] ds'.
\end{align*}
In the second equality, we used
\begin{align*}
&\; \pa_{s'}\left[\f{(X'(s')- X'(s))X'(s)}{(X(s')-X(s))^2}-\f{X''(s)}{X(s')-X(s)} \right]\\
=&\; -\f{2X'(s')(X'(s')- X'(s))X'(s)}{(X(s')-X(s))^3}+\f{X''(s')X'(s)+X'(s')X''(s)}{(X(s')-X(s))^2},
\end{align*}
and thanks to the regularity assumptions on $X$,
\[
\lim_{s'\to s}\left[\f{(X'(s')- X'(s))X'(s)}{(X(s')-X(s))^2}-\f{X''(s)}{X(s')-X(s)} \right]=\f{X'''(s)X'(s)-X''(s)^2}{2X'(s)^2}.
\]
On the other hand, by \eqref{eqn: def of curvature s variable},
\begin{align*}
&\pa_s I_2\\
= &\; \Im \left[\f{6X'(s')^2 X'(s)^2 }{(X(s')-X(s))^4} \right] \Im\left[\frac{X(s')-X(s)}{X'(s)}\right]\\
&\; +\Im \left[\f{2X'(s')^2 X'(s) }{(X(s')-X(s))^3}\cdot \f{X''(s)}{X'(s)} \right] \Im\left[\frac{X(s')-X(s)}{X'(s)}\right]\\
&\; -\Im \left[\f{2X'(s')^2 X'(s) }{(X(s')-X(s))^3} \right] \Im\left[\f{X''(s) }{X'(s)}\cdot \frac{(X(s')-X(s))}{X'(s)} \right]\\
=&\; \Im \left[ 6J(s,s')^2\right] \Im\left[\frac{X(s')-X(s)}{X'(s)}\right]
-\kappa(s)|X'(s)| \Im \left[2J(s,s')^2\right] \frac{|X(s')-X(s)|^2}{|X'(s)|^2}.
\end{align*}
Here we used the following identity for arbitrary $A,B,C\in \mathbb{C}$,
\begin{align*}
&\;\Im\left[AB\right]\cdot \Im\, C - \Im \, A \cdot \Im\left[BC\right]\\
= &\; -\Im\, B\cdot \left[\Re\, A \cdot \Im\, \bar{C} + \Im \, A \cdot \Re \, \bar{C} \right]\\
= &\;-\Im\, B\cdot \Im\, \left[A/C\right] |C|^2.
\end{align*}
Combining the above calculations, we obtain that
\beqo
\begin{split}
\pa_t Z(s)
= &\; \f{1}{4\pi} \pv\int_{\T} \Re\left[J(s,s') \left(\f{X''(s')}{X'(s')} - \f{X''(s)}{X'(s)}\right)\right] ds'\\
&\; + \f{1}{4\pi}\int_{\T}\Im \left[2J(s,s')^2\right] \frac{|X(s')-X(s)|^2}{|X'(s)|^2}
\left(3\Im\,I(s,s')-\kappa(s)|X'(s)|
\right) ds'.
\end{split}
\eeqo
Then \eqref{eqn: equation for Z prelim} follows.
\end{proof}
\end{lem}

\begin{prop}
\label{prop: L2 estimate for Z}
Let $\ka_*(t)$ be defined in Proposition \ref{prop: max principle and decay estimate for curvature}.
Suppose $\Phi_*(0)<\f{\pi}{4}$.
For $t\geq 0$,
\beq
\begin{split}
&\; \|Z(\cdot,t)\|_{L^2(\BT)}^2
+ \f{1}{8\pi} \int_0^t \int_\BT \int_{\T} |J(s,s',\tau)| \cos \Phi(s,s',\tau)  \big|Z(s',\tau) - Z(s,\tau)\big|^2 \, ds'\,ds \,d\tau \\
\leq &\; \|Z(\cdot,0)\|_{L^2(\BT)}^2
+ C\max\big\{\kappa_*(0),R_X^{-1}\big\}^3R_X\mathcal{E}(0),
\end{split}
\label{eqn: L^2 estimate for Z}
\eeq
where $C>0$ is universal.

\begin{proof}

Observe that
\[
\pa_{s'}\left(\Im\,I(s,s')-\f13\kappa(s)|X'(s)|
\right) = \Im\, J(s,s'),
\]
and (see \eqref{eqn: limit of Im Y})
\[
\lim_{s'\to s}\left( \Im\,I(s,s')-\f13\kappa(s)|X'(s)| \right) = \f16 \ka(s)|X'(s)|,
\]
so
\begin{align*}
&\; \int_{\T}
\Im \, J(s,s')
\left(3\Im\,I(s,s')-\kappa(s)|X'(s)|
\right) ds'\\
= &\; \f32
\int_{\T}
\pa_{s'}\left[
\left(\Im\,I(s,s')-\f13 \kappa(s)|X'(s)|
\right)^2\right] ds'
=0.
\end{align*}
Hence, \eqref{eqn: equation for Z prelim} can be rewritten as
\begin{align*}
&\;\pa_t Z(s)\\
=
&\; \f{1}{4\pi} \pv\int_{\T} \left[\Re\, J(s,s') \big(Z(s') - Z(s)\big)
- \Im\, J(s,s') \big(W(s') - W(s)\big)
\right] ds'\\
&\; + \f{1}{2\pi}\int_{\T} \frac{|X'(s')||X'(s)|\sin 2\Phi(s,s')}{|X(s')-X(s)|^2}
\left(|X'(s')|-|X'(s)|\right)
\left(\f{3\Im\,I(s,s')}{|X'(s)|}-\kappa(s)
\right) ds'\\
&\; + \f{1}{2\pi}\int_{\T} \frac{|X'(s')||X'(s)|^2}{|X(s')-X(s)|^2}
\left[\sin 2\Phi(s,s')- 2\sin \Phi(s,s')\right]
\left(\f{3\Im\,I(s,s')}{|X'(s)|}-\kappa(s)
\right) ds'.
\end{align*}
Taking inner product with $Z$ yields
\beqo
\begin{split}
&\; \frac{d}{dt}\int_{\T}|{Z}(s)|^2\, ds
= 2\int_{\T}{Z}(s)\pa_t{Z}(s)\,ds \\
= &\;
\f{1}{2\pi} \int_\BT \pv\int_{\T} \left[\Re\, J(s,s') \big(Z(s') - Z(s)\big)Z(s)
- \Im\, J(s,s') \big(W(s') - W(s)\big)Z(s)
\right] ds'\,ds \\
&\; + \f{1}{\pi}\int_{\T} Z(s) \int_{\T}  \frac{|X'(s')||X'(s)|\sin 2\Phi(s,s')}{|X(s')-X(s)|^2}
\left(|X'(s')|-|X'(s)|\right)
\left(\f{3\Im\,I(s,s')}{|X'(s)|}-\kappa(s)
\right) ds'\,ds\\
&\; + \f{1}{\pi}\int_{\T} Z(s) \int_{\T} \frac{|X'(s')||X'(s)|^2}{|X(s')-X(s)|^2}
\left[\sin 2\Phi(s,s')- 2\sin \Phi(s,s')\right]
\left(\f{3\Im\,I(s,s')}{|X'(s)|}-\kappa(s)
\right) ds'\,ds \\
=: &\; \CZ_1 + \CZ_2 + \CZ_3.
\end{split}
\eeqo

For $\CZ_1$, we interchange the $s$- and the $s'$-variables and apply the Young's inequality to obtain
\beqo
\begin{split}
\CZ_1 = &\;
-\f{1}{4\pi} \int_\BT \int_{\T} \Re\, J(s,s') \big|Z(s') - Z(s)\big|^2 \, ds'\,ds \\
&\;
-\f{1}{2\pi} \int_\BT \int_{\T}
\Im\, J(s,s')   W(s) \big(Z(s')-Z(s)\big)
\, ds'\,ds\\
\leq &\;
-\f{1}{6\pi} \int_\BT \int_{\T} |J(s,s')| \cos \Phi(s,s')  \big|Z(s') - Z(s)\big|^2 \, ds'\,ds \\
&\;
+ C\ka_*^2 \int_\BT \int_{\T}
|J(s,s')| \sin^2 \Phi(s,s') |X'(s)|^2 \, ds'\,ds,
\end{split}
\eeqo
where $C$ is a universal constant given that $\Phi_* < \pi/4$.
Lemma \ref{lem: the curve is outside the osculating circle} implies
\[
\left|\f{3\Im\,I(s,s')}{|X'(s)|}-\kappa(s)
\right|\leq C\ka_*.
\]
So for $\CZ_2$, by the Cauchy-Schwarz inequality,
\begin{align*}
\CZ_2
\leq &\; C\ka_* \int_{\T} \int_{\T}  |Z(s)| \frac{|X'(s')||X'(s)| |\sin 2\Phi(s,s')|}{|X(s')-X(s)|^2}
\left||X'(s')|-|X'(s)|\right| ds'\,ds\\
\leq &\; C\ka_* \left(\int_{\T} \int_{\T}  |Z(s)|^2 |X'(s)|\cdot  \frac{|X'(s')| |\sin 2\Phi(s,s')|^2}{|X(s')-X(s)|^2}\, ds'\,ds\right)^{1/2}\\
&\; \cdot \left(\int_{\T} \int_{\T}  \frac{|X'(s')| |X'(s)|}{|X(s')-X(s)|^2}
\left||X'(s')|-|X'(s)|\right|^2 ds'\,ds\right)^{1/2}\\
\leq &\; C\ka_* \left(\int_\BT |Z(s)|^2 |X'(s)|\,ds \right)^{1/2}
\left\| \int_{\T} \f{|X'(s')||\sin 2\Phi(s,s')|^2}{|X(s')-X(s)|^2} \, ds'\right\|_{L_s^\infty(\BT)}^{1/2}
\\
&\; \cdot \left(\int_{\T} \int_{\T}  |J(s,s')|
\left||X'(s')|-|X'(s)|\right|^2 \cos \Phi(s,s')\, ds'\, ds\right)^{1/2}.
\end{align*}
We derive as in \eqref{eqn: positive term due to nonzero Phi} that
\[
\int_{\T} \f{|X'(s')||\sin 2\Phi(s,s')|^2}{|X(s')-X(s)|^2} \, ds'
\leq  C\ka_* \Phi_*.
\]
On the other hand, with $d_* = d_*(t)$ defined in Proposition \ref{prop: geometric property of the curve},
\beq
\begin{split}
&\; \int_\BT \int_{\T} |J(s,s')|\big|Z(s') - Z(s)\big|^2 \, ds'\,ds \\
\geq &\; \int_\BT \int_{\T} \f{|X'(s)||X'(s')|}{d_*(t)^2}\big|Z(s') - Z(s)\big|^2 \, ds'\,ds\\
= &\;
\int_\BT \int_{\T} \f{|X'(s)||X'(s')|}{d_*(t)^2}(|Z(s')|^2 +|Z(s)|^2) \, ds'\,ds
=
\frac{2 \mathcal{L}(t)}{d_*(t)^2}\int_{\T} {|X'(s)|}|Z(s)|^2\,ds.
\end{split}
\label{eqn: bound L2 norm of Z in the arclength coordinate by the H 1/2 norm}
\eeq
In the last line, we used
\[
\int_{\T}|X'(s)|Z(s)\,ds
=\int_{\T}\partial_s|X'(s)|\,ds=0,
\]
and the definition of $\CL(t)$.
Using \eqref{eqn: R_X d_* CL comparable} in Proposition \ref{prop: geometric property of the curve},
\[
\int_\BT |Z(s)|^2 |X'(s)|\,ds
\leq CR_X \int_\BT \int_{\T} |J(s,s')|\big|Z(s') - Z(s)\big|^2 \, ds'\,ds.
\]
Combining all these estimates, we obtain a bound for $\CZ_2$
\begin{align*}
\CZ_2
\leq &\; C\ka_*^{3/2}\Phi_*^{1/2}R_X^{1/2}\left( \int_\BT \int_{\T} |J(s,s')|\big|Z(s') - Z(s)\big|^2 \cos \Phi(s,s') \, ds'\,ds\right)^{1/2}
\\
&\; \cdot \left(\int_{\T} \int_{\T}  |J(s,s')|
\left||X'(s')|-|X'(s)|\right|^2 \cos \Phi(s,s')\, ds'\, ds\right)^{1/2},
\end{align*}
where $C$ is universal.

Similarly,
\beqo
\begin{split}
\CZ_3
\leq
&\;
C\ka_* \int_{\T} \int_{\T} |Z(s)| \frac{|X'(s')||X'(s)|^2}{|X(s')-X(s)|^2}
|\sin \Phi(s,s')|^3
\,ds'\,ds \\
\leq &\;
C\ka_* \Phi_* \left(\int_{\T} \int_{\T} |Z(s)|^2 |X'(s)| \frac{|X'(s')|}{|X(s')-X(s)|^2}
|\sin \Phi(s,s')|^2 \, ds'\,ds\right)^{1/2}\\
&\;\cdot
\left(\int_{\T} \int_{\T} \frac{|X'(s')||X'(s)|^3}{|X(s')-X(s)|^2}
\sin^2 \Phi(s,s')\,
ds'\,ds\right)^{1/2} \\
\leq &\;
C\ka_*^{3/2} \Phi_*^{3/2} R_X^{1/2}\left( \int_\BT \int_{\T} |J(s,s')|\big|Z(s') - Z(s)\big|^2 \cos \Phi(s,s') \, ds'\,ds\right)^{1/2}\\
&\;\cdot
\left(\int_{\T} \int_{\T} |J(s,s')|
\sin^2 \Phi(s,s')|X'(s)|^2\, ds'\,ds\right)^{1/2}.
\end{split}
\eeqo
Combining the estimates for $\CZ_j$ $(j = 1,2,3)$ and applying the Young's inequality, we obtain that
\beq
\begin{split}
&\; \frac{d}{dt}\int_{\T}|{Z}(s)|^2\,ds
+\f{1}{8\pi} \int_\BT \int_{\T} |J(s,s')| \cos \Phi(s,s')  \big|Z(s') - Z(s)\big|^2 \, ds'\,ds \\
\leq &\;
C\big(\ka_*^2 + \ka_*^{3} \Phi_*^{3} R_X\big) \int_\BT \int_{\T}
|J(s,s')| \sin^2 \Phi(s,s') |X'(s)|^2 \, ds'\,ds\\
&\; + C\ka_*^3 \Phi_* R_X
\int_{\T} \int_{\T} |J(s,s')|
\left||X'(s')|-|X'(s)|\right|^2 \cos \Phi(s,s')\, ds'\, ds.
\end{split}
\label{eqn: time derivative of the integral of Z squared}
\eeq

In view of the energy estimate \eqref{eqn: energy estimate}, we have that
\beq
\begin{split}
&\;\f{d}{dt}\int_{\BT} \f12 |X'(s)|^2\,ds\\
\leq &\;  -\f{1}{16\pi} \int_{\BT}
\int_{\T}| J(s,s')| \left(|X'(s)|-|X'(s')|\right)^2 \cos \Phi(s,s') \,ds'\,ds\\
&\; -\f{1}{16 \pi} \int_{\BT}
\int_{\T} | J(s,s')| \big(|X'(s)| + |X'(s')| \big)^2
\sin^2 \Phi(s,s')\, ds'\,ds\\
=: &\; -\f1{16\pi}\CD(t).
\end{split}
\label{eqn: energy estimate rewritten}
\eeq
Here we used, with $\Phi=\Phi(s,s')$, $|\Phi|\leq\pi/4$, $\cos \Phi\geq\cos 2\Phi\geq0$, and
\[
\cos \Phi-\cos 2\Phi 
\geq \cos^2 \Phi-\cos 2\Phi
= \sin^2 \Phi.
\]
On the other hand, 
Proposition \ref{prop: max principle and decay estimate for curvature} implies that $1\leq \ka_*(t)R_X\leq \max\{\ka_*(0)R_X, C_*\}$ with $C_* = 7+5\sqrt{2}$.
Applying these estimates to \eqref{eqn: time derivative of the integral of Z squared} yields
\beq
\begin{split}
&\; \frac{d}{dt}\|Z(\cdot,t)\|_{L^2}^2
+\f{1}{8\pi} \int_\BT \int_{\T} |J(s,s')| \cos \Phi(s,s')  \big|Z(s') - Z(s)\big|^2 \, ds'\,ds \\
\leq &\;
C_\dag \max\{\ka_*(0)R_X, C_*\}^3 R_X^{-2} \CD(t),
\end{split}
\label{eqn: time derivative of the integral of Z squared simplified}
\eeq
where $C_\dag>0$ is universal.
Adding \eqref{eqn: energy estimate rewritten} and \eqref{eqn: time derivative of the integral of Z squared simplified} with suitable coefficients, we obtain that
\beqo
\begin{split}
&\; \frac{d}{dt}\Big[
\|Z(\cdot,t)\|_{L^2}^2 + 16\pi C_\dag \max\{\ka_*(0)R_X, C_*\}^3 R_X^{-2} \CE(t)\Big]\\
&\;
+\f{1}{8\pi} \int_\BT \int_{\T} |J(s,s')| \cos \Phi(s,s')  \big|Z(s') - Z(s)\big|^2 \, ds'\,ds
\leq  0.
\end{split}
\eeqo
Then \eqref{eqn: L^2 estimate for Z} follows.
\end{proof}
\end{prop}

We can further derive a decay estimate for $\|Z\|_{L^2}$.
\begin{prop}
\label{prop: L2 decay for Z}
There exist universal constants $\g,C>0$, such that for all $t\geq 0$,
\[
\|Z(\cdot,t)\|_{L^2(\BT)}^2
\leq
C e^{-\g t} \left(\|Z(\cdot,0)\|_{L^2(\BT)}^2
+ \max\big\{\kappa_*(0)^3R_X ,\|X'(s,0)\|_{L^\infty} R_X^{-3} \big \} \mathcal{E}(0)\right).
\]

\begin{proof}
The case $t\in[0,1]$ follows from Proposition \ref{prop: L2 estimate for Z}, so we assume $t\geq 1$.
We first apply Proposition \ref{prop: bound for the max |Phi|} to see that, if $t\geq 1$,
\beqo
\int_0^t \cos 2\Phi_*(\tau)\,d\tau \geq Ct,
\eeqo
where $C>0$ is universal.
Hence, for $t\geq 1$, Proposition \ref{prop: max principle and decay estimate for curvature} gives that
\[
\kappa_*(t)R_X
\leq 1+C \exp\left[Ct^{-1}\right] \leq C,
\]
and Proposition \ref{prop: lower bound for |X'| and the well-stretched condition} implies that
\[
\min_s|X'(s,t)|
\geq R_X \exp\big[-2\coth(Ct)\big]
\geq C R_X,
\]
where $C>0$ is universal.

By Proposition \ref{prop: geometric property of the curve}
and \eqref{eqn: bound L2 norm of Z in the arclength coordinate by the H 1/2 norm}, for $t\geq 1$,
\begin{align*}
&\; \f{1}{8\pi} \int_\BT \int_{\T} |J(s,s')| \cos \Phi(s,s')  \big|Z(s') - Z(s)\big|^2 \, ds'\,ds \\
\geq &\; \f{\cos \Phi_*}{8\pi} \cdot \frac{2 \mathcal{L}(t)}{d_*(t)^2}\int_{\T} {|X'(s)|}|Z(s)|^2\,ds\\
\geq &\; CR_X^{-1} \min_{s}|X'(s,t)|\|Z(\cdot,t)\|_{L^2}^2
\geq C\|Z(\cdot,t)\|_{L^2}^2.
\end{align*}
On the other hand, arguing as in \eqref{eqn: positive term due to nonzero Phi} and using Lemma \ref{lem: energy estimate} and Lemma \ref{lem: max principle for |X'|},
\begin{align*}
\int_\BT \int_{\T}
|J(s,s')| \sin^2 \Phi(s,s') |X'(s)|^2 \, ds'\,ds
\leq&\; C\int_\BT |X'(s)|^3 \ka_* \sin \Phi_* \,ds \\
\leq &\; CR_X^{-1} \CE(0) \|X'(s,0)\|_{L^\infty}\Phi_*(t).
\end{align*}
Combining these estimates with \eqref{eqn: time derivative of the integral of Z squared} yields that, for $t\geq 1$,
\beqo
\frac{d}{dt}\|Z(\cdot,t)\|_{L^2}^2
+\g \|Z(\cdot,t)\|_{L^2}^2 
\leq
C\big(R_X^{-3} \CE(0) \|X'(s,0)\|_{L^\infty} + R_X^{-2} \CD(t)\big)\Phi_*(t) ,
\eeqo
where $\g,C>0$ are universal constants.

For simplicity, we assume $\g \leq (2\pi^2)^{-1}$, so that $e^{\g \tau} e^{-\tau/\pi^2}\leq e^{-\g \tau}$.
By the Gr\"{o}nwall's inequality and Proposition \ref{prop: bound for the max |Phi|}, for $t\geq 1$,
\begin{align*}
\|Z(\cdot,t)\|_{L^2}^2
\leq &\;
e^{-\g(t-1)} \|Z(\cdot,1)\|_{L^2}^2\\
&\; + C \int_1^t  e^{-\g(t-\tau)} \left(R_X ^{-3}\CE(0)
\| X'(s,0)\|_{L^\infty} + R_X ^{-2}\CD(\tau) \right)
\Phi_*(0)e^{-\tau/\pi^2}\,d\tau\\
\leq &\;
e^{-\g(t-1)} \|Z(\cdot,1)\|_{L^2}^2
+ C R_X^{-2} \int_1^t e^{-\g(t+\tau)} \CD(\tau)\,d\tau\\
&\; + C R_X^{-3}\CE(0) \|X'(s,0)\|_{L^\infty} \int_1^t  e^{-\g(t+\tau)} \,d\tau\\
\leq &\;
e^{-\g(t-1)} \|Z(\cdot,1)\|_{L^2}^2
+ C e^{-\g t} R_X^{-3} \CE(0)\| X'(s,0)\|_{L^\infty} .
\end{align*}
In the last inequality, we used $R_X\leq \f1{2\pi}\CL(0)\leq \|X'(s,0)\|_{L^\infty}$ due the isoperimetric inequality, and
\[
\int_1^t \CD(\eta)\,d\eta
\leq 16\pi\big(\CE(1) - \CE(t)\big)\leq C\CE(0)
\]
derived from \eqref{eqn: energy estimate rewritten}.
This together with \eqref{eqn: L^2 estimate for Z} implies the desired estimate.
\end{proof}
\end{prop}

Finally, we arrived at the uniform boundedness of $\|X''\|_{L^2}$.
\begin{cor}
\label{cor: uniform boundedness for X''}
Suppose $\Phi_*(0)<\pi/4$, $\CE(0)<+\infty$, $\ka_*(0)<+\infty$, and $Z(\cdot,0)\in L^2(\BT)$.
Then $\|X''(\cdot,t)\|_{L^2(\BT)}$ is uniformly bounded for $t\geq 0$, and the bound only depends on $R_X$, $\CE(0)$, $\ka_*(0)$, $\|X'(\cdot,0)\|_{L^{\infty}}$, and $\|Z(\cdot,0)\|_{L^2}$.
\begin{proof}
By \eqref{eqn: introducing Z and W} and \eqref{eqn: def of Z and W},
\[
|X''(s)|^2
= |X'(s)|^2\big(|Z(s)|^2+|X'(s)|^2|\kappa(s)|^2\big),
\]
so by Proposition \ref{prop: max principle and decay estimate for curvature}, Lemma \ref{lem: energy estimate} and Lemma \ref{lem: max principle for |X'|},
\[
\|X''\|_{L^2}^2
\leq \|X'\|_{L^{\infty}}^2
\big(\|Z\|_{L^2}^2+2\CE(t) \kappa_*(t)^2\big)
\leq C\|X'(\cdot,0)\|_{L^{\infty}}^2
\big(\|Z\|_{L^2}^2+ \CE(0) \kappa_*(0)^2\big).
\]
Then the claim follows from Proposition \ref{prop: L2 estimate for Z}.
\end{proof}
\end{cor}

\section{Proof of the Main Results}
\label{sec: proof of main results}
Now we are ready to prove Theorem \ref{thm: main thm}.
\begin{proof}[Proof of Theorem \ref{thm: main thm}]
Since $X_0(s)\in h^{1,\al}(\BT)$ and $X_0$ satisfies the well-stretched condition, say with constant $\lam>0$, by Theorem 1.3 of \cite{mori2019well}, there exists $T>0$ and a unique
\[
X(s,t)\in C([0,T];C^{1,\al}(\BT))\cap C^1([0,T];C^{\al}(\BT))
\]
such that $X(s,t)$ is a strong solution of \eqref{eqn: contour dynamic equation complex form} with the initial condition $X(s,0) = X_0(s)$.
It holds that
\begin{enumerate}[label = (\roman*)]
  \item $X(s,t)$ satisfies \eqref{eqn: contour dynamic equation complex form} in $\BT\times (0,T]$, and $X(\cdot,t)\to X_0(t)$ in $C^{1,\al}(\BT)$ as $t\to 0$;
  \item \label{property: well-stretched property of local solution}
  for all $t\in [0,T]$, $X(\cdot,t)$ satisfies the well-stretched condition with constant $\lam/2$ (see Section 3.3 of \cite{mori2019well});
\item \label{property: smoothness} by Theorem 1.4 of \cite{mori2019well}, for any $\d\in (0,T]$ and any $k\in \BN$, $X\in C^1([\d,T]; C^k(\BT))$.
\end{enumerate}

Take $t_0\in (0,T/2]$, $t_0\leq 1$, such that $\Phi_*(t)<\pi/4$ for all $t\in [0,t_0]$.
This is achievable thanks to the time continuity of $X$ in $C^{1,\al}(\BT)$ and $\Phi_*(0)<\pi/4$.
Indeed, the property \ref{property: well-stretched property of local solution} above implies that $\Phi(s_1,s_2,t)$ is continuous.
Clearly, $t_0$ may be chosen arbitrarily small.
By the properties \ref{property: well-stretched property of local       solution} and \ref{property: smoothness}, for any fixed $t\in [t_0,T]$, $s\mapsto X(s,t)$ is injective, $\min_s|X'(s,t)|\geq \lam/2$ and $X(\cdot,t)\in C^\infty(\BT)$.
In particular, $\|X'(\cdot,t_0)\|_{L^\infty}$, $\CE(t_0)$, $\ka_*(t_0)$, and $\|Z(\cdot,t_0)\|_{L^2}$ are all finite, and $X(s,t)\in C^1([t_0,T];C^k(\BT))$ for any $k\in \BN$.
Therefore, $X(s,t)$ on $\BT\times [t_0,T]$ satisfies our assumptions \ref{assumption: geometry}-\ref{assumption: non-degeneracy} in Section \ref{sec: preliminary}.
We treat $t_0$ as the new initial time, and find that:
\begin{enumerate}[label = (\alph*)]
  \item By Corollary \ref{cor: uniform boundedness for X''} and the Sobolev embedding $H^2(\BT)\hookrightarrow C^{1,1/2}(\BT)$,
      $\|X(\cdot,t)\|_{\dot{C}^{1,1/2}(\BT)}$ admits a uniform upper bound for $t\in [t_0,T]$.
      The bound only depends on $X(\cdot,t_0)$ but not on $T$.

  \item By Proposition \ref{prop: bound for the max |Phi|} and Proposition \ref{prop: lower bound for |X'| and the well-stretched condition}, there exists a constant $\tilde{\lam}>0$ only depending on $t_0$, such that for all $t\in [2t_0,T]$, $X(\cdot,t)$ satisfies the well-stretched condition with constant $\tilde{\lam}R_X$.
      In particular, $\tilde{\lam}$ does not depend on $X_0$ or $T$.
\end{enumerate}
Then by Theorem 1.3 and Theorem 1.8 of \cite{mori2019well}, the solution can be uniquely extended to $[0,+\infty)$.
For any $\d>0$ and any $k\in \BN$, $X\in C_{loc}^1([\d,+\infty); C^k(\BT))$.

In the sequel, we study long-time behavior of the global solution $X(s,t)$.
For convenience, we introduce the following terminology.
Suppose $Q(t)$ is a time-varying quantity defined in terms of the solution $X(\cdot ,t)$ for $t\in (0,+\infty)$.
By saying $Q(t)$ converges to $Q_*\in \BR$ exponentially (or decays exponentially in the case of $Q_* = 0$), we mean that for any $\d>0$, there exists a constant $C>0$ that may depend on $\d$ and $X_0$, and a universal constant $c>0$ that does not depend on $\d$ or $X_0$, such that $|Q(t)-Q_*|\leq C e^{-ct}$ for all $t\geq \d$.
If not otherwise stated, we will always adopt this convention in the rest of the proof.

By Proposition \ref{prop: bound for the max |Phi|} and the arbitrary smallness of $t_0$ above, $\Phi_*(t)< \pi/4$ for all $t\geq 0$, $\Phi_*(t)$ is continuous and non-increasing in $[0,+\infty)$, and $\Phi_*(t)$ decays exponentially, satisfying the claimed bound.
By Proposition \ref{prop: geometric property of the curve},
\[
\CL(t)= 2\sup_{s'}L(s,s')
\leq \f{\pi d_*}{1-\sin \Phi_*}
\leq \f{2\pi R_X }{1-\sin \Phi_*}\cdot \tan \left(\f{\pi}{4}+\f{\Phi_*}{2}\right).
\]
On the other hand, the isoperimetric inequality gives $\CL(t)\geq 2\pi R_X$.
Hence, by Proposition \ref{prop: bound for the max |Phi|}, as $t\to +\infty$, $\CL(t)$ converge to $2\pi R_X$ exponentially.

Proposition \ref{prop: max principle and decay estimate for curvature} implies that $\|\ka(s)R_X-1\|_{L^\infty}$ decays exponentially, and satisfies the desired bound in Theorem \ref{thm: main thm}.

By Proposition \ref{prop: lower bound for |X'| and the well-stretched condition}, for any $t>0$, and any distinct $s,s'\in \BT$,
\beqo
|X(s,t)-X(s',t)|\geq
|s-s'|_\BT \cdot R_X\cdot  C \exp\left[-2\coth\left(\f{\b}{ 2} \int_0^t \cos 2\Phi_*(\tau)\,d\tau\right)\right],
\eeqo
where $C$ and $\b$ are universal constants.
By Proposition \ref{prop: bound for the max |Phi|}, the function
\[
t\mapsto C \exp\left[-2\coth\left(\f{\b}{ 2} \int_0^t \cos 2\Phi_*(\tau)\,d\tau\right)\right]
\]
has a universal non-negative lower bound $\lam_\circ(t)$ that is strictly increasing on $[0,+\infty)$ such that $\lam_\circ(0) = 0$ and $\lam_\circ(t)>0$ for $t>0$.
This proves the desired claim for the well-stretched condition.

Lastly, we prove the exponential convergence to an equilibrium.
By the Cauchy-Schwarz inequality,
\beqo
\left(\int_\BT \big|\pa_s |X'(s)|\big|\,ds\right)^2 \leq \int_\BT \left(\f{\pa_s |X'(s)|}{|X'(s)|}\right)^2 ds
\int_\BT |X'(s)|^2 \,ds
= \|Z\|_{L^2}^2\|X'\|_{L^2}^2.
\eeqo
Using the smoothness of $X(\cdot,t)$ for $t>0$, Lemma \ref{lem: energy estimate}, 
and Proposition \ref{prop: L2 decay for Z}, we find 
\beqo
\int_\BT \big|\pa_s |X'(s)|\big|\,ds\mbox{ decays exponentially.}
\eeqo
Since
\begin{align*}
\big||X'(s)|-R_X\big|
\leq &\; \f{1}{2\pi} \left|\int_\BT (|X'(s')|-R_X)\,ds'\right| + \int_\BT \big|\pa_{s'}|X'(s')|\big|\,ds'\\
= &\; \f{1}{2\pi}|\CL(t)- 2\pi R_X| + \int_\BT \big|\pa_{s'}|X'(s')|\big|\,ds',
\end{align*}
$|X'(s)|$ converge uniformly to $R_X$ exponentially.
According to \eqref{eqn: introducing Z and W} and \eqref{eqn: def of Z and W},
\[
X''(s) -iX'(s) = \big[Z(s)+i\big(\ka(s)|X'(s)|-1\big)\big]X'(s).
\]
By Proposition \ref{prop: max principle and decay estimate for curvature}, Proposition \ref{prop: L2 decay for Z}, and the convergence of $|X'|$ shown above,
\beq
X''(s) -iX'(s) \mbox{ converges in $L^2(\BT)$ to $0$ exponentially.}
\label{eqn: X''-iX'}
\eeq

We write \eqref{eqn: contour dynamic equation} and \eqref{eqn: Stokeslet in 2-D} into the complex form
\beq
\begin{split}
\pa_t X(s)
= &\;
\f1{4\pi}
\int_\BT -\ln |X(s)-X(s')|X''(s')\,ds' \\
&\; + \f1{4\pi}
\int_\BT \f{X(s)-X(s')}{|X(s)-X(s')|^2} \Re \left[X''(s')\overline{(X(s)-X(s'))}\right] ds'.
\end{split}
\label{eqn: complex form of the contour dynamic equation original}
\eeq
This is equivalent to \eqref{eqn: contour dynamic equation complex form} given the smoothness of $X(s,t)$ and the well-stretched condition for $t\geq \d$ for any given $\d>0$.
Observe that, by integration by parts,
\beqo
\begin{split}
&\; \int_\BT \left(-\ln |X(s)-X(s')|iX'(s')
+ \f{X(s)-X(s')}{|X(s)-X(s')|^2} \Re \left[iX'(s')\overline{(X(s)-X(s'))}\right]\right) ds'\\
= &\; \int_\BT \left( i\pa_{s'}\ln |X(s)-X(s')|(X(s')-X(s))
+ \f{X(s')-X(s)}{|X(s)-X(s')|^2} \Re \left[iX'(s')\overline{(X(s')-X(s))} \right] \right) ds'\\
= &\; \int_\BT \f{X(s')-X(s)}{|X(s)-X(s')|^2}
\left(
i\Re\left[X'(s')\overline{(X(s')-X(s))}\right]
-\Im \left[X'(s')\overline{(X(s')-X(s))}\right] \right)ds'\\
= &\; i \int_\BT \f{X(s')-X(s)}{|X(s)-X(s')|^2}\cdot
 X'(s')\overline{(X(s')-X(s))} \, ds'\\
= &\; i \int_\BT  X'(s')\, ds' = 0.
\end{split}
\eeqo
Combining this with \eqref{eqn: complex form of the contour dynamic equation original} yields
\begin{align*}
\pa_t X(s)
= &\;
\f1{4\pi}
\int_\BT -\ln \left[\f{|X(s)-X(s')|}{R_X}\right]
\big(X''(s')-iX'(s')\big)\,ds' \\
&\; + \f1{4\pi}
\int_\BT \f{X(s)-X(s')}{|X(s)-X(s')|^2} \Re \left[\big(X''(s')-iX'(s')\big)\overline{(X(s)-X(s'))}\right] ds'.
\end{align*}
Here we inserted a factor of $1/R_X$ in the logarithm, which does not change the value of the integral.
Given arbitrary $\d>0$, thanks to the uniform well-stretched condition of $X(\cdot,t)$ for all $t\geq \d$,
\begin{align*}
|\pa_t X(s)|
\leq &\;
C
\int_\BT \Big(\big|\ln |s-s'|\big|+1\Big) \big|X''(s')-iX'(s')\big|\,ds' \\
\leq &\;
C
\Big\|\big|\ln |s-s'|\big|+1\Big\|_{L^2_{s'}} \big\|X''-iX'\big\|_{L^2}
\leq C\big\|X''-iX'\big\|_{L^2},
\end{align*}
where $C$ depends on $\d$ only.
Therefore, $\pa_t X$ uniformly converges to $0$ exponentially.
This further implies that there exists $X_\infty=X_\infty(s)\in L^\infty(\BT)$ such that
\beq
\mbox{$X(\cdot,t)$ converges uniformly to $X_\infty(s)$, with $\|X(\cdot,t)-X_\infty(\cdot)\|_{L^\infty}$ decaying exponentially.}
\label{eqn: uniform convergence in L^inf}
\eeq

By the uniform convergence, $X_\infty$ satisfies the well-stretched condition.
In view of the uniform bound for $\|X''\|_{L^2}$ for all $t\geq 1$, $X_\infty\in H^2(\BT)$ and $X(\cdot,t)$ converges to $X_\infty$ weakly in $H^2(\BT)$ and strongly in $C^1(\BT)$ as $t\to +\infty$.
By \eqref{eqn: X''-iX'} we have $X_\infty''-iX_\infty'=0 $.
As $|X'(s)|$ converges uniformly to $R_X$ exponentially, we have $|X'_\infty(s)|=R_X$.
Therefore, 
there exists $x_\infty\in \BC$ and $\xi_\infty\in \BT$, such that
\[
X_\infty(s) = x_\infty + R_X e^{i(s+\xi_\infty)}.
\]

Now we prove strong $H^2$-convergence to $X_\infty$.
Since $X_\infty''- iX_\infty'\equiv 0$,
\begin{align*}
&\; \|X(s,t)-X_\infty(s)\|_{\dot{H}^2}\\
\leq &\;
\big\|\big(X''(s,t)-iX'(s,t)\big)-\big(X''_\infty(s)-iX'_\infty(s)\big)\big\|_{L^2}
+\|X(s,t)-X_\infty(s)\|_{\dot{H}^1}\\
\leq &\; \|X''(s,t)-iX'(s,t)\|_{L^2}
+ C\|X(s,t)-X_\infty(s)\|_{L^2}^{1/2}
\|X(s,t)-X_\infty(s)\|_{\dot{H}^2}^{1/2}.
\end{align*}
By Young's inequality,
\[
\|X(s,t)-X_\infty(s)\|_{\dot{H}^2}
\leq C\|X''(s,t)-iX'(s,t)\|_{L^2}
+ C\|X(s,t)-X_\infty(s)\|_{L^2}.
\]
Then \eqref{eqn: X''-iX'} and \eqref{eqn: uniform convergence in L^inf} allow us to conclude that
\[
\mbox{$\|X(\cdot,t)-X_\infty(\cdot)\|_{H^2}$ decays exponentially.}
\]
\end{proof}


\end{document}